\documentclass[11pt, a4paper]{amsart}
\usepackage[top=15mm,bottom=15mm,left=15mm,right=15mm]{geometry}
\usepackage{amsfonts}
\usepackage[foot]{amsaddr}
\makeatletter
\renewcommand{\email}[2][]{%
\ifx\emails\@empty\relax\else{\g@addto@macro\emails{,\space}}\fi%
  \@ifnotempty{#1}{\g@addto@macro\emails{\textrm{(#1)}\space}}%
  \g@addto@macro\emails{#2}%
}
\makeatother

\usepackage[utf8]{inputenc}
\usepackage{enumitem}
\usepackage{csquotes}
\usepackage{graphicx,subcaption} 

\usepackage{soul}
\usepackage{amsmath, amsthm, amssymb, mathtools, amscd, graphicx, color}
\usepackage{accents}
\usepackage{bbm}
\usepackage{parskip}
\usepackage{textcomp}
\usepackage{comment}
\usepackage[colorlinks=true,urlcolor=blue,citecolor=blue,linkcolor=blue]{hyperref}

\newtheorem{theorem}{Theorem}[section]
\newtheorem{lemma}[theorem]{Lemma}
\newtheorem{proposition}[theorem]{Proposition}

\theoremstyle{definition}
\newtheorem{definition}[theorem]{Definition}

\newtheorem{claim}[theorem]{Claim}
\theoremstyle{remark}

\numberwithin{equation}{section}
\newcommand{\ubar}[1]{\underaccent{\bar}{#1}}
\newcommand{\absmod}[1]{\left|#1\right|}

\author{Javed Hazarika}
\author{Debashis Paul}
\address{Applied Statistics Division, Indian Statistical Institute, Kolkata}
\email{javarika@gmail.com; debpaul.isi@gmail.com}

\date{25/09/2024}
\title{Limiting Spectral Distribution of a Random Commutator Matrix}

\begin{document}
\begin{abstract}
We study the spectral properties of a class of random matrices of the form $S_n^{-} = n^{-1}(X_1 X_2^* - X_2 X_1^*)$ where
$X_k = \Sigma^{1/2}Z_k$, for $k=1,2$, 
$Z_k$'s are independent $p\times n$ complex-valued random matrices, and $\Sigma$ is a $p\times p$ positive semi-definite matrix, independent of 
the $Z_k$'s. We assume that $Z_k$'s have independent entries with zero mean and unit variance. The skew-symmetric/skew-Hermitian matrix $S_n^{-}$ will be referred to as a random commutator matrix associated with the samples $X_1$ and $X_2$. We show that, when the dimension $p$ and sample size $n$ increase simultaneously, so that $p/n \to c \in (0,\infty)$, there exists a limiting spectral distribution (LSD) for $S_n^{-}$, supported on the imaginary axis, under the assumptions that the spectral 
distribution of $\Sigma$ converges weakly and the entries of $Z_k$'s have moments of sufficiently high order. This nonrandom LSD can be 
described through its Stieltjes transform, which satisfies a coupled Mar\v{c}enko-Pastur-type functional equations. 
In the special case when 
$\Sigma = I_p$, we show that the LSD of $S_n^{-}$ is a mixture of a degenerate distribution at zero (with positive mass if $c > 2$), and a continuous distribution with a symmetric density function supported on a compact interval on the imaginary axis. Moreover, we show that the 
companion matrix $S_n^{+} = n^{-1}\Sigma_n^\frac{1}{2}(Z_1Z_2^* + Z_2Z_1^*)\Sigma_n^\frac{1}{2}$,  under identical assumptions, has an LSD supported on the real line, which can be similarly characterized.
\end{abstract}

\maketitle
\section{Introduction}
Since the seminal works on the behaviour of the 
empirical distribution of eigenvalues of large-dimensional symmetric matrices and sample 
covariance matrices by Wigner 
\cite{Wigner1958}  and  Mar\v{c}enko and Pastur \cite{MarcenkoPastur1967} respectively,
there have been extensive studies on establishing limiting behavior of
various classes of random matrices. With the traditional definitions of sample size and dimension for multivariate observations, one may refer to the high-dimensional asymptotic regime where these quantities are proportional as the random matrix regime. In the random matrix regime, there have been discoveries of nonrandom limits for the empirical distribution of sample eigenvalues of various classes of symmetric or hermitian matrices. Notable classes of examples include matrices known as Fisher matrices (or ``ratios'' of independent sample covariance matrices (\cite{CLTFMatrix1}, \cite{CLTFMatrix2}), signal-plus-noise matrices (\cite{SignalNoise}) arising in signal processing, sample covariance corresponding to data with separable population covariance structure (\cite{Lixin06},\cite{SepCovar}), with a given variance profile (\cite{varProfile}, symmetrized sample autocovariance matrices associated with stationary linear processes (\cite{Staionary1}, \cite{Staionary2}, \cite{Staionary3}), sample cross covariance matrix (\cite{CrossCov}), etc. Studies of the spectra of these classes of random matrices mentioned above are often motivated by various statistical inference problems. 
In this paper, we study the asymptotic behavior of the spectra of a class of random matrices that we refer to as ``random commutator matrices'' under the random matrix regime, and discuss a potential application to
a statistical inference problem involving covariance matrices. 

As the setup for introducing these random matrices, suppose we have $p$-variate independent samples of the same size $n$ 
(expressed as $p\times n$ matrices) denoted by
$X_j = [X_{j,1}:\cdots:X_{j,n}]$, for $j = 1,2$,
from two populations with zero mean and variance $\Sigma$. We shall study the spectral 
properties of the matrix $S_n^{-}$ defined as $S_n^{-} = n^{-1}(X_1 X_2^* - X_2 X_1^*)$, where $X_i^*$ denotes the Hermitian conjugate of $X_i$.
Given the analogy with a \textit{commutator matrix}, we shall refer to $S_n^{-}$ as a 
``sample commutator matrix'' associated with the data $(X_1,X_2)$.
A distinctive feature of $S_n^{-}$ is that it is skew-symmetric, so that the eigenvalues of  $S_n^{-}$ are purely imaginary numbers.

As a primary contribution, in this paper we establish the existence of limits
for the empirical spectral distribution (ESD) of $S_n^{-}$, when $p,n \to \infty$ such that $p/n \to c \in (0,\infty)$, and describe the limiting spectral distribution (LSD) through its Stieltjes transform, under additional technical assumptions
on the statistical model. This LSD 
can be derived as a unique solution of a 
pair of functional equations describing its Stieltjes transform. The proof techniques are largely based on the matrix decomposition based approach popularized by \cite{BaiSilv09}. 
Furthermore, 
in the special case when $\Sigma = I_p$, we completely describe 
the LSD of $S_n^{-}$ as a mixture distribution on the imaginary axis with a point mass at zero (only if $c > 2$),
and a symmetric distribution with a density. Establishment of 
this result requires a very careful analysis of the Stieltjes transform
of the LSD of $S_n^{-}$, since the latter satisfies a cubic equation 
for each complex argument. Somewhat remarkably, the density function 
of the continuous component of the LSD can be derived in a closed (albeit
complicated) functional form that depends only on the value of $c$.

As a further contribution, we also study the asymptotic
behaviour of the spectrum of the companion matrix $S_n^{+}$, which is 
symmetric/Hermitian, when $\Sigma_1 = \Sigma_2$. We also characterize the 
limiting spectral distribution when the common covariance is the identity matrix. 
The results follow a similar pattern, which is why 
we state these results in parallel with our main results (about 
the spectral distribution of $S_n^{-}$).

We provide brief statistical motivations for studying the spectrum of $S_n^{-}$. As a first interesting observation, suppose that we have a complex random matrix $X = X_1 + \mathbbm{i}X_2$ where $X_1,X_2$ are real matrices. Then $S_n^-$ is the imaginary part of the $n^{-1}XX^*$ which is the sample covariance matrix associated with the data given by $X$. 

As another potential application, consider a set of $n$ \textit{paired} $p$-dimensional observations from a multivariate Gaussian distribution. Denote the two samples as $X_1 =\Sigma^\frac{1}{2}W$ and $X_2 = \Sigma^\frac{1}{2}Z$ where    $W = (W_{ij}), Z = (Z_{ij}) \in \mathbb{R}^{p \times n}$ having independent entries with zero mean and unit variance. The experimenter suspects an element-wise dependence, i.e., $\operatorname{Corr}(Z_{ij}, W_{ij}) = \rho$, and would like to test the hypothesis $H_0: \rho = 0$ against $H_1: \rho \neq 0$.

We can characterize this dependence in terms of another independent Gaussian random matrix $V = (V_{ij})$, with i.i.d. standard normal entries, as follows. Observe that, distributionally, we have the following representation:
$$
W_{ij} = \rho Z_{ij} + \sqrt{1 - \rho^2}V_{ij}, \qquad \mbox{for}~i=1,\ldots,p,~j=1,\ldots,n.
$$
Denoting $[\cdot,\cdot]$ to be the associated commutator operator, we see that 
\begin{align}\label{test1}
[X_1, X_2] = \Sigma^\frac{1}{2}[Z, W]\Sigma^\frac{1}{2} =& \Sigma^\frac{1}{2}(ZW^* - WZ^*)\Sigma^\frac{1}{2}\\
       =& \Sigma^\frac{1}{2}\bigg(Z(\rho Z^* + \sqrt{1-\rho^2})V^*) - (\rho Z + \sqrt{1-\rho^2}V)Z^* \bigg)\Sigma^\frac{1}{2}\notag\\
       =& \sqrt{1-\rho^2}\Sigma^\frac{1}{2}(ZV^*-VZ^*)\Sigma^\frac{1}{2} = \sqrt{1-\rho^2}\Sigma^\frac{1}{2}[Z, V]\Sigma^\frac{1}{2} \notag    
\end{align}
Note that under the null hypothesis, $Z$ and $W$ are independent, thus allowing us to derive the limiting spectral distribution of $[X_1, X_2]$ using the results derived in this paper. Even under the alternative, (\ref{test1}) allows us to derive the limiting spectral distribution of $[X_1, X_2]$, by making used of the fact that $Z$ and $V$ are independent. Indeed, under the alternative, the only change in the form of the LSD is that the support of the LSD shrinks by a factor of $\sqrt{1-\rho^2}$. This result can be helpful in deriving asymptotic properties of test statistics for testing $H_0:\rho=0$ vs. $H_1:\rho \neq 0$ if such statistics are derived 
from linear functionals of the eigenavlues of 
$[X_1, X_2]$.

The rest of the paper is organized as follows. In Section \ref{sec:model},
we describe the basic data model and introduce the key objects.
Section \ref{sec:Measures} contains results relevant to measures on the imaginary axis. The main result (Theorem \ref{t3.1}) on the existence of LSD under a general common covariance for the two populations is presented in Section \ref{sec:general_covariance}. 
Section \ref{sec:identity_covariance} is focused on giving a detailed description of the LSD when the common covariance is the identity matrix. It also includes numerical validations of the theoretical distribution. Section \ref{sec:Hermitian} briefly 
describes the few analogous results for the Hermitian case, 
i.e., corresponding to the matrix $S_n^{+}$. 
Appendix-A carries a few general purpose results related to matrices and convergence of random variables. Appendix-B contains results and proofs of theorems stated in Section \ref{sec:general_covariance}. Results and proofs of Theorems stated in Section  \ref{sec:identity_covariance} are presented in Appendix-C. Appendix-D contains the proofs of the results in Section \ref{sec:Hermitian}.

\section{Model and preliminaries}\label{sec:model}

Suppose $\{Z_{1}^{(n)}\}_{n=1}^{\infty}, \{Z_{2}^{(n)}\}_{n=1}^{\infty}$ are sequences of random matrices each having dimension $p \times n$ such that $p/n \rightarrow c \in (0, \infty)$. The entries have zero mean, unit variance and uniformly bounded moments of order $4 + \eta_0$ for some $\eta_0 > 0$. Let $\Sigma_{n} \in \mathbb{C}^{p \times p}$ be a sequence of random positive definite matrices such that the empirical spectral distributions (ESD) of $\{\Sigma_{n}\}_{n=1}^{\infty}$ converge weakly to a probability distribution function $H$ in an almost sure sense. We are interested in the limiting behaviour (as $p, n \rightarrow \infty$) of the ESDs of matrices of the type 
\begin{align*}
&S_n^\pm := \frac{1}{n}X_{1}^{(n)} (X_{2}^{(n)})^{*} \pm \frac{1}{n}X_{1}^{(n)}(X_{2}^{(n)})^{*} \text{, where } \label{defining_Sn}\tag{1.1} \\
&X_{k}^{(n)} := (\Sigma_{n})^{\frac{1}{2}}Z_{k}^{(n)} \text{ for } k \in \{1, 2\}\label{defining_Xk}\tag{1.2}
\end{align*}
Henceforth, for simplicity we will use $Z_k$ (corresp., $X_k$) to denote $Z_k^{(n)}$ (corresp., $X_k^{(n)}$), respectively for $k=1,2$. We use the method of Stieltjes Transforms to arrive at the non-random LSD of such matrices. The main results of this paper are mentioned in Theorem \ref{t3.1} and Theorem \ref{DensityDerivation}.

Note that $S_n^+$ is Hermitian and $S_n^-$ is skew-Hermitian. As such their eigenvalues are completely supported on the real and imaginary axes respectively. An interesting thing we discovered through our analysis is that the results for $S_n^+$ (corresp., $S_n^-$) bear striking resemblance to each other. The proof techniques adopted for both cases are also very similar. To avoid repetition and because of our belief that the results from $S_n^-$ might find more practical applications, we will focus on the results of the skew-Hermitian version.

\section{Measures on the imaginary axis}\label{sec:Measures}

The existing definition of Stieltjes transform and basic results involving the weak convergence of probability measures are adequate when we consider measures defined over (subsets of) the real line. However, we are dealing with skew-Hermitian matrices which have purely imaginary eigenvalues. In this section, we modify/ develop existing results to derive some analogous results that fit our case.

Let $\mathbbm{i} = \sqrt{-1}$ and $\mathbb{C}_L := \{-u + \mathbbm{i}v : u > 0, v \in \mathbb{R}\}, \mathbb{C}_R := \{u + \mathbbm{i}v : u > 0, v \in \mathbb{R}\}$ denote the left and right halves of the complex plane respectively (excluding the imaginary axis). For a complex number $z$, we use $\Re(z)$ and $\Im(z)$ to denote its real and imaginary parts respectively. 
\begin{definition}
\textbf{Stieltjes Transform:} For a probability measure $\mu$ on the imaginary axis, define\\
$s_{\mu}: \mathbb{C} \backslash \operatorname{supp}(\mu) \rightarrow \mathbb{C}$, \hspace{3mm} $\displaystyle s_{\mu}(z) = \int_{\mathbb{R}} \dfrac{\mu(dt)}{\mathbbm{i}t - z}$    
\end{definition}

With this definition, we immediately observe the following properties.
\begin{description}
    \item[1] $s_\mu(.)$ is analytic on its domain and $s_{\mu}(\mathbb{C}_L) \subset \mathbb{C}_R$ and $s_{\mu}(\mathbb{C}_R) \subset \mathbb{C}_L$ 
    \item[2] $|s_{\mu}(z)| \leq 1/|\Re(z)|$
    \item[3] If $\mu$ admits a density at $\mathbbm{i}x$ where $x \in \mathbb{R}$, then 
    \begin{align}\label{inversion_density}
    f_{\mu}(x) = \dfrac{1}{\pi} \underset{\epsilon \downarrow 0}{\lim} \Re  (s_{\mu}(-\epsilon + \mathbbm{i} x))
    \end{align}
    \item[4] If $\mu$ admits a point mass at $\mathbbm{i}x$ where $x \in \mathbb{R}$, then 
    \begin{align}\label{inversion_pointMass}
    \mu(\{x\}) = \underset{\epsilon \downarrow 0}{\lim} \epsilon s_{\mu}(-\epsilon + \mathbbm{i} x)        
    \end{align}
    \item[5] For $\mathbbm{i}a,\mathbbm{i}b$ continuity points of $\mu$, we have
    \begin{align}\label{measureOfInterval}
        \mu([\mathbbm{i}a, \mathbbm{i}b]) = \dfrac{1}{\pi}\underset{\epsilon \downarrow 0}{\lim}\int_a^b\Re(s_\mu(-\epsilon + \mathbbm{i}x)dx
    \end{align}
\end{description}

\begin{definition}
    \textbf{Distribution Function over the imaginary axis:} For a skew-Hermitian matrix $S \in \mathbb{C}^{p \times p}$ with eigenvalues $\{\mathbbm{i}\lambda_j\}_{j=1}^p$, we define the empirical spectral distribution of $S$ as
    \begin{align}\label{ESD_of_skHerm}
F^S:\mathbbm{i}\mathbb{R} \rightarrow [0, 1]; \hspace{3mm} F^S(\mathbbm{i}x) = \frac{1}{p}\sum_{j=1}^p \mathbbm{1}_{\{\lambda_j \leq x\}}        
    \end{align}
\end{definition}
Note that $-\mathbbm{i}S$ is Hermitian with eigenvalues $\{\lambda_j\}_{j=1}^p$. Reconciling (\ref{ESD_of_skHerm}) with the established definition of ESD for Hermitian matrices (as per Section 2 of \cite{BaiSilv95}), we thus have 
\begin{align}\label{ESD_sksym_equal_sym}
F^S(\mathbbm{i}x) = F^{-\mathbbm{i}S}(x)  \hspace{2mm} \forall x \in \mathbb{R}   
\end{align}
We will be employing this strategy of looking at the real counterparts of distribution functions on the imaginary axis throughout the paper. In particular, if $F$ denotes the distribution function of a purely imaginary random variable $X$, then denoting $\Tilde{F}$ as the distribution function of $-\mathbbm{i}X$ we have 
\begin{align}\label{CDF_imag_equal_real}
F(\mathbbm{i}x) = \Tilde{F}(x)  \text{ for } x \in \mathbb{R}  
\end{align}
This allows us to define the analogous Levy metric between distribution functions $F, G$ on the imaginary axis as 
$L_{im}(F, G) := L(\tilde{F}, \tilde{G})$ where $L(\cdot, \cdot)$ is the ``standard" Levy metric between distributions over the real line. Similarly, we define the uniform metric between $F$ and $G$ as $||F-G||_{im} := ||\tilde{F} -\tilde{G}||$ where $||\cdot||$ represents the ``standard" uniform metric between distributions over the real line. Therefore, using Lemma B.18 of \cite{BaiSilv09} leads to the following analogous inequality between Levy and uniform metrics.
\begin{align}\label{Levy_vs_uniform}
    L_{im}(F, G) = L(\tilde{F}, \tilde{G}) \leq ||\tilde{F} -\tilde{G}|| = ||F-G||_{im}    
\end{align}

In a similar vein, the weak convergence of a sequence of probability distributions (${F_n}_{n=1}^\infty$) over the imaginary axis is equivalent to the weak convergence of their real counterparts to an appropriate probability distribution over the real line. The following is an analogue of a celebrated result linking convergence of Stieltjes transforms of measures to the weak convergence of measures on the real axis.
\begin{proposition}\label{GeroHill}
    For a skew-Hermitian matrix $S_n$, let $s_n$ be the Stieltjes transform of $F^{S_n}$. If $s_n(z) \xrightarrow{a.s.} s_F(z)$ for $z \in \mathbb{C}_L$ and $\underset{y \rightarrow +\infty}{\lim}ys_F(-y) = 1$, then $F^{S_n} \xrightarrow{d} F$ a.s. where $s_F$ is the Stieltjes transform of $F$. 
\end{proposition}
\begin{proof}
The proof can be adapted with similar arguments from Theorem 1 of \cite{GeroHill03} which is stated below.

``Suppose that ($P_n$) are real Borel probability measures with Stieltjes transforms ($S_n$) respectively. If $\underset{n\rightarrow\infty}{\lim}S_n(z)=S(z)$ for all z with $\Im(z) > 0$, then there exists a Borel probability measure $P$ with Stieltjes transform $S_P=S$ if and only if 
$$\underset{y\rightarrow\infty}{\lim}\mathbbm{i}yS(\mathbbm{i}y)=-1$$ in which case $P_n\rightarrow P$ in distribution."
\end{proof}

The next proposition states a sufficient condition for the existence of density of a measure over the imaginary axis at a point in its support.

\begin{proposition}\label{SilvChoiResult1}
    Let $m_G(.)$ is the Stieltjes Transform of G, a probability measure $G$ on the imaginary axis. Then G is differentiable at $\mathbbm{i}x_0$, if $m^*(\mathbbm{i}x_0) \equiv \underset{z \in \mathbb{C}_L \rightarrow \mathbbm{i}x_0}{\lim}\Re(m_G(z))$ exists where  and its derivative at $\mathbbm{i}x_0$ is $({1}/{\pi})m^*(\mathbbm{i}x_0)$.
\end{proposition}
\begin{proof}
The proof is similar to that of Theorem 2.1 of \cite{SilvChoi} which is stated below.

``Let G be a p.d.f. and $x_0 \in \mathbb{R}$. Suppose $\Im(m_G(x_0)) \equiv \underset{z \in \mathbb{C}^+ \rightarrow x_0}{\lim}\Im(m_G(z))$ exists. Then G is differentiable at $x_0$, and its derivative is $({1}/{\pi})\Im(m_G(x_0))$."    
\end{proof}

\section{LSD under an arbitrary common scaling matrix}\label{sec:general_covariance}

In this section, we address the problem assuming that the columns of $X_k= [X_{k1}:\ldots:X_{kn}]$ for $k \in \{1,2\}$ are independent samples from probability distributions with a common variance structure (say $\Sigma$). The main result of this section is stated below. We will use $S_n$ instead of $S_n^-$ (\ref{defining_Sn}) for simplicity.

\begin{theorem}\label{t3.1}
Suppose the following hold.
\begin{description}
    \item[T1] $c_n := p/n \rightarrow c \in (0, \infty)$
    \item[T2] $\{Z_{1}^{(n)}\}_{n=1}^{\infty}, \{Z_{2}^{(n)}\}_{n=1}^{\infty} \in \mathbb{C}^{p \times n}$ are random matrices each of dimension $p \times n$ with independent entries having zero mean, unit variance and and uniform bound exists on moments of order $4 + \eta_0$ for some $\eta_0 > 0$
    \item[T3] $\Sigma_{n} \in \mathbb{C}^{p \times p}$ is a sequence of p.s.d. matrices independent of $Z_{1n}, Z_{2n}$ such that the ESDs of $\{\Sigma_{n}\}_{n=1}^{\infty}$ converges weakly to the probability distribution $H \neq \delta_0$ almost surely, i.e. $F^{\Sigma_n} \xrightarrow{d} H$ a.s.
    \item[T4] Further, $\exists$ $C > 0$ such that $\underset{n \rightarrow \infty}{\limsup} \frac{1}{p}\operatorname{trace}(\Sigma_n) < C$.
\end{description}
Then, $F^{S_n} \xrightarrow{d} F$ a.s. where F is a non-random distribution with Stieltjes Transform at $z \in \mathbb{C}_L$ given by
\begin{equation*}\label{s_main_eqn}\tag{4.1}
s(z) = \dfrac{1}{z}\bigg(\dfrac{2}{c}-1\bigg) + \dfrac{1}{\mathbbm{i}cz}\bigg(\dfrac{1}{\mathbbm{i} +ch(z)} - \dfrac{1}{-\mathbbm{i} +ch(z)}\bigg)    
\end{equation*}
where $h(z) \in \mathbb{C}_R$ is the unique number such that 
\begin{equation*}\label{h_main_eqn}\tag{4.2}
h(z) = \int \dfrac{\lambda dH(\lambda)}{\bigg[-z + \lambda\bigg(\dfrac{1}{\mathbbm{i} +ch(z)} + \dfrac{1}{-\mathbbm{i} + ch(z)}\bigg)\bigg]}    
\end{equation*}
Further $h$ is analytic in $\mathbb{C}_L$ and has a continuous dependence on $H$.
\end{theorem}

\textbf{REMARK}: For a justification of $H \neq \delta_0$ in Theorem \ref{t3.1}, please refer to Lemma \ref{tightness} which also serves as a proof for tightness of the sequence $\{F^{S_n}\}_{n=1}^\infty$.   

\subsubsection{Notations}\label{notations}
The following expressions will be used frequently.
\begin{itemize}
    \item $X_k = [X_{k1}: \ldots: X_{kn}]$ and $Z_k = [Z_{k1}: \ldots: Z_{kn}]$ for $k \in \{1, 2\}$
    \item $z_{ij}^{(k)}$ refers to the $ij^{th}$ element of $Z_k$ for $k \in \{1,2\}$
    \item $S_n = \frac{1}{n}(X_1X_2^* - X_2X_1^*) = \frac{1}{n}\sum_{r=1}^n (X_{1r}X_{2r}^* - X_{2r}X_{1r}^*)$
    \item $S_{nj} = \frac{1}{n}\sum_{r\neq j} (X_{1r}X_{2r}^* - X_{2r}X_{1r}^*)$ for $j \in \{1, \ldots, n\}$
    \item For $z \in \mathbb{C}_L$, $Q(z) = (S_n - zI_p)^{-1}$ and $Q_{-j}(z) = (S_{nj} - zI_p)^{-1}$
\end{itemize}

Since we will be using the method of Stieltjes Transform to prove the weak convergence, we define the central objects of our work below.
\begin{definition}\label{defining_sn}
Let For $z \in \mathbb{C}_L$, let $s_n(z) := \frac{1}{p}\operatorname{trace}\{Q(z)\}$ be the Stieltjes Transform of $F^{S_n}$. 
\end{definition}
\begin{definition}\label{defining_hn}
For $z \in \mathbb{C}_L$, $h_n(z) := \frac{1}{p}\operatorname{trace}\{\Sigma_nQ(z)\}$    
\end{definition}

Both $s_n(z), h_n(z)$ are analytic functions on $\mathbb{C}_L$ as they are both Stieltjes Transforms of certain measures. To see this, let $S_n = P\Lambda P^*$ where $\Lambda = \operatorname{diag}(\{\mathbbm{i}\lambda_j\}_{j=1}^p)$ and $P$ is an appropriate unitary matrix.  Then 
$$h_n(z) = \frac{1}{p}\operatorname{trace}\{\Sigma_n Q(z)\} = \frac{1}{p}\sum_{j=1}^p \frac{a_{jj}}{\mathbbm{i}\lambda_j - z} \text{ where } a_{jj} = (P^*\Sigma_nP)_{jj} \geq 0$$ is the Stieltjes Transform of a measure on the imaginary axis assigning a mass of $a_{jj}/p$ to the point $\mathbbm{i}\lambda_j$. For simplicity, $h_n(z)$ and $s_n(z)$ will be referred to as $h_n$ and $s_n$ respectively with $z$ being an arbitrary fixed point in $\mathbb{C}_L$ unless explicitly specified. 

\textbf{REMARK:} The assumptions on $\Sigma_n$ hold in an almost sure sense. Moreover, $F^{\Sigma_n}$ converges weakly to a non-random $H$ almost surely. By the end of this Section, we show that conditioning on $\Sigma_n$, $F^{S_n}$ converges weakly to a non-random limit $F$ that depends on $\Sigma_n$ only through their limit $H$ which is non-random. This result holds irrespective of whether $\Sigma_n$ is random or not. Therefore, we will henceforth treat $\{\Sigma_n\}_{n=1}^\infty$ as a non-random sequence.

\subsection{Sketch of the proof}\label{ProofSketch}

The theorem will be proved in the following steps. For arbitrary $z \in \mathbb{C}_L$,
\begin{enumerate}
    \item[1] Show that there can be at most one  solution to (\ref{h_main_eqn}) in $\mathbb{C}_R$. 
    \item[2] To prove existence, a solution to (\ref{h_main_eqn}) is furnished by showing that $h_n(z)$ converges almost surely to a quantity ($h^\infty(z) \in \mathbb{C}_R$) that uniquely satisfies the equation. This is proved by showing that any subsequential limit of $h_n$ satisfies equation (\ref{h_main_eqn}) which by uniqueness implies all subsequential limits are the same.
    \item[3] This unique solution ($h^\infty(z)$) when plugged into (\ref{s_main_eqn}) gives a function ($s^\infty(z)$) which is shown to be a Stieltjes transform (of a measure on the imaginary axis). Let $F$ be the probability distribution characterised by $s^\infty(z)$.
    \item[4] Finally we will show that $s_n(z) \xrightarrow{a.s.} s^\infty(z)$ and thus $F^{S_n} \xrightarrow{d} F$ a.s. from Proposition \ref{GeroHill}.
\end{enumerate}

\begin{definition}
Define the complex-valued functions $\rho, \rho_2$ as
\begin{align*}
\rho(z) :=& \frac{1}{\mathbbm{i} + z} + \frac{1}{-\mathbbm{i} + z}, z \not \in \{\mathbbm{i}, -\mathbbm{i}\} \label{definingRho}\tag{4.3}\\
\rho_2(z) :=& \frac{1}{|\mathbbm{i} + z|^2} + \frac{1}{|-\mathbbm{i} + z|^2}, z \not \in \{\mathbbm{i}, -\mathbbm{i}\} \label{definingRho2}\tag{4.4}
\end{align*}
Then for $z \not \in \{\mathbbm{i}, -\mathbbm{i}\}$ we have,
\begin{align}\label{realOfRho}\tag{4.5}
\Re(\rho(z)) =& \frac{\Re(\overline{\mathbbm{i}+z})}{|\mathbbm{i}+z|^2} + \frac{\Re(\overline{-\mathbbm{i}+z})}{|-\mathbbm{i}+z|^2} = \rho_2(z)\Re(z)
\end{align}    
\end{definition}
REMARK: Note that $\rho(\overline{z}) = \overline{\rho(z)}$ and $\rho$ is analytic in any open set not containing $\pm\mathbbm{i}$. Also $\rho_2(z) > 0$ in its domain. Now we prove the unique solvability of (\ref{h_main_eqn}).

\subsection{Proof of Uniqueness}

\begin{theorem}\label{Uniqueness}
There exists at most one solution to the following equation within the class of functions that map $\mathbb{C}_L$ to $\mathbb{C}_R$.
    $$h(z) = \displaystyle \int \dfrac{\lambda dH(\lambda)}{-z + \dfrac{\lambda}{\mathbbm{i} +ch(z)} + \dfrac{\lambda}{-\mathbbm{i} +ch(z)}} = \int \dfrac{\lambda dH(\lambda)}{-z + \lambda \rho(ch(z))}$$
where H is any probability distribution function such that $\operatorname{supp}(H) \subset \mathbb{R}_{+}$ and $H \neq \delta_0$.\end{theorem}

\begin{proof}
    Suppose for some $z = -u + \mathbbm{i}v \in \mathbb{C}_L$, $\exists$ $h_1, h_2 \in \mathbb{C}_R$ such that for $j \in \{1,2\}$, we have $$h_j = \displaystyle \int \dfrac{\lambda dH(\lambda)}{-z + \lambda \rho(ch_j)}$$
    Further let $\Re(h_j) = h_{j1}, \Im(h_j) = h_{j2}$ where $h_{j1} > 0$ by assumption for $j \in \{1,2\}$. Using (\ref{realOfRho}), we have  
\begin{align*}
        & h_{j1} = \Re(h_j)
         = \displaystyle\int \dfrac{\lambda\Re(\overline{-z + \lambda\rho(ch_j)})dH(\lambda)}{|-z + \lambda \rho(ch_j)|^2} = \int\dfrac{u\lambda + \lambda^2 [\rho_2(ch_j)\Re(ch_j)]}{|-z + \lambda\rho(ch_j)|^2}dH(\lambda)\\
\implies & h_{j1} = uI_1(h_j, H) + ch_{j1}\rho_2(ch_j)I_2(h_j, H) \label{realOf_h}\tag{4.6}\\           
\text {where } I_k(h_j, H) :=& \int \dfrac{\lambda^k dH(\lambda)}{|-z + \lambda \rho(ch_j)|^2} \text{ for } k \in \{1,2\}
\end{align*}

Note that $I_k(h_j, H) > 0, k \in \{1,2\}$ due to the conditions on H. Since $h_{j1} > 0$ and $u > 0$, using (\ref{realOf_h}), we must have 
\begin{align*}\tag{4.7}\label{lessThanOne}
 c\rho_2(ch_j)I_2(h_j, H) < 1
\end{align*}
    Then we have
\begin{equation*}
    \begin{split}
         h_1 - h_2
        =& \displaystyle\int \dfrac{(\rho(ch_2) - \rho(ch_1))\lambda^2}{[-z + \lambda\rho(ch_1)][-z + \lambda\rho(ch_2)]}dH(\lambda)\\
        =& (h_1-h_2)\displaystyle\int \dfrac{\dfrac{c\lambda^2}{(\mathbbm{i} +ch_1)(\mathbbm{i} +ch_2)} + \dfrac{c\lambda^2}{(-\mathbbm{i} +ch_1)(-\mathbbm{i} +ch_2)}}{[-z + \lambda\rho(ch_1)][-z + \lambda\rho(ch_2)]}dH(\lambda)\\
    \end{split}
\end{equation*}

By $\Ddot{H}$older's inequality, we have $|h_1-h_2| \leq |h_1-h_2|(T_1 + T_2)$ where 

\begin{align*}
T_1 =& \displaystyle\sqrt{\int\dfrac{c|\mathbbm{i} +ch_1|^{-2}\lambda^2dH(\lambda)}{|-z + \lambda\rho(ch_1)|^2}}\sqrt{\int\dfrac{c|\mathbbm{i} +ch_2|^{-2}\lambda^2dH(\lambda)}{|-z + \lambda\rho(ch_2)|^2}}\\    
=& \sqrt{c|\mathbbm{i} +ch_1|^{-2}I_2(h_1, H)}\sqrt{c|\mathbbm{i} +ch_2|^{-2}I_2(h_2, H)}\\
&\text{and}\\
T_2 =& \displaystyle\sqrt{\int\dfrac{c|-\mathbbm{i} +ch_1|^{-2}\lambda^2dH(\lambda)}{|-z + \lambda\rho(ch_1)|^2}}\sqrt{\int\dfrac{c|-\mathbbm{i} +ch_2|^{-2}\lambda^2dH(\lambda)}{|-z +\lambda\rho(ch_2)|^2}}\\
=& \sqrt{c|-\mathbbm{i} +ch_1|^{-2}I_2(h_1, H)}\sqrt{c|-\mathbbm{i} +ch_2|^{-2}I_2(h_2, H)}
\end{align*}
Noting that for $a,b,c,d \geq 0$, $\displaystyle\sqrt{ac} + \sqrt{bd} \leq \sqrt{a+b}\sqrt{c+d}$ with equality if and only if $a=b=c=d=0$, we get
\begin{align*}
    &T_1 + T_2\\
    =& \sqrt{c|\mathbbm{i} +ch_1|^{-2}I_2(h_1, H)}\sqrt{c|\mathbbm{i} +ch_2|^{-2}I_2(h_2, H)} +  \sqrt{c|-\mathbbm{i} +ch_1|^{-2}I_2(h_1, H)}\sqrt{c|-\mathbbm{i} +ch_2|^{-2}I_2(h_2, H)}\\
    \leq& \sqrt{(c|\mathbbm{i} +ch_1|^{-2} + c|-\mathbbm{i} +ch_1|^{-2})I_2(h_1, H)}\sqrt{(c|\mathbbm{i} +ch_2|^{-2} + c|-\mathbbm{i} +ch_2|^{-2})I_2(h_2, H)}\\
    =& \sqrt{c\rho_2(ch_1)I_2(h_1, H)}\sqrt{c\rho_2(ch_2)I_2(h_2, H)}
    <  1 \text{, using } (\ref{lessThanOne})
\end{align*}
This implies that $|h_1 - h_2|< |h_1 - h_2|$ which is a contradiction thus proving the uniqueness of $h(z) \in \mathbb{C}_R$.
\end{proof}
The proof of continuous dependence of $h$ on the distribution function $H$ is given in Section \ref{Continuity}.\\

\subsection{Existence of Solution}

It turns out that proving (2)--(4) of Section \ref{ProofSketch} is easier when we assume a set of stronger conditions (as in \cite{Lixin06}) \textbf{A1-A2} listed in Assumptions \ref{A123}. Our plan is to first prove the theorem under these assumptions. Then we build upon these to show that (2)--(4) of Section \ref{ProofSketch} will hold even under the general conditions given in Theorem \ref{t3.1}.

\subsubsection{Assumptions}\label{A123}
\begin{itemize}
    \item \textbf{A1} There exists a constant $\tau > 0$ such that $\underset{n \in \mathbb{N}}{\operatorname{sup}} ||\Sigma_n||_{op} \leq \tau$
    \item \textbf{A2} For $k \in \{1,2\}$,
    $\mathbb{E}z_{ij}^{(k)} = 0, |z_{ij}^{(k)}| \leq B_n$, where  $B_n := n^a$ with $\frac{1}{4+\eta_0}< a < \frac{1}{4}$ and $\eta_0 > 0$ is described in (T2) of Theorem \ref{t3.1}
\end{itemize}

We first establish a few important properties of $s_n(\cdot)$ and $h_n(\cdot)$ in Lemma \ref{CompactConvergence} and Lemma \ref{ConcentrationOfSnH}. Then we construct a sequence of deterministic matrices $\Bar{Q}(z) \in \mathbb{C}^{p \times p}$ satisfying $|\frac{1}{p}\operatorname{trace}(Q(z)-\Bar{Q}(z))| \xrightarrow{a.s.} 0$ in Theorem \ref{DeterministicEquivalent}. Finally, we prove the existence of (the) solution to (\ref{h_main_eqn}) in Theorem \ref{ExistenceA123}. 

\begin{lemma}\label{CompactConvergence}
    \textbf{(Compact Convergence)} 
    $\{s_n(z)\}_{n=1}^\infty$ and $\{h_n(z)\}_{n=1}^\infty$ form a normal family, i.e. every subsequence has a further subsequence that converges uniformly on compact subsets of $\mathbb{C}_L$. 
\end{lemma}
\begin{proof}
By Montel's theorem (Theorem 3.3 of \cite{SteinShaka}), it is sufficient to show that $s_n$ and $h_n$ are uniformly bounded on every compact subset of $\mathbb{C}_L$. Let $K \subset \mathbb{C}_L$ be an arbitrary compact subset. Define $u_0 := \inf\{|\Re(z)|: z \in K\}$. Then it is clear that $u_0 > 0$. Then for arbitrary $z \in K$, we have 
$$|s_n(z)| = \frac{1}{p}|\operatorname{trace}(Q(z))| \leq \frac{1}{|\Re(z)|} \leq \frac{1}{u_0}$$
    and using by (R5) of \ref{R123} and (T4) of Theorem \ref{t3.1}, we get for sufficiently large $n$,
    \begin{align}\label{h_nBound}\tag{4.8}
        |h_n(z)| =& \frac{1}{p}|\operatorname{trace}(\Sigma_nQ)| \leq \bigg(\frac{1}{p}\operatorname{trace}(\Sigma_n)\bigg) ||Q(z)||_{op} 
        \leq \frac{C}{|\Re(z)|} \leq \frac{C}{u_0}
    \end{align}
using the fact that for sufficiently large $n$, $\frac{1} {p}\operatorname{trace}(\Sigma_n) < C$ from T4 of Theorem \ref{t3.1}.
\end{proof}
 
\begin{theorem}\label{DeterministicEquivalent}
Let $M_n \in \mathbb{C}^{p \times p}$ be a sequence of deterministic matrices with $||M_n||_{op} \leq B$ for some $B \geq 0$. For $z \in \mathbb{C}_L$,
    $\frac{1}{p}\operatorname{trace}\{(Q(z) - \Bar{Q}(z))M_n\} \xrightarrow{a.s.} 0$
where $\Bar{Q}(z) := \bigg(-zI_p + \rho(c_n\mathbb{E}h_n(z))\Sigma_n\bigg)^{-1}$
\end{theorem}

\textbf{REMARK}: Let $z = -u + \mathbbm{i}v$ with $u > 0$. By Lemma \ref{boundedAwayFromZero}, $\Re(c_n\mathbb{E}h_n(z)) \geq K_0 > 0$ for sufficiently large $n$, where $K_0$ depends only on $z, c, \tau$ and $H$. So for large $n$, $\rho(c_n\mathbb{E}h_n(z))$ is well defined. Expressing $\Sigma_n = P\Lambda P^*$ with $\Lambda = \operatorname{diag}(\{\lambda_j\}_{j=1}^p)$, the $j^{th}$ eigenvalue of $\Bar{Q}(z)$ is $\sigma_j:= (-z + \lambda_j\rho(c_n\mathbb{E}h_n))^{-1}$. Then, for sufficiently large $n$ using (\ref{realOfRho}),
\begin{align*}\label{realOf_Qbar_EV_inverse}\tag{4.9}
\Re(\sigma_j^{-1}) = \Re(-z + \lambda_j\rho(c_n\mathbb{E}h_n)) = -\Re(z) + \lambda_j\Re(c_n\mathbb{E}h_n)\rho_2(c_n\mathbb{E}h_n) \geq u > 0    
\end{align*}
In particular, $(-zI_p + \rho(c_n\mathbb{E}h_n(z))\Sigma_n)$ is invertible for large $n$.

The construction of this deterministic sequence of matrices directly leads to proving the existence of a solution to (\ref{h_main_eqn}). The proof can be found in Section \ref{ProofDeterministicEquivalent}. 
\begin{definition}
    For $z \in \mathbb{C}_L$ and $\Bar{Q}(z)$ as defined in Theorem \ref{DeterministicEquivalent}, we define the following
     \begin{align}
        \Tilde{h}_n(z) :=&\frac{1}{p}\operatorname{trace}\{\Sigma_n \Bar{Q}(z)\} = \displaystyle\int\dfrac{\lambda dF^{\Sigma_n}(\lambda)}{-z + \lambda \rho(c_n\mathbb{E}h_n)} \label{definingHnTilde}\tag{4.10}\\
        \Bar{\Bar{Q}}(z) :=& \bigg(-zI_p + \rho(c_n\Tilde{h}_n(z))\Sigma_n\bigg)^{-1} \label{definingQBarBar}\tag{4.11}\\
        \Tilde{\Tilde{h}}_n(z) :=& \frac{1}{p}\operatorname{trace}\{\Sigma_n \Bar{\Bar{Q}}(z)\} = \displaystyle\int\dfrac{\lambda dF^{\Sigma_n}(\lambda)}{-z + \lambda \rho(c_n\Tilde{h}_n(z))} \label{definingHnTilde2}\tag{4.12}
    \end{align}
\end{definition}

\begin{theorem}\label{ExistenceA123} \textbf{(Existence of Solution)}
Under Assumptions \ref{A123}, for $z \in \mathbb{C}_L$, $h_n(z) \xrightarrow{a.s.} h^\infty(z) \in \mathbb{C}_R$ which satisfies (\ref{h_main_eqn}). Moreover, $s_n(z) \xrightarrow{a.s.} s^\infty(z)$ where $\underset{y \rightarrow +\infty}{\lim}ys^\infty(-y) = 1$ and 
$$s^\infty(z) = \frac{1}{z}\bigg(\frac{2}{c}-1\bigg) + \frac{1}{\mathbbm{i}cz}\bigg(\frac{1}{\mathbbm{i} +ch^\infty(z)} - \frac{1}{-\mathbbm{i} + ch^\infty(z)}\bigg)$$
\end{theorem}

The proof essentially expands Steps 2,3 and 4 outlined in Section \ref{ProofSketch}. These steps require the construction of additional quantities such as $\Tilde{h}_n(z)$ (\ref{definingHnTilde}) and $\Tilde{\Tilde{h_n}}(z)$ (\ref{definingHnTilde2}). The proof can be found in Section 
\ref{ProofExistenceA123}.

\subsection{Proof of Existence of solution under general conditions}\label{ExistenceGeneral}
In this section, we will prove (2)-(4) of Section \ref{ProofSketch} under the general conditions of Theorem \ref{t3.1}. To achieve this, we will create a sequence of matrices \textit{similar} to $\{S_n\}_{n=1}^\infty$ but satisfying \textbf{A1-A2} of Section \ref{A123}. The below steps give an outline of the proof with the essential details split into individual modules. 
\begin{description}
    \item[Step1]  For a p.s.d.  matrix A and $\tau > 0$, let $A^\tau$ represent the matrix obtained by replacing all eigenvalues greater than $\tau$ with 0 in the spectral decomposition of A. For a distribution function $H$ supported on $\mathbb{R}_+$, define $H^{\tau}(t) := \mathbbm{1}_{\{t > \tau\}} + (H(t) + 1 - H(\tau))\mathbbm{1}_{\{0 \leq t \leq \tau\}}$. $H^\tau$ is a distribution function that transfers all mass of $H$ beyond $\tau$ to the point $0$.
    \item[Step2] Denote $\Lambda_n := \Sigma_n^\frac{1}{2}$ and $\Lambda_n^\tau := (\Sigma_n^\tau)^\frac{1}{2}$
    \item[Step3] We have $S_n := \frac{1}{n}\Lambda_n(Z_1Z_2^* - Z_2Z_1^*)\Lambda_n$
    \item[Step4] Define $T_n := \frac{1}{n}\Lambda_n^{\tau}(Z_1Z_2^* - Z_2Z_1^*)\Lambda_n^{\tau}$. $\Lambda_n^\tau$ satisfies \textbf{A1} of Section \ref{A123}.
    \item[Step5] Recall that for $k \in \{1, 2\}$, we have $Z_k = ((z_{ij}^{(k)})) \in \mathbb{C}^{p \times n}$. With $B_n = n^a$ as in \textbf{A2} of Assumptions \ref{A123}, define $\hat{Z}_k := (\hat{z}_{ij}^{(k)})_{ij}$ with $\hat{z}_{ij}^{(k)} = z_{ij}^{(k)}\mathbbm{1}_{\{|z_{ij}^{(k)}| \leq B_n\}}$. Now, let $U_n := \frac{1}{n}\Lambda_n^\tau(\hat{Z_1}\hat{Z_2}^* - \hat{Z_2}\hat{Z_1}^*)\Lambda_n^\tau$.
    \item[Step6] Construct $\Tilde{U}_n:=\frac{1}{n}\Lambda_n^\tau(\Tilde{Z_1}\Tilde{Z_2}^* - \Tilde{Z_2}\Tilde{Z_1}^*)\Lambda_n^\tau$ where $\Tilde{Z}_k = \hat{Z}_k - \mathbb{E}\hat{Z}_k$. Now, $\Sigma_n^\tau$ satisfies \textbf{A1} and $\Tilde{Z}_k$, $k=1,2$ satisfy \textbf{A2} of Assumptions \ref{A123}. Let $s_n(z), t_n(z), u_n(z), \Tilde{u}_n(z)$ be the the Stieltjes transforms of $F^{S_n}, F^{T_n},$ $F^{U_n}, F^{\Tilde{U}_n}$ respectively.
    \item[Step7] By Theorem \ref{ExistenceA123}, $F^{\Tilde{U}_n} \xrightarrow{a.s.} F^\tau$ for some $F^\tau$ which is characterised by a pair $(h^\tau, s^\tau)$ satisfying (\ref{s_main_eqn}) and (\ref{h_main_eqn}) with $H^\tau$ instead of $H$. In particular, $|\Tilde{u}_n(z) - s^\tau(z)| \xrightarrow{a.s.} 0$ by the same theorem.
    \item[Step8] Next we show that $h^\tau$ converges to some limit as $\tau \rightarrow \infty$. Using Montel's Theorem, we are able to show that any arbitrary subsequence of $\{h^\tau\}$ has a further subsequence $\{h^{\tau_m}\}_{m=1}^\infty$ that converges uniformly on compact subsets (of $\mathbb{C}_L$) as $m \rightarrow \infty$. Each subsequential limit will be shown to belong to $\mathbb{C}_R$ and satisfy (\ref{h_main_eqn}). Hence by Theorem \ref{Uniqueness}, all these subsequential limits must be the same which we denote by $h^\infty$. Therefore, $h^\tau \xrightarrow{\tau \rightarrow \infty} h^\infty$.
    \item[Step9] We derive $s^\infty$ from $h^\infty$ using (\ref{s_main_eqn}) and show that $s^\infty$ satisfies the condition in Proposition \ref{GeroHill} to be a Stieltjes transform of a measure over the imaginary axis. So, there exists some distribution $F^\infty$ corresponding to $s^\infty$. 
   \textbf{Step8} and \textbf{Step9} are shown explicitly in Lemma \ref{3.4.1}.
    \item[Step10]  We have
$$|s_n(z) - s^\infty(z)| \leq |s_n(z) - t_n(z)| + |t_n(z)-u_n(z)| + |u_n(z)-\Tilde{u}_n(z)| + |\Tilde{u}_n(z)-s^\tau(z)| + |s^\tau(z) - s^\infty(z)|$$
We will show that each term on the RHS goes to 0 as $n, \tau \rightarrow \infty$. \begin{itemize}
    \item  From Lemma \ref{3.4.2} and (\ref{Levy_vs_uniform}), $L_{im}(F^{S_n}, F^{T_n}) \leq ||F^{S_n} - F^{T_n}||_{im} \xrightarrow{a.s.} 0$
    \item From Lemma \ref{3.4.3} and (\ref{Levy_vs_uniform}), $L_{im}(F^{T_n}, F^{U_n}) \leq ||F^{T_n} - F^{U_n}||_{im} \xrightarrow{a.s.} 0$ 
    \item From Lemma \ref{3.4.3} and (\ref{Levy_vs_uniform}), $L_{im}(F^{U_n}, F^{\Tilde{U}_n}) \leq ||F^{U_n} - F^{\Tilde{U}_n}||_{im} \xrightarrow{a.s.} 0$
    \item Application of Lemma \ref{lA.1} on the above three items gives $|s_n(z) - t_n(z)| \xrightarrow{a.s.} 0$, $|t_n(z) - u_n(z)| \xrightarrow{a.s.} 0$ and $|u_n(z) - \Tilde{u}_n(z)| \xrightarrow{a.s.} 0$ respectively.
    \item From \textbf{Step7}, we have $|\Tilde{u}_n(z) - s^\tau(z)| \xrightarrow{a.s.} 0$
    \item $|s^\tau(z) - s^\infty(z)| \rightarrow 0$ is shown in Lemma \ref{3.4.1}.
\end{itemize}
\item[Step11] Hence  $s_n(z) \xrightarrow{a.s.} s^\infty(z)$ which is a Stieltjes transform. Therefore by Proposition \ref{GeroHill}, $F^{S_n} \xrightarrow{a.s.} F^\infty$ where $F^\infty$ is characterised by $(h^\infty, s^\infty)$ which satisfy (\ref{s_main_eqn}) and (\ref{h_main_eqn}). This concludes the proof of Theorem \ref{t3.1} in the general case.
\end{description}

\subsection{Properties of the LSD}\label{Section4.5}
In this section, we mention a few properties about the LSD in Theorem \ref{t3.1}. 
\begin{lemma}\label{symmetryAboutZero}
    For any $z \in \mathbb{C}_L$, $h(\overline{z}) = \overline{h(z)}$ and $s(\overline{z}) = \overline{s(z)}$ where $h, s$ are as in Theorem \ref{t3.1}. As a consequence, the LSD is symmetric about $0$. 
\end{lemma}

\begin{proof}
We have the following equality.
    \begin{align*}
            h(z) &= \int \dfrac{\lambda dH(\lambda)}{-z + \lambda\rho(ch(z))}
            \implies \overline{h(z)} = \int \dfrac{\lambda dH(\lambda)}{-\overline{z} + \lambda\overline{\rho(ch(z))}} = \int \dfrac{\lambda dH(\lambda)}{-\overline{z} + \lambda{\rho(c\overline{h(z)})}}\\
    \end{align*}
Note that $z \in \mathbb{C}_L \Leftrightarrow \overline{z} \in \mathbb{C}_L$ and $h(z) \in \mathbb{C}_R$ uniquely satisfies (\ref{h_main_eqn}). We observe that $\overline{h(z)} \in \mathbb{C}_R$ satisfies (\ref{h_main_eqn}) with $z$ replaced by $\overline{z}$. Therefore by Theorem \ref{Uniqueness}, $h(\overline{z}) = \overline{h(z)}$ for $z \in \mathbb{C}_L$. Finally we see that, 
\begin{align*}
s(\overline{z}) =& \dfrac{1}{\overline{z}}(\dfrac{2}{c}-1) + \dfrac{1}{\mathbbm{i}c\overline{z}}\bigg(\dfrac{1}{\mathbbm{i} + ch(\overline{z})} - \dfrac{1}{-\mathbbm{i}+ch(\overline{z})}\bigg)\\
=& \overline{\frac{1}{z}(\frac{2}{c}-1)}+ \overline{\frac{1}{-\mathbbm{i}cz}\bigg(\frac{1}{-\mathbbm{i}+ch(z)} - \frac{1}{\mathbbm{i} + ch(z)}\bigg)} = \overline{s(z)}
\end{align*}
Using this, (\ref{inversion_pointMass}) and (\ref{measureOfInterval}), the symmetry about $0$ is immediate.
\end{proof}

\begin{theorem}\label{PointMass}
Suppose $H = (1-\beta) \delta_0 + \beta H_1$ where $H_1$ is a probability distribution over $\mathbb{R}$ which has no point mass at $0$ and $0 < \beta \leq 1$. (Note that $\beta > 0$ to exclude the case $H = \delta_0$). Then,
\begin{enumerate}
    \item[1] When $0 < c < 2/\beta$, the LSD $F$ has a point mass at $0$ equal to $1 - \beta$.
    \item[2] When $2/\beta \leq c$, the LSD $F$ has a point mass at $0$ equal to $1 - 2/c$.
\end{enumerate}   
 \end{theorem}

\begin{proof}
    Note that from Lemma \ref{symmetryAboutZero}, $\overline{h(-\epsilon)} = h(-\epsilon) \implies h(-\epsilon) \in \mathbb{R}$ for any $\epsilon > 0$. Also, since $h(\mathbb{C}_L) \subset \mathbb{C}_R$, we must have $h(-\epsilon) > 0$. We will now show that when $c\geq 2/\beta$, we must have 
    \begin{align}\label{lim_h_is_infty}
        \underset{\epsilon \downarrow 0}{\lim}h(-\epsilon) = \infty
    \end{align}  
    If we assume the contrary, then there exists some $M > 0$ such that   
    \begin{align}\label{h_epsilon_bound}
        h(-\epsilon) < M
    \end{align}
Then, for any sequence $\{\epsilon_n\}_{n=1}^\infty$ with $\epsilon_n \downarrow 0$, we have $|h(-\epsilon_n)| < M$. So there exists a subsequence $\{n_k\}_{k=1}^\infty$ such that $\underset{k \rightarrow \infty}{\lim} h(-\epsilon_{n_k}) = z_0$ where $z_0 \geq 0$ since $h(\mathbb{C}_L) \subset \mathbb{C}_R$. By Fatou's Lemma, we observe the following inequality.

\begin{align}
     &h(-\epsilon_{n_k}) = \int \frac{\lambda dH(\lambda)}{\epsilon_{n_k} + \lambda \rho(ch(-\epsilon_{n_k}))} = \int \frac{\beta\lambda dH_1(\lambda)}{\epsilon_{n_k} + \lambda \rho(ch(-\epsilon_{n_k}))}\\
  \implies& z_0 = \underset{k \rightarrow \infty}{\liminf} h(-\epsilon_{n_k}) \geq \int \underset{k \rightarrow \infty}{\liminf} \frac{\beta \lambda dH_1(\lambda)}{\epsilon_{n_k} + \lambda \rho(ch(-\epsilon_{n_k}))} \notag
\end{align}

\textbf{Case1}: If $z_0 = 0$, then we get $0 \geq \infty$.

\textbf{Case2:} If $z_0 > 0$, we observe that for large $k$
$$\frac{\beta\lambda}{\epsilon_{n_k} + \lambda\rho(ch(-\epsilon_{n_k}))} \leq \frac{2\beta}{\rho(cz_0)}$$

Applying D.C.T., we get another contradiction as follows
  \begin{align}\label{finite_limit}
      &z_0 = \int \frac{\beta dH_1(\lambda)}{\rho(cz_0)} \notag\\
      \implies& z_0\rho(cz_0) = \beta \notag\\
      \implies& z_0 \frac{2cz_0}{1+c^2z_0^2} = \beta \notag\\
      \implies& 2cz_0^2 = \beta(1 + c^2z_0^2) \notag \\
      \implies& c(2 - c\beta)z_0^2 = \beta > 0
  \end{align}
The LHS of the last inequality is clearly non-positive whereas the RHS is positive. Therefore, $\not \exists M > 0$ such that (\ref{h_epsilon_bound}) holds. Thus, $\underset{\epsilon \downarrow 0}{\lim}h(-\epsilon) = \infty$ for $2/\beta \leq c$.

\textbf{REMARK:} One implication of this is that the existence of a bound ($M > 0$) for some $c < 2/\beta$ is sufficient to imply that any subsequential limit must be equal to the common value of $\sqrt{\frac{\beta}{c(2-c\beta)}}$.

When $0 < c < 2/\beta$, we will show that 
\begin{align}\label{value_of_limit}
 \underset{\epsilon \downarrow 0}{\lim}h(-\epsilon) = +\sqrt{\dfrac{\beta}{c(2-c\beta)}} = z_0 \text{ (say)}   
\end{align}

Note that in light of the remark following (\ref{finite_limit}), all we need to do is show that $h(-\epsilon)$ is bounded. Define the function $G(\epsilon, t): \mathbb{R}_+^2 \rightarrow \mathbb{R}_+$ as follows.
\begin{align}\label{eq:G_epsilon_t_repr}
G(\epsilon, t) := \int \frac{\lambda dH(\lambda)}{\epsilon + \lambda \rho(ct)} = \int \frac{\beta \lambda dH_1(\lambda)}{\epsilon + \lambda \rho(ct)} \text{, where } t > 0, \epsilon > 0
\end{align} 

\textbf{REMARK:} Note that $\rho(x) = \dfrac{2x}{1+x^2}$ attains it maximum at $x = 1$ and goes to $0$ at both ends. For any fixed $\epsilon > 0$, $G(\epsilon, t)$ is continuous in $t$ and attains its minimum when $\rho(ct)$ is maximum, i.e. when $t = 1/c$. So, $G(\epsilon, \cdot)$ decreases in $(0, 1/c)$ and increases in $(1/c, \infty)$. 
Since 
$$\dfrac{\lambda}{\epsilon + \lambda \rho(1)} = \dfrac{\lambda}{\epsilon + \lambda} \leq 1$$  
we apply D.C.T. to get
\begin{align}
    \underset{\epsilon \downarrow 0}{\lim} G(\epsilon, 1/c) = \int \underset{\epsilon \downarrow 0}{\lim} \frac{\beta\lambda dH_1(\lambda}{\epsilon + \lambda \rho(1)}= \frac{\beta}{\rho(1)} = \beta
\end{align}
Since $G(\epsilon, 1/c)$ (strictly) increases as $\epsilon > 0$ decreases, we have $G(\epsilon, 1/c) < \beta$ for any $\epsilon>0$. Now we present the following chain of arguments to establish (\ref{value_of_limit}). 
    Then, observe that 
\begin{align}
    z_0 =& \frac{\beta(1+c^2z_0^2)}{2cz_0} =\frac{\beta}{\rho(cz_0)}\notag\\
    \implies z_0 =& \underset{\epsilon \downarrow 0}{\lim} \int \frac{\beta \lambda dH_1(\lambda)}{\epsilon + \lambda \rho(cz_0)} \text{, by D.C.T.} \notag\\
    \implies z_0 =& \underset{\epsilon \downarrow 0}{\lim}G(\epsilon, z_0) \notag
\end{align}
\begin{enumerate}
    \item[1] For a fixed $t$, $G(\epsilon, t)$ increases as $\epsilon > 0$ decreases.
    \item[2] For any $\epsilon > 0$, $h(-\epsilon)$ is the (unique) point where $G(\epsilon, h)$ cuts the line $y = x$.
    \item[3] So, $G(\epsilon, z_0)$ increases to $z_0$ as $\epsilon > 0$ goes to $0$.
    \item[4] Let $C_H = \int \lambda dH(\lambda) = \beta \int \lambda dH_1(\lambda)$. Clearly, $0 < C_H < \infty$. So, for any $\epsilon>0$, we have $G(\epsilon,0) = C_H/\epsilon > 0$.    
    \item[5] When $0 < c \leq 1/\beta$, we have $z_0 \leq 1/c$. 
    By the remark following (\ref{eq:G_epsilon_t_repr}), $G(\epsilon, \cdot)$ is decreasing on  $(0, 1/c]$. 
    Further, since $G(\epsilon, 0) > 0$ and $G(\epsilon,z_0) < z_0$, in a right neighborhood of $0$, the graph of
    $G(\epsilon, \cdot)$ is in $\{x < y\}$
    and in a neighborhood of $z_0$ it is 
    in $\{x > y\}$.
    Therefore, the curve of $G(\epsilon, \cdot)$ must intersect the diagonal to the left of $z_0$. Thus $h(-\epsilon) \leq z_0$.
    \item[6]  When $1/\beta < c < 2/\beta$, we have $z_0 > 1/c$ and $1/c < \beta$. 
    Since $G(\epsilon,1/c)$ increases to $\beta$ as $\epsilon$ decreases to zero, for $0 < \epsilon < \epsilon_0(c,\beta)$, we have $G(\epsilon,1/c) > 1/c$. 
    Further, by the remark following (\ref{eq:G_epsilon_t_repr}), in $(1/c, \infty)$, $G(\epsilon, \cdot)$ is an increasing 
    function. Also, $G(\epsilon,z_0) < z_0$. 
    Therefore, in $[1/c, z_0]$ the graph of $G(\epsilon, \cdot)$ moves from the region $\{x < y\}$ to the region $\{x > y\}$.  Hence, the graph of $G(\epsilon, \cdot)$ must intersect the diagonal to the left of $z_0$. Thus $h(-\epsilon) \leq z_0$.
    \item[7] Therefore, $z_0$ is an upper bound for $h(-\epsilon)$ for any $0 < c < 2/\beta$.
\end{enumerate}
Now using the arguments presented in (\ref{finite_limit}) and the subsequent remark, we get (\ref{value_of_limit}).

The results of the theorem follow immediately when we plug in (\ref{lim_h_is_infty}) and (\ref{value_of_limit}) to the below formula for the point mass at $0$.
\begin{align}
    F(\{0\})  = \underset{\epsilon \downarrow 0}{\lim}\epsilon s(-\epsilon) = 1 - \frac{2}{c} + \underset{\epsilon \downarrow 0}{\lim}\frac{2}{c}\frac{1}{1 + c^2 h^2(-\epsilon)}
\end{align}

\begin{figure}[ht]
    \centering
    \includegraphics[width=0.75\linewidth]{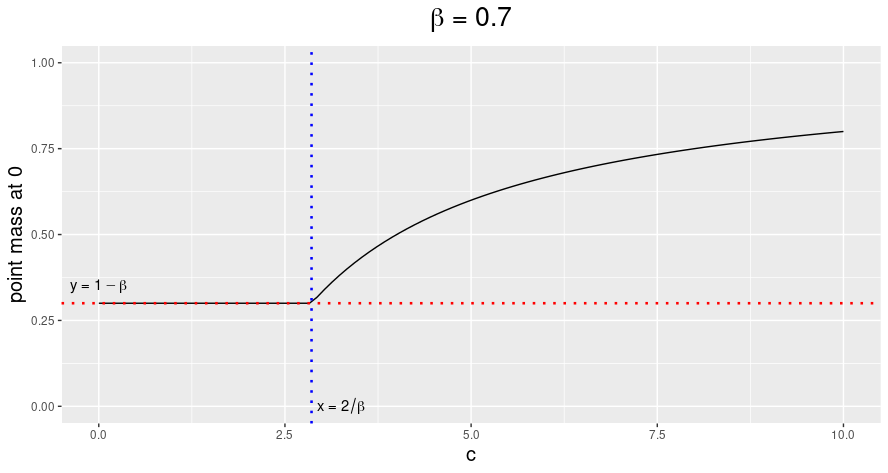}
    \caption{The value of the point mass of the LSD at $0$ as $c$ varies when $\beta = 0.7$}
    \label{fig:1}
\end{figure}
Figure \ref{fig:1} shows the behavior of the LSD at the point $0$. The point mass at 0 remains at $1 - \beta$ as long as $0 < c < 2/\beta$. Once $c \geq 2/\beta$, the point mass is given by $1 - 2/c$.  Note that the results of the theorem depend only on $\beta$ and $c$ and not on the choice of $H$.
\end{proof}

\section{LSD when the common covariance is the identity matrix}\label{sec:identity_covariance}

\subsection{Some properties of the Stieltjes Transform}

When $\Sigma_n = I_p$ a.s., $F^{\Sigma_n} = \delta_1$ $\forall n \in \mathbb{N}$ and $F^{\Sigma_n} \xrightarrow{d} \delta_1$ a.s. From Theorem \ref{t3.1}, there exists a probability distribution function $F$ on the imaginary axis such that $F^{S_n} \xrightarrow{d} F$. The LSD $F$ is characterised by $(h, s_F)$ with $h$ satisfying (\ref{h_main_eqn}) with $H = \delta_1$ and $(h,s_F)$ satisfies (\ref{s_main_eqn}). Moreover, $s_F(z)$ is the Stieltjes Transform of $F$ at $z \in \mathbb{C}_L$. The goal of this section is to recover closed form expressions for the distribution $F$ which is achieved in Theorem \ref{DensityDerivation}.

We first note that $h(z)$, the unique solution to (\ref{h_main_eqn}) with positive real part is the same as $s_F(z)$ in this case. This is shown below. Writing $h(z)$ as $h$ for simplicity, we have from (\ref{h_main_eqn})
\begin{align*}
     &\frac{1}{h} = -z + \frac{1}{\mathbbm{i} +ch}+\frac{1}{-\mathbbm{i} +ch} \text{, note that } \Re(ch) > 0 \text{ by Lemma } \ref{boundedAwayFromZero}\\
     \implies & c^2zh^3 + (c^2-2c)h^2 +zh + 1=0\\
     \implies & c^2zh^3 + zh = -1 - h^2(c^2-2c)\\
     \implies & c^3zh^3 +czh = -c - c^2h^2(c-2)\\
     \implies & czh(c^2h^2 + 1) = 2-c + c^2h^2(2-c) - 2 = (2-c)(c^2h^2+1) - 2\\
     \implies & czh = 2-c - \frac{2}{1 + c^2h^2} = 2-c + \frac{1}{\mathbbm{i}}\bigg(\frac{1}{\mathbbm{i} +ch} - \frac{1}{-\mathbbm{i} +ch}\bigg)\\
     \implies& h = \frac{1}{z}\bigg(\frac{2}{c}-1\bigg) + \frac{1}{\mathbbm{i}cz}\bigg(\frac{1}{\mathbbm{i} +ch} - \frac{1}{-\mathbbm{i} +ch}\bigg) = s_F(z) \text{, by } (\ref{s_main_eqn}) \label{h_equal_s}\tag{5.1}
\end{align*}

Therefore, the Stieltjes Transform ($s_F(z)$) of the LSD at $z \in \mathbb{C}_L$ can be recovered by finding the unique solution with positive real part (exactly one exists by Theorem \ref{t3.1}) to the following equation. 
\begin{align*}\label{5A}\tag{5.2}
     &\frac{1}{m(z)} = -z + \frac{1}{\mathbbm{i} +cm(z)} + \frac{1}{-\mathbbm{i} +cm(z)}
 \end{align*}

By Lemma \ref{boundedAwayFromZero} and (\ref{h_equal_s}), for $z \in \mathbb{C}_L$, $\Re(ch(z)) = \Re(cs_F(z)) > 0$. In particular we do not have to worry about issues like $\pm \mathbbm{i} = cs_F(z)$. Therefore, we simplify (\ref{5A}) to an equivalent functional cubic equation which is more amenable for recovering the roots.
\begin{align*}\label{5B}\tag{5.3}
         &c^2zm^3(z) + (c^2 - 2c)m^2(z) + zm(z) + 1 = 0
\end{align*}

For $z \in \mathbb{C}_L$, we extract the functional roots $\{m_j(z)\}_{j=1}^3$ of (\ref{5B}) using \textit{Cardano's method} (subsection 3.8.2 of \cite{Abramowitz}) and select the one which has a positive real component. This will serve as the Stieltjes Transform of the LSD. 

\subsection{Deriving the functional roots}\label{Cardano}

We define the following quantities as functions of $c \in (0 , \infty)$.
\begin{equation*}\tag{5.4}\label{RQD}
    \left\{ \begin{aligned} 
    &q_0 = \frac{1}{3c^2}; \hspace{18mm} q_2 =-\frac{(c-2)^2}{9c^2}; \hspace{5mm} \Tilde{q} = (q_0, q_2)\\
    &r_1 = -\frac{c+1}{3c^3}; \hspace{12mm} r_3 = -\frac{(c-2)^3}{27c^3}; \hspace{5mm} \Tilde{r} = (r_1, r_3)\\
    &d_0 = \dfrac{1}{27c^6}; \hspace{17mm} d_2 = \dfrac{2c^2 + 10c - 1}{27c^6}; \hspace{5mm}d_4 = \dfrac{(1-2/c)^3}{27c^2}; \hspace{5mm} \Tilde{d} = (d_0,d_2, d_4)\\
    &Q(z) := q_0 + \frac{q_2}{z^2}; \hspace{7mm} R(z) := \frac{r_1}{z} + \frac{r_3}{z^3}; z \in \mathbb{C}\backslash\{0\}
\end{aligned} \right.
\end{equation*}

By Cardano's method, the three roots of the cubic equation (\ref{5B}) are given as follows where $1, \omega_1, \omega_2$ are the cube roots of unity.
\begin{equation*}\tag{5.5}\label{RootFormula}
\left\{ \begin{aligned} 
m_1(z) &= -\dfrac{1-2/c}{3z} + S_{0} + T_{0}\\
m_2(z) & = -\dfrac{1-2/c}{3z} + \omega_1S_{0} + \omega_2T_{0}\\
m_3(z) & = -\dfrac{1-2/c}{3z} + \omega_2S_{0} + \omega_1T_{0}
\end{aligned} \right.
\end{equation*} 

where $S_0$ and $T_0$ are (complex) quantities satisfying 
\begin{align*}\tag{5.6}\label{S0T0}
S_0^3 + T_0^3 &= 2R(z)\\
S_0T_0 &= -Q(z)
\end{align*}
Note that if ($S_0, T_0$) satisfy (\ref{S0T0}), then so do ($\omega_1S_0, \omega_2T_0$) and ($\omega_2S_0, \omega_1T_0$). But exactly one of the functional roots of (\ref{5B}) is the Stieltjes Transform $s_F(z)$. However, the ambiguity in the definition of $S_{0}$ and $T_{0}$ prevents us from pinpointing which one among $\{m_j(z)\}_{j=1}^3$ is the Stieltjes transform of $F$ at $z$ unless we explicitly solve for the roots. 

An interesting property of the roots of equation (\ref{5B}) can be found in Proposition \ref{prop_C.4}. We also have a result regarding the continuity of the Stieltjes Transform ($s_F$) as a function of $c$ in Proposition \ref{continuity_in_c} and regarding the location of the Stieltjes Transform in Proposition \ref{Q2_to_Q1}.

\subsection{Deriving the density of the LSD}

Certain properties of the LSD such as symmetry about $0$ and existence of point mass of $1-2/c$ at $0$ when $c > 2$ have already been established in Section \ref{Section4.5}. Here, we introduce certain quantities that will be required while evaluating the density of the LSD in Theorem \ref{DensityDerivation}. 

The LSD $F$ is totally supported on (a subset of) the imaginary axis. Denote $\Tilde{F}$ as the distribution of a real random variable $Y$ where $\mathbbm{i}Y \sim F$. Theorem \ref{DensityDerivation} is regarding the density of $\Tilde{F}$. We first introduce a few quantities that are essential to parametrize said density.

\begin{definition}\label{supportParameters}
 For $c > 0$, let $\Tilde{d}, R(.), Q(.)$ be as in (\ref{RQD}). Then define
\begin{enumerate}
    \item $R_{\pm} := \dfrac{d_2 \pm \sqrt{d_2^2-4d_0d_4}}{2d_0}; \hspace{5mm}$
    $R_\pm$ are real numbers as shown in Lemma \ref{DistributionParameters}
    \item $L_c := \sqrt{R_{-}\mathbbm{1}_{\{R_{-} > 0\}}}$; \hspace{3mm}  $U_c := \sqrt{R_{+}}$
    \item $S_c := (-U_c, -L_c)\cup (L_c, U_c)$; It denotes the smallest open set excluding the point $0$\footnote{The point $0$ is treated separately in Theorem \ref{DensityDerivation} as the density at $0$ exists only when $0 < c < 2$} where the density of the LSD is finite.
    \item For $x \neq 0$, let $r(x) := \underset{\epsilon \downarrow 0}{\lim} R(-\epsilon + \mathbbm{i}x)$ and $q(x) := \underset{\epsilon \downarrow 0}{\lim} Q(-\epsilon + \mathbbm{i}x)$\\
    Results related to these limits are established in Lemma \ref{DistributionParameters}
    \item For $x \neq 0$, $d(x) :=  d_0 - \dfrac{d_2}{x^2} + \dfrac{d_4}{x^4}$ 
\end{enumerate}
\end{definition}

\begin{theorem}\label{DensityDerivation}
    $\Tilde{F}$ is differentiable at $x \neq 0$ for any $c > 0$. Define $V_\pm(x) := |r(x)| \pm \sqrt{-d(x)}$. The functional form of the density is given by    
    $$f_c(x) = \dfrac{\sqrt{3}}{2\pi}\left((V_+(x))^{\frac{1}{3}} - (V_-(x))^{\frac{1}{3}}\right)\mathbbm{1}_{\{x \in S_c\}}$$ 
    At $x=0$, the derivative exists when $0 < c < 2$ and is given by
    $$f_c(0) = \dfrac{1}{\pi\sqrt{2c - c^2}}$$
    The density is continuous wherever it exists.
    \end{theorem}
The proof can be found in Section \ref{ProofDensityDerivation}.

\subsection{Simulation study}

We ran simulations for different values of $c$. Figures \ref{fig:c=0.5}, \ref{fig:c=1}, \ref{fig:c=2}, \ref{fig:c=3} \ref{fig:c=4} and \ref{fig:c=5} below show the comparison of the ESDs of these matrices for different values of $c$ keeping $p = 2000$ against the derived theoretical distribution. The matrices were generated with independent observations, a random half of which were simulated from $\mathcal{N}(0, 1)$\footnote{$\mathcal{N}(\mu,\sigma^2)$ represents the Gaussian distribution with mean $\mu$ and variance $\sigma^2$} and the other half from $\mathcal{U}(-\sqrt{3}, \sqrt{3}$)\footnote{$\mathcal{U}(a,b)$ represents the uniform distribution over the interval $(a,b)$}.

\begin{figure}[ht]
 \begin{subfigure}{0.43\textwidth}
     \includegraphics[width=\textwidth]{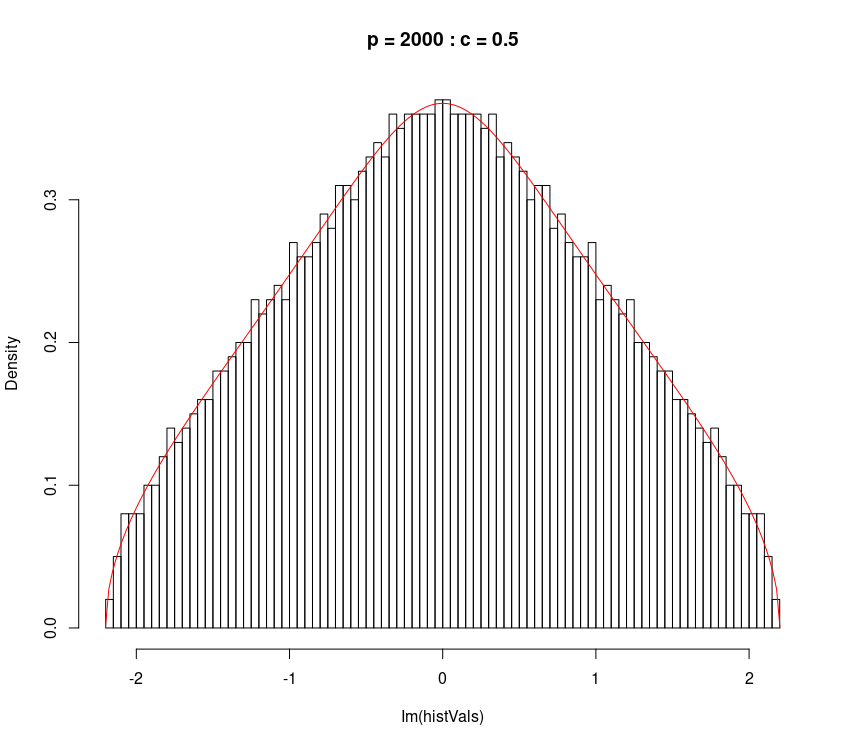}
     \caption{c=0.5}
     \label{fig:c=0.5}
 \end{subfigure}
 \hfill
 \begin{subfigure}{0.43\textwidth}
     \includegraphics[width=\textwidth]{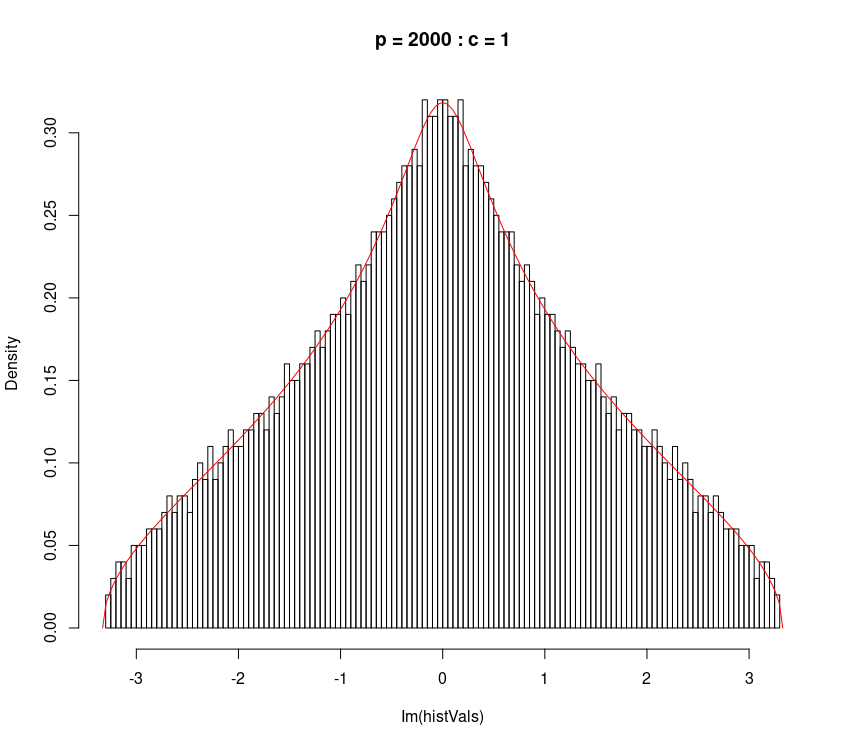}
     \caption{c=1}
     \label{fig:c=1}
 \end{subfigure}
  \medskip

 \begin{subfigure}{0.43\textwidth}
     \includegraphics[width=\textwidth]{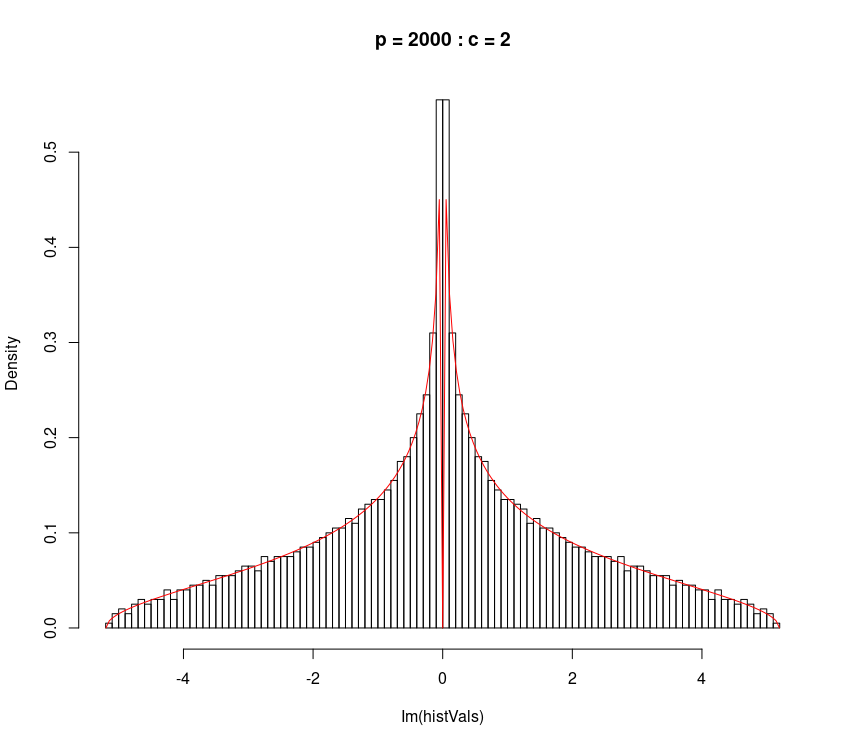}
     \caption{c=2}
     \label{fig:c=2}
 \end{subfigure}
 \hfill
 \begin{subfigure}{0.43\textwidth}
     \includegraphics[width=\textwidth]{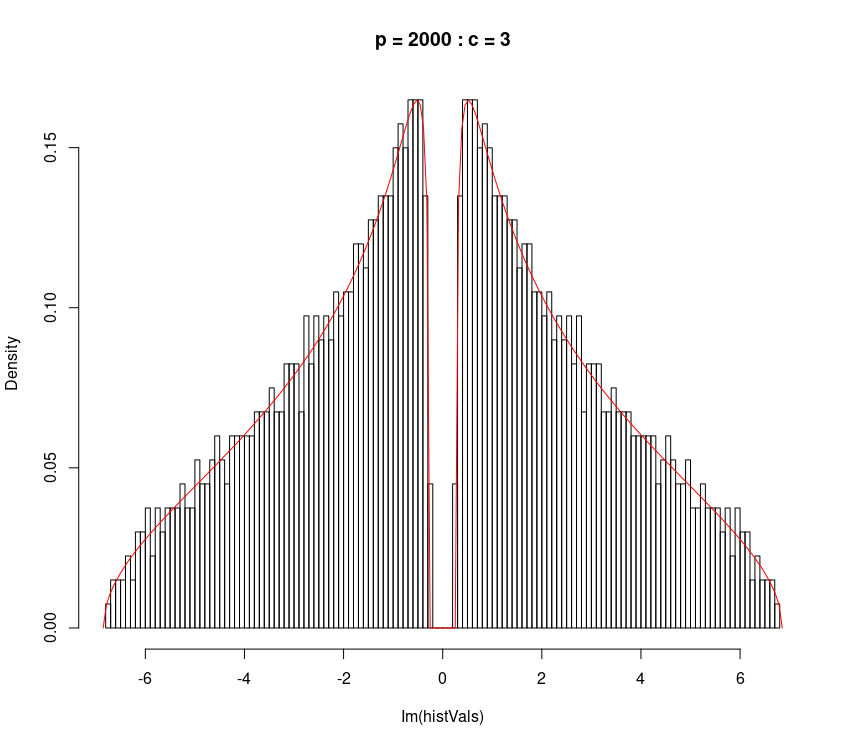}
     \caption{c=3}
     \label{fig:c=3}
 \end{subfigure}

 \medskip
 \begin{subfigure}{0.43\textwidth}
     \includegraphics[width=\textwidth]{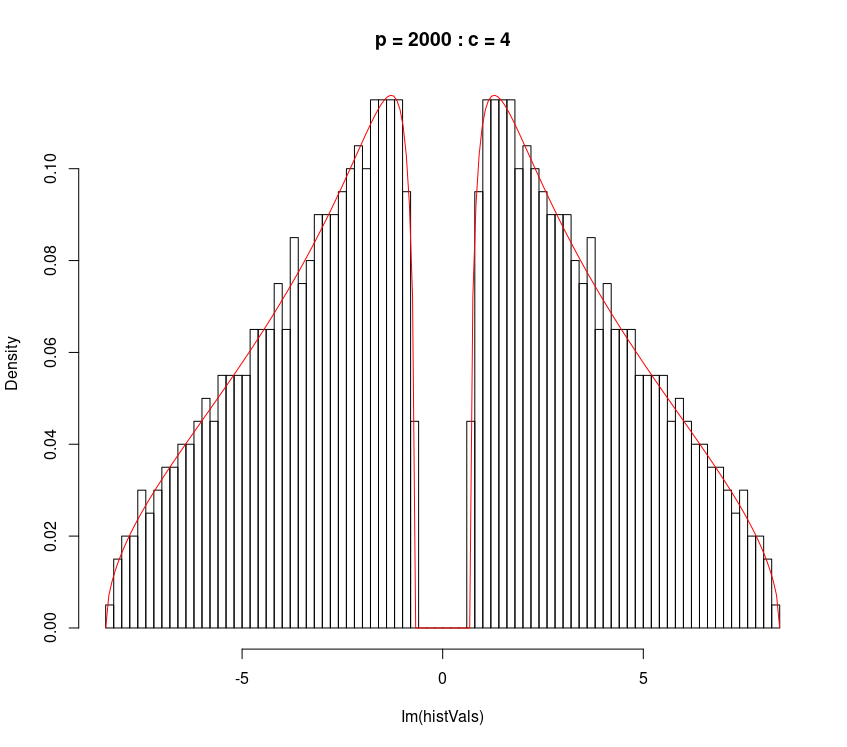}
     \caption{c=4}
     \label{fig:c=4}
 \end{subfigure}
 \hfill
 \begin{subfigure}{0.43\textwidth}
     \includegraphics[width=\textwidth]{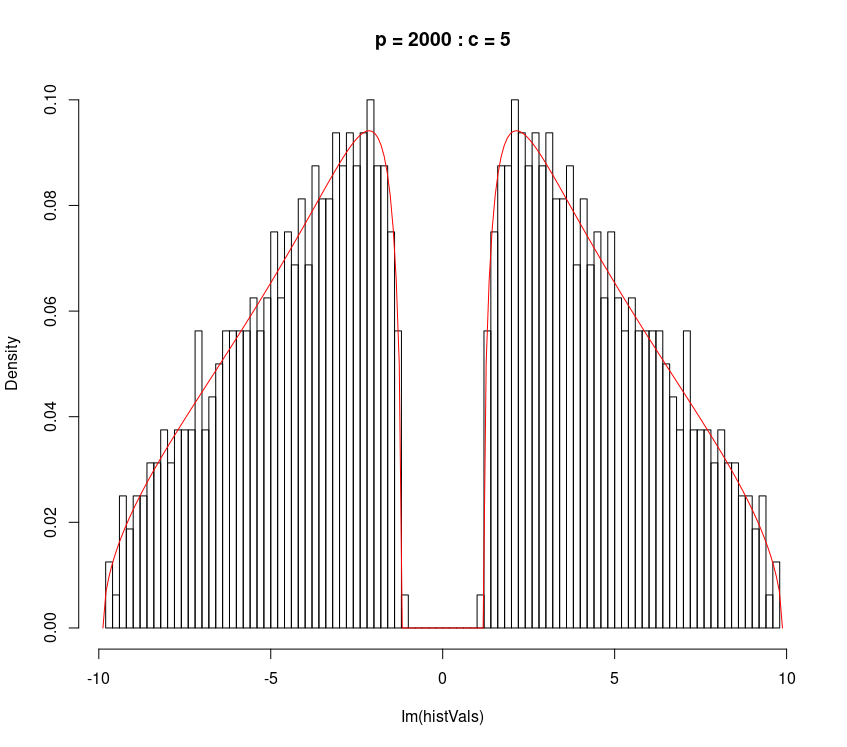}
     \caption{c=5}
     \label{fig:c=5}
 \end{subfigure}
 \caption{Simulated vs. theoretical limit distributions at various levels of $c$ for $\Sigma_n = I_p$}
 \label{Label}
\end{figure}

\newpage
\section{The Hermitian Case}\label{sec:Hermitian}
We have also derived analogous results case for the Hermitian case (i.e. $S_n^+$ in (\ref{defining_Sn})). The conditions under which the ESDs of $S_n^+$ approach a non-random limit are exactly the same as those in Theorem \ref{t3.1}.
The main result is stated below.

\begin{theorem}\label{t6.1}
Let $S_n^+$ be as defined in (\ref{defining_Sn}). Under (\textbf{T1}-\textbf{T4}) of Theorem \ref{t3.1}, $F^{S_n^+} \xrightarrow{d} F^+$ a.s. where $F^+$ is a non-random distribution with  Stieltjes Transform at $z \in \mathbb{C}^+$\footnote{$\mathbb{C}^+ := \{z \in \mathbb{C}: \Im(z) > 0\}$} given by
\begin{equation*}\label{6.1}\tag{6.1}
s(z) = \dfrac{1}{z}\bigg(\dfrac{2}{c}-1\bigg) + \dfrac{1}{cz}\bigg(\dfrac{1}{-1 +ch(z)} - \dfrac{1}{1 +ch(z)}\bigg)    
\end{equation*}
where $h(z) \in \mathbb{C}^+$ is the unique number such that 
\begin{equation*}\label{6.2}\tag{6.2}
h(z) = \int \dfrac{\lambda dH(\lambda)}{\bigg[-z + \lambda\bigg(\dfrac{1}{1 +ch(z)} + \dfrac{1}{-1 + ch(z)}\bigg)\bigg]}    
\end{equation*}
Further $h$ is analytic in $\mathbb{C}^+$ and has a continuous dependence on $H$.
\end{theorem}
The proof can be found in Appendix \ref{sec:AppendixD}.

Similar results as those in Section \ref{sec:identity_covariance} have also been established for the Hermitian case. As mentioned in the introduction, the limiting distribution (support, point mass at $0$, density) is exactly the same as the one in the skew-Hermitian case except that it is supported on the real axis.

\newpage

\subsection*{Acknowledgement}

The authors are grateful to Professor Arup Bose for some insightful discussions and suggestions. 

\bibliographystyle{plain}
\bibliography{bibli}

\newpage
\appendix
\section{General Results}
\subsection{A few basic results related to matrices }\label{R123}
\begin{itemize}
    \item \textbf{(R0): Resolvent identity}: $A^{-1} - B^{-1} = A^{-1}(B - A)B^{-1} = B^{-1}(B - A)A^{-1}$ 
    \item \textbf{(R1)} For $p \times p$ skew-Hermitian matrices A and B, $||F^A - F^B||_{im} = ||F^{-\mathbbm{i}A} - F^{-\mathbbm{i}B}|| \leq \frac{1}{p}\operatorname{rank}(A - B)$

The first equality follows from (\ref{ESD_sksym_equal_sym}) and the last inequality follows from Lemma 2.4 of \cite{BaiSilv95}.    
    \item \textbf{(R2)} For a rectangular matrix, we have $\operatorname{rank}(A) \leq$ the number of non-zero entries of A. This is a result from Lemma 2.1 of \cite{BaiSilv95}.
    \item \textbf{(R3)} $\operatorname{rank}(AXB-PXQ) \leq \operatorname{rank}(A-P)+\operatorname{rank}(B-Q)$ where the dimensions of relevant matrices are compatible
    \item \textbf{(R4)} \textbf{Cauchy Schwarz Inequality}: $|a^*Xb| \leq ||X||_{op}|a|_2|b|_2$ where $|a|_2 = \sqrt{a^*a}$
    \item \textbf{(R5)}  For a p.d. matrix B and any square matrix A, $|\operatorname{trace}(AB)| \leq ||A||_{op} \operatorname{trace}(B)$. 
    
    To see this let $B = LL^*$ and let $L = [L_1: \ldots: L_n]$ and $B = (B_{ij})_{ij}$. Using (R4) of \ref{R123} we get, \begin{align*}
        &|\operatorname{trace}(AB)|
        = |\operatorname{trace}(ALL^*)|
        = |\operatorname{trace}(L^*AL)|
        = |\sum_{j=1}^n L_j^*AL_j|
        \leq \sum_{j=1}^n |L_j^*AL_j|\\
        \leq& \sum_{j=1}^n ||A||_{op}L_j^*L_j
        =\sum_{j=1}^n ||A||_{op}B_{jj} = ||A||_{op}\operatorname{trace}(B)
    \end{align*}
    \item (\textbf{(R6)}) For $A, B \in \mathbb{R}^{N \times N}$, $|\operatorname{trace}(AB)| \leq N||A||_{op}||B||_{op}$
\end{itemize}

\begin{lemma}\label{lA.1}
    Let $\{F_n, G_n\}_{n=1}^\infty$ be sequences of distribution functions on $\mathbbm{i}\mathbb{R}$ with $s_{F_n}(z), s_{G_n}(z)$ denoting their respective Stieltjes transforms at $z \in \mathbb{C}_L$. If $L_{im}(F_n, G_n) \rightarrow 0$, then $|s_{F_n}(z) - s_{G_n}(z)| \rightarrow 0$.
\end{lemma}
\begin{proof}
For a distribution function $F$ on $\mathbbm{i}\mathbb{R}$, we denote its real counterpart as $\Tilde{F}$. Let $\mathcal{P}(\mathbb{R})$ represent the set of all probability distribution functions on $\mathbb{R}$. Then the bounded Lipschitz metric is defined as follows 
$$\beta: \mathcal{P}(\mathbb{R}) \times \mathcal{P}(\mathbb{R}) \rightarrow \mathbb{R}_{+} \text{, where } \beta(\tilde{F}, \tilde{G}) := \underset{}{\sup} \bigg\{\bigg|\int hd\tilde{F} - \int hd\tilde{G}\bigg|: ||h||_{BL} \leq 1\bigg\}$$

$$\text{and }||h||_{BL} = \sup\{|h(x)|: x \in \mathbb{R}\} + \underset{x \neq y}{\sup}\dfrac{|h(x) - h(y)|}{|x - y|}
$$

From Corollary 18.4 and Theorem 8.3 of \cite{Dudley}, we have the following relationship between Levy (L) and bounded Lipschitz ($\beta$) metrics.
\begin{align*}\label{Levy_B}\tag{A.1}
\frac{1}{2}\beta(\tilde{F}, \tilde{G}) \leq L(\tilde{F}, \tilde{G}) \leq 3\sqrt{\beta(\tilde{F}, \tilde{G})}
\end{align*}

Fix $z \in \mathbb{C}_L$ arbitrarily. Define $g_z(x) := (\mathbbm{i}x - z)^{-1}$. First of all, $|g_z(x)| \leq 1/|\Re(z)|$ $\forall x \in \mathbb{R}$. Also, 
$$|g_z(x_1) - g_z(x_2)| = \bigg|\frac{1}{\mathbbm{i}x_1 - z} - \frac{1}{\mathbbm{i}x_2 - z}\bigg| = \frac{|x_1 - x_2|}{|\mathbbm{i}x_1 - z||\mathbbm{i}x_2-z|} \leq |x_1 - x_2|\frac{1}{\Re^2(z)}$$

Note that $||g_z||_{BL} \leq {1}/{|\Re(z)|} + {1}/{\Re ^2(z)} < \infty$. Then for $g := {g_z}/{||g_z||_{BL}}$, we have $||g||_{BL} = 1$.

By (\ref{Levy_B}) and (\ref{Levy_vs_uniform}), we have 
\begin{align*}
L_{im}(F_n, G_n) \rightarrow 0
\leftrightarrow &  L(\tilde{F}_n, \tilde{G}_n) \rightarrow 0 \\
\leftrightarrow  & \beta(\tilde{F}_n, \tilde{G}_n) \rightarrow 0\\
\implies & \bigg|\int_\mathbb{R}g(x)d\tilde{F}_n(x) - \int_\mathbb{R}g(x)d\tilde{G}_n(x)\bigg| \rightarrow 0\\
\implies & \bigg|\int_\mathbb{R}\frac{1}{\mathbbm{i}x - z}dF_n(x) - \int_\mathbb{R}\frac{1}{\mathbbm{i}x-z}dG_n(x)\bigg| \rightarrow 0\\
\implies & |s_{F_n}(z) - s_{G_n}(z)| \longrightarrow 0
\end{align*}
\end{proof}

We state the following result (Lemma B.26 of \cite{BaiSilv09}) without proof.
\begin{lemma}\label{lA.2}
Let $A = (a_{ij})$ be an $n \times n$ non-random matrix and $x = (x_1,\ldots, x_n)^T$ be a vector of independent entries. Suppose $\mathbb{E}x_i = 0, \mathbb{E}|x_i|^2 = 1$, and $\mathbb{E}|x_i|^l \leq \nu_l$. Then for $k \geq 1$, $\exists$ $C_k > 0$ independent of $n$ such that 
\begin{align*}
  \mathbb{E}|x^*Ax - \operatorname{trace}(A)|^k \leq C_k\bigg((\nu_4\operatorname{trace}(AA^*))^\frac{k}{2} + \nu_{2k}\operatorname{trace}\{(AA^*)^\frac{k}{2}\}\bigg)  
\end{align*}
\end{lemma}

\textbf{Simplification}: For deterministic matrix $A$ with $||A||_{op} < \infty$, let $B = \frac{A}{||A||_{op}}$. Then $||B||_{op} = 1$ and by (R6) of \ref{R123}, $\operatorname{trace}(BB^*) \leq n||B||_{op}^2 = n$ and $\operatorname{trace}\{(BB^*)^\frac{k}{2}\} \leq n||B||_{op}^k = n$. Then by Lemma \ref{lA.2}, we have

\begin{align*}
    &\mathbb{E}|x^*Bx - \operatorname{trace}(B)|^k 
    \leq C_k\bigg((\nu_4\operatorname{trace}(BB^*))^\frac{k}{2} + \nu_{2k}\operatorname{trace}\{(BB^*)^\frac{k}{2}\}\bigg)\\
    \implies & \frac{\mathbb{E}|x^*Ax - \operatorname{trace}(A)|^k }{||A||_{op}^k} \leq C_k[(n\nu_4)^\frac{k}{2} + n\nu_{2k}]\\
    \implies & \mathbb{E}|x^*Ax - \operatorname{trace}(A)|^k  \leq C_k||A||_{op}^k[(n\nu_4)^\frac{k}{2} + n\nu_{2k}]     \label{I1}\tag{A.2}
\end{align*}

We will be using this form of the inequality going forward.

\begin{lemma}\label{quadraticForm}
     Let $\{z_{jn}: 1 \leq j \leq n\}$ be a triangular array of complex valued random vectors. For $1 \leq r \leq n$, denote the $r^{th}$ element of $z_{jn}$ as $z_{jn}^{(r)}$. Suppose $\mathbb{E}z_{jn}^{(r)} = 0, \mathbb{E}|z_{jn}^{(r)}|^2 = 1$ and for $k \geq 1$, $\nu_k := \mathbb{E}|z_{jn}^{(r)}|^k \leq n^{ak}$ where $0 < a < \frac{1}{4}$. Suppose $A$ is independent of $z_{jn}$ and $||A||_{op} \leq B$ a.s. for some $B > 0$. Then, $$\underset{1\leq j \leq n}{\max} \absmod{\dfrac{1}{n}z_{jn}^*Az_{jn} - \dfrac{1}{n}\operatorname{trace}(A)} \xrightarrow{a.s} 0$$
\end{lemma}
\begin{proof}
    For arbitrary $\epsilon > 0$ and $k \geq 1$, we have
\begin{align*}
p_n:= &\mathbb{P}\bigg(\underset{1\leq j \leq n}{\max} \left|\dfrac{1}{n}z_{jn}^*Az_{jn} - \dfrac{1}{n}\operatorname{trace}(A)\right| > \epsilon\bigg)\\
 \leq& \sum_{j=1}^{n} \mathbb{P}\bigg(\left|\dfrac{1}{n}z_{jn}^*Az_{jn} - \dfrac{1}{n}\operatorname{trace}(A)\right| > \epsilon\bigg) \text{, by union bound}\\
\leq & \sum_{j=1}^{n} \dfrac{\mathbb{E}\left|\dfrac{1}{n}z_{jn}^*Az_{jn} - \dfrac{1}{n}\operatorname{trace}(A)\right|^{k}}{\epsilon^{k}} \text{, for any } k \in \mathbb{N}\\
= & \sum_{j=1}^{n} \dfrac{\mathbb{E} \bigg(\mathbb{E} \bigg[|\dfrac{1}{n}z_{jn}^*Az_{jn} - \dfrac{1}{n}\operatorname{trace}(A)|^k\bigg|A\bigg]\bigg)}{\epsilon^{k}}\\
\leq & \sum_{j=1}^{n} \dfrac{\mathbb{E}||A||_{op}^{k}C_{k}((n\nu_4)^\frac{k}{2} + n\nu_{2k})}{n^{k}\epsilon^{k}}  \text{ by } (\ref{I1})\\
\leq & \sum_{j=1}^{n} \dfrac{K[(n^{1+4a})^\frac{k}{2} + n^{1+2ak}]}{n^{k}} \text{, where } K = C_k\bigg(\frac{B}{\epsilon}\bigg)^k\\
=& \dfrac{nK[n^{\frac{k}{2}+2ak} + n^{1+2ak}]}{n^{k}}
= \frac{K}{n^{k(\frac{1}{2} - 2a) - 1}} + \frac{K}{n^{k(1 - 2a) - 2}}
\end{align*}

Since $0 < a < \frac{1}{4}$, we can choose $k \in \mathbb{N}$ large enough so that $\min\{k(\frac{1}{2} - 2a) - 1, k(1 - 2a) - 2\} > 1$ to ensure that $\sum_{n=1}^\infty p_n$ converges. Therefore by Borel Cantelli lemma we have the result.
\end{proof}

\begin{lemma}\label{lA.5}
    Let $\{X_{jn}, Y_{jn}: 1 \leq j \leq n\}_{n=1}^\infty$ be triangular arrays of random variables. Suppose $\underset{1 \leq j \leq n}{\max}|X_{jn}| \xrightarrow{a.s.} 0$ and $\underset{1 \leq j \leq n}{\max}|Y_{jn}| \xrightarrow{a.s.} 0$. Then $\underset{1 \leq j \leq n}{\max}|X_{jn} + Y_{jn}| \xrightarrow{a.s.} 0$.
\end{lemma}
\begin{proof}
Let $A_x := \{\omega: \underset{n \rightarrow \infty}{\lim}\underset{1 \leq j \leq n}{\max}|X_{jn}(\omega)| = 0\}$, $A_y := \{\omega: \underset{n \rightarrow \infty}{\lim}\underset{1 \leq j \leq n}{\max}|Y_{jn}(\omega)| = 0\}$. Then $\mathbb{P}(A_x) = 1 = \mathbb{P}(A_y)$. Then $\forall \omega \in A_x \cap A_y$, we have $0 \leq |X_{jn}(\omega) + Y_{jn}(\omega)| \leq |X_{jn}(\omega)| + |Y_{jn}(\omega)|$. Hence, 
$\underset{n \rightarrow \infty}{\lim}\underset{1 \leq j \leq n}{\max}|X_{jn}(\omega) + Y_{jn}(\omega)| = 0$. But, $\mathbb{P}(A_x \cap A_y) = 1$. Therefore, the result follows.
\end{proof}

\begin{lemma}\label{lA.6}
     Let $\{A_{jn}, B_{jn}, C_{jn}, D_{jn}: 1 \leq j \leq n\}_{n=1}^\infty$ be triangular arrays of random variables. Suppose $\underset{1 \leq j \leq n}{\max}|A_{jn} - C_{jn}| \xrightarrow{a.s.} 0$ and $\underset{1 \leq j \leq n}{\max}|B_{jn} - D_{jn}| \xrightarrow{a.s.} 0$ and $\exists N_0 \in \mathbb{N}$ such that $|C_{jn}| \leq B_1$ a.s. and $|D_{jn}| \leq B_2$ a.s. when $n > N_0$ for some $B_1, B_2 \geq 0$. Then $\underset{1 \leq j \leq n}{\max}|A_{jn}B_{jn} - C_{jn}D_{jn}| \xrightarrow{a.s.} 0$.
\end{lemma}
\begin{proof}

Let $\Omega_1 = \{\omega: \underset{n \rightarrow \infty}{\lim}\underset{1 \leq j \leq n}{\max}|A_{jn}(\omega) - C_{jn}(\omega)|= 0\}$, $\Omega_2 = \{\omega: \underset{n \rightarrow \infty}{\lim}\underset{1 \leq j \leq n}{\max}|B_{jn}(\omega) - D_{jn}(\omega)|= 0\}$, $\Omega_3 = \{\omega: |C_{jn}(\omega)| \leq B_1 \text{ for } n > N_0\}$ and $\Omega_4 = \{\omega: |D_{jn}(\omega)| \leq B_2 \text{ for } n > N_0\}$. Then $\Omega_0 = \cap_{j=1}^4\Omega_j$ is a set of probability $1$. Then $\forall \omega \in \Omega_0$, $\underset{1 \leq j \leq n}{\max} |B_{jn}(\omega)| \leq B_2$ eventually for large $n$. Therefore for  $\omega \in \Omega_0$ and large $n$,

\begin{align*}
\underset{1 \leq j \leq n}{\max}|A_{jn}B_{jn} - C_{jn}D_{jn}| 
\leq& \underset{1 \leq j \leq n}{\max}|A_{jn} - C_{jn}||B_{jn}| + \underset{1 \leq j \leq n}{\max}|C_{jn}||B_{jn} - D_{jn}|\\
\leq & B_2 \underset{1 \leq j \leq n}{\max}|A_{jn} - C_{jn}| + B_1 \underset{1 \leq j \leq n}{\max}|B_{jn} - D_{jn}| \xrightarrow{a.s.} 0   
\end{align*}
\end{proof}

\begin{lemma}\label{lA.7}
     Let $\{X_{jn}, Y_{jn}: 1 \leq j \leq n\}_{n=1}^\infty$ be triangular arrays of random variables such that $\underset{1 \leq j \leq n}{\max}|X_{jn} - Y_{jn}| \xrightarrow{a.s.} 0$. Then $|\frac{1}{n}\sum_{j=1}^n (X_{jn} - Y_{jn})| \xrightarrow{a.s.} 0$.
\end{lemma}
\begin{proof}
    Let $M_n := \underset{1 \leq j \leq n}{\max}|X_{jn} - Y_{jn}|$. We have $|\frac{1}{n}\sum_{j=1}^n (X_{jn} - Y_{jn})| \leq \frac{1}{n}\sum_{j=1}^n |X_{jn} - Y_{jn}| \leq M_n$. Let $\epsilon > 0$ be arbitrary. Then $\exists \Omega_0 \subset \Omega$ such that $\mathbb{P}(\Omega_0) = 1$ and $\forall \omega\in \Omega_0$, we have $M_n(\omega) < \epsilon$ for sufficiently large $n \in \mathbb{N}$. Then, $\mathbb{P}(\{\omega: |\frac{1}{n}\sum_{j=1}^n (X_{jn} - Y_{jn})| < \epsilon\}) = 1$. Since $\epsilon > 0$ is arbitrary, the result follows.
\end{proof}

\begin{lemma} \label{Rank2Woodbury}
Let $B \in \mathbb{C}^{p \times p}$ be of the form $B = A - zI_p$ for some skew-Hermitian matrix $A$ and $z \in \mathbb{C}_L$. For vectors $u, v \in \mathbb{C}^p$, define $\langle u, v\rangle := u^*B^{-1}v$. Then,
\begin{description}
    \item[1] $(B+ uv^{*} - vu^{*})^{-1}u = B^{-1}(\alpha_1 u + \beta_1 v)$; $\alpha_1 = (1 - \langle u,v\rangle)D(u,v)$; 
     $\beta_1 = \langle u,u\rangle D(u,v)$
    \item[2] $(B+ uv^{*} - vu^{*})^{-1}v = B^{-1}(\alpha_2 v + \beta_2 u)$; $\alpha_2 = (1 + \langle v,u\rangle) D(u,v)$; $\beta_2 = -\langle v, v\rangle D(u,v)$
\end{description}
 where $D(u,v) =  \bigg((1 - \langle u,v\rangle)(1 + \langle v,u\rangle) +  \langle u,u\rangle\langle v,v\rangle\bigg)^{-1}$
\end{lemma}
\begin{proof}

Clearly, $B$ cannot have zero as eigenvalue. So $\langle u,v \rangle$ is well defined. For $P \in \mathbb{C}^{p\times p}, Q,R \in  \mathbb{C}^{p\times n}$ with $P + QR^*$ and $P$ being invertible, the Woodbury formula states that
\begin{align*}\label{ShermanMorrison}\tag{A.4}
  (P + QR^*)^{-1} =& P^{-1} - P^{-1}Q(I_n + R^*P^{-1}Q)^{-1}R^*P^{-1}\\
  \implies (P + QR^*)^{-1}Q =& P^{-1}Q - P^{-1}Q(I_n + R^*P^{-1}Q)^{-1}R^*P^{-1}Q\\
  =& P^{-1}Q\bigg(I_n - (I_n + R^*P^{-1}Q)^{-1}R^*P^{-1}Q\bigg)
\end{align*}

Let $P = B$, $Q=[u:v]$ and $R=[v: -u]$. Note that $\operatorname{det}(I_2 + R^*P^{-1}Q)^{-1} = D(u,v)$. So, $D(u,v)$ is well-defined. Finally, observing that $B + uv^*-vu^* = P + QR^*$, we use (\ref{ShermanMorrison}) to get
\begin{align*}
    &(B + uv^*-vu^*)^{-1}[u:v]\\
    =& B^{-1}[u:v]\bigg(I_2 - (I_2 + R^*P^{-1}Q)^{-1}R^*P^{-1}Q\bigg)\\
    =& B^{-1}[u:v] \bigg(I_2 -{\begin{bmatrix}
        1+\langle v,u \rangle & \langle v,v \rangle\\
        -\langle u,u \rangle & 1-\langle u,v \rangle
    \end{bmatrix}}^{-1}\begin{bmatrix}
        \langle v,u \rangle &  \langle v,v \rangle\\
        -\langle u,u \rangle &  -\langle u,v \rangle\\        
    \end{bmatrix}\bigg)\\
    =&  B^{-1}[u:v] \bigg(I_2 - D(u,v)\begin{bmatrix}
        1-\langle u,v \rangle & -\langle v,v \rangle\\
        \langle u,u \rangle & 1+\langle v,u \rangle
    \end{bmatrix}\begin{bmatrix}
        \langle v,u \rangle &  \langle v,v \rangle\\
        -\langle u,u \rangle &  -\langle u,v \rangle\\        
    \end{bmatrix}\bigg)\\
    =& D(u,v)B^{-1}[u:v] \begin{bmatrix}
        \alpha_1 & \beta_2\\
        \beta_1 & \alpha_2
    \end{bmatrix}
\end{align*}

\end{proof}

\section{Proofs related to Section \ref{sec:general_covariance}}

\subsection{A few preliminary results}
Here we establish a few results that are required to prove the main theorems of Section \ref{sec:general_covariance}. All of them (except Lemma \ref{tightness}) are proved under Assumptions \ref{A123}. These results directly lead to the proofs of Theorem \ref{DeterministicEquivalent} and Theorem \ref{ExistenceA123}.

\begin{lemma}\label{tightness}
    Under the conditions of Theorem \ref{t3.1}, if we instead had $H = \delta_0$, we have $F^{S_n} \xrightarrow{d} \delta_0$ a.s. When $H\neq \delta_0$, $\{F^{S_n}\}_{n=1}^\infty$ is a tight sequence.
\end{lemma}
\begin{proof}

Let $A := \frac{1}{\sqrt{n}}[Z1:Z2]$ and $B := \frac{1}{\sqrt{n}}[Z_2: -Z_1]^*$. Then, $S_n = \Sigma_n^\frac{1}{2}AB\Sigma_n^\frac{1}{2}$. Now define $M_n = AB\Sigma_n$. Note that $S_n$ and $M_n$ share the same set of eigenvalues and
$$Z_{0n} := \frac{1}{n}(Z_1Z_1^* + Z_2Z_2^*) = AA^* = B^*B$$

In the equations to follow, we highlight the fact that the support of $F^{S_n}$ (and hence of $F^{M_n}$) is purely imaginary. However, the supports of $F^{\sqrt{M_nM_n^*}}$, $F^{Z_{0n}}$ and $F^{\Sigma_n}$ are purely real.

For arbitrary $K_1, K_2, K_3 > 0$, let $K = K_1K_2K_3$. Using Lemma 2.3 of \cite{BaiSilv95},
\begin{align*}
  F^{S_n}\{(-\infty, -\mathbbm{i}K)\cup (\mathbbm{i}K, \infty)\}
= & F^{M_n}\{(-\infty, -\mathbbm{i}K) \cup (\mathbbm{i}K, \infty)\}\\
= & F^{\sqrt{M_nM_n^*}}\{(K, \infty)\}\\
\leq & F^{\sqrt{AA^*}}\{(K_1, \infty)\} + F^{\sqrt{BB^*}}\{(K_2, \infty)\} + F^{\sqrt{\Sigma_n^2}}\{(K_3, \infty)\}\\
=& F^{Z_{0n}}\{(K_1^2, \infty)\} + F^{Z_{0n}}\{(K_2^2, \infty)\} + F^{{\Sigma_n}}\{(K_3, \infty)\} \tag{B.1}\label{B.1}
\end{align*}
In the second term of the last equality, we used the fact that $BB^*$ and $B^*B$ share the same set of eigenvalues.

Now we prove the first result. Suppose $F^{\Sigma_n} \xrightarrow{d} H=\delta_0$ a.s. Choose $\epsilon, K_1, K_2 > 0$ arbitrarily and set $K_3 = {\epsilon}/{K_1K_2}$. Since $\{F^\Sigma_n\}_{n=1}^\infty$ converges weakly to $\delta_0$,
\begin{align*}
  \underset{n \rightarrow \infty}{\limsup}F^{{\Sigma_n}}\bigg\{(\dfrac{\epsilon}{K_1K_2}, \infty)\bigg\} = 0  
\end{align*}

Note that $\{F^{Z_{0n}}\}_{n=1}^\infty$ is a tight sequence. Now letting $K_1, K_2 \rightarrow \infty$ in (\ref{B.1}), we see that
$$\underset{n \rightarrow \infty}{\limsup}F^{S_n}\{(-\infty, -\epsilon)\cup (\epsilon, \infty)\} \leq \underset{K_1 \rightarrow \infty}{\lim}F^{Z_{0n}}\{(K_1^2, \infty)\} + \underset{K_2 \rightarrow \infty}{\lim}F^{Z_{0n}}\{(K_2^2, \infty)\} = 0$$ Since $\epsilon > 0$ was chosen arbitrarily, we conclude that $F^{S_n} \xrightarrow{d} \delta_0$ a.s. This justifies why we exclusively stick to the case where $H \neq \delta_0$ in Theorem \ref{t3.1}.

Now suppose $F^{\Sigma_n} \xrightarrow{d} H \neq \delta_0$ a.s. The tightness of $\{F^{S_n}\}_{n=1}^\infty$ is immediate from (\ref{B.1}) upon observing the tightness of $\{F^{Z_{0n}}\}_{n=1}^\infty$ and $\{F^{\Sigma_n}\}_{n=1}^\infty$.
\end{proof}

\begin{lemma}\label{Rank2Perturbation}
Let $M_n \in \mathbb{C}^{p \times p}$ be a sequence of deterministic matrices with bounded operator norm, i.e. $||M_n||_{op} \leq B$ for some $B \geq 0$. Under Assumptions \ref{A123}, for $z \in \mathbb{C}_L$ and sufficiently large $n$, we have $$\underset{1 \leq j \leq n}{\max}|\operatorname{trace}\{M_nQ(z)\} - \operatorname{trace}\{M_nQ_{-j}(z)\}| \leq \dfrac{4cC B}{\Re^2 (z)} \text{ a.s.}$$
Consequently, $\underset{1 \leq j \leq n}{\max}|\frac{1}{p}\operatorname{trace}\{M_n(Q - Q_{-j})\}| \xrightarrow{a.s.} 0$
\end{lemma}

\begin{proof}
By (R0) and (R4) of (\ref{R123}), for any $1\leq j \leq n$,
\begin{align*}
        & |\operatorname{trace}\{M_nQ\} - \operatorname{trace}\{M_nQ_{-j}\}|\\
        =&|\operatorname{trace}\{M_n(S_n - zI_p)^{-1}\} - \operatorname{trace}\{M_n(S_{nj}-zI_p)^{-1}\}|\\
        =& |\operatorname{trace}\{M_nQ(\dfrac{1}{n}X_{1j}X_{2j}^* - \dfrac{1}{n}X_{2j}X_{1j}^*)Q_{-j}\}|\\
        =&\frac{1}{n}|X_{2j}^*Q_{-j}M_nQ X_{1j} - X_{1j}^*Q_{-j} M_nQ X_{2j}|\\
        \leq& \dfrac{1}{n}|X_{2j}^*Q_{-j}MQ X_{1j}| +\dfrac{1}{n}|X_{1j}^*Q_{-j} MQ X_{2j}|\\
        \leq& ||Q_{-j}M_nQ||_{op}\bigg(\sqrt{\frac{1}{n}X_{2j}^*X_{2j}}\sqrt{\frac{1}{n}X_{1j}^*X_{1j}} + \sqrt{\frac{1}{n}X_{1j}^*X_{1j}}\sqrt{\frac{1}{n}X_{2j}^*X_{2j}}\bigg)  \label{B.2}\tag{B.2}
\end{align*}

First of all, we have
$$||Q_{-j}M_nQ||_{op} < B/\Re ^2(z) \text{ since } ||Q_{-j}||_{op},||Q||_{op} \leq 1/|\Re(z)|, ||M_n||_{op} < B$$

Secondly, for $k \in \{1,2\}$, we have $X_{kj}^*X_{kj} = Z_{kj}^*\Sigma_nZ_{kj}$ where $\Sigma_n$ satisfies \textbf{A1} and $Z_1, Z_2$ satisfy \textbf{A2} respectively of Assumptions \ref{A123}. Let $z_{jn} = Z_{kj}, 1 \leq j \leq n, n \in \mathbb{N}$ and $A = \Sigma_n$. Then by Lemma \ref{quadraticForm}, we have 
$$\underset{1 \leq j \leq n}{\max}\bigg|\frac{1}{n}X_{kj}^*X_{kj} - \frac{1}{n}\operatorname{trace}(\Sigma_n)\bigg| \xrightarrow{a.s.} 0.$$

From (T1) and (T4) of Theorem \ref{t3.1}, 
$$\dfrac{1}{n}\operatorname{trace}(\Sigma_n) = c_n\bigg(\dfrac{1}{p}\operatorname{trace}(\Sigma_n)\bigg) < 2cC$$ for sufficiently large $n$. This implies that for $k \in \{1,2\}$ and large $n$,
$$\underset{1 \leq j \leq n}{\max}\bigg|\frac{1}{n}X_{kj}^*X_{kj}\bigg| < 2cC \text{ a.s.}$$

Combining everything with (\ref{B.2}), for large $n$, we must have
\begin{align*}
    \underset{1 \leq j \leq n}{\max}|\operatorname{trace}\{M_nQ\} - \operatorname{trace}\{M_nQ_{-j}\}|
    \leq \frac{B}{\Re^2(z)}(2cC + 2cC) = \frac{4cC B}{\Re ^2(z)} \text{ a.s.}
\end{align*}

For $z \in \mathbb{C}_L$, it is clear that for arbitrary $\epsilon > 0$,
$\underset{1 \leq j \leq n}{\max}|\frac{1}{p}\operatorname{trace}\{M(Q - Q_{-j})\}| < \epsilon$ a.s. for large $n$. Therefore, $\underset{1 \leq j \leq n}{\max}|\frac{1}{p}\operatorname{trace}\{M(Q - Q_{-j})\}| \xrightarrow{a.s.} 0$.

\end{proof}

\begin{lemma}\label{ConcentrationOfSnH}
\textbf{Concentration of Stieltjes Transforms}
    Under Assumptions \ref{A123} for $z \in \mathbb{C}_L$, we have $|s_n(z) - \mathbb{E}s_n(z)| \xrightarrow{a.s.} 0$ and $|h_n(z) - \mathbb{E}h_n(z)| \xrightarrow{a.s.} 0$ .
\end{lemma}
\begin{proof}
Define $\mathcal{F}_k = \sigma(\{X_{1j}, X_{2j}:k+1 \leq j \leq n\})$ and for a measurable function $f$, we denote $\mathbb{E}_{k}f(X) := \mathbb{E}(f(X)|\mathcal{F}_k)$ for $0 \leq k \leq n-1$ and $\mathbb{E}_{n}f(X) := \mathbb{E}f(X)$. Then, we observe that

\begin{align*}
    h_n(z) - \mathbb{E}h_n(z)
    =& \dfrac{1}{p}\operatorname{trace}\{\Sigma_nQ\} - \mathbb{E}\bigg(\dfrac{1}{p}\operatorname{trace}\{\Sigma_nQ\}\bigg)\\
    =& \dfrac{1}{p}\sum_{k=1}^{n}(\mathbb{E}_k - \mathbb{E}_{k-1})\operatorname{trace}\{\Sigma_nQ\}\\    
    =& \dfrac{1}{p}\sum_{k=1}^{n}(\mathbb{E}_k - \mathbb{E}_{k-1})\underbrace{(\operatorname{trace}\{\Sigma_nQ\} - \operatorname{trace}\{\Sigma_nQ_{-k}\})}_{:= Y_k}\\
    =& \dfrac{1}{p}\sum_{k=1}^{n}\underbrace{(\mathbb{E}_k - \mathbb{E}_{k-1}) Y_k}_{:= D_k}
    = \frac{1}{p}\sum_{k=1}^{n}D_k
\end{align*}
From (\ref{B.2}), we have

$$|Y_k| = |\operatorname{trace}\{\Sigma_nQ\} - \operatorname{trace}\{\Sigma_nQ_j\}| \leq \dfrac{\tau}{\Re^2(z)}W_{nk} \text{, where } W_{nk} := \frac{1}{n}(||X_{1k}||^2 + ||X_{2k}||^2)$$
So, we have  $\absmod{D_k} \leq \dfrac{2\tau}{\Re^2(z)}W_{nk}$. By Lemma 2.12 of \cite{BaiSilv09}, there exists $K_4$ depending only on $z \in \mathbb{C}_L$ such that
\begin{align*}\label{concInequality}\tag{B.3}
   \mathbb{E}\absmod{h_n(z) - \mathbb{E}h_n(z)}^4 = \mathbb{E}\absmod{\frac{1}{p}\sum_{k=1}^n D_k}^4 \leq& \frac{K_4}{p^4} \mathbb{E}\bigg(\sum_{k=1}^n |D_k|^2\bigg)^2 \leq \frac{16K_4\tau^4}{p^4\Re^8(z)} \mathbb{E}\bigg(\sum_{k=1}^n |W_{nk}|^2\bigg)^2 \\
   =& \frac{K_0}{p^4}\mathbb{E}\bigg(\sum_{k=1}^n|W_{nk}|^4+\sum_{k\neq l}|W_{nk}|^2|W_{nl}|^2\bigg)\\
   =& \frac{K_0}{p^4}\bigg(\sum_{k=1}^n\mathbb{E}|W_{nk}|^4+\sum_{k\neq l}\mathbb{E}|W_{nk}|^2\mathbb{E}|W_{nl}|^2\bigg)
\end{align*}

We have the following inequalities.
\begin{enumerate}
    \item $||X_{1k}||^m = (Z_{1k}^*\Sigma_nZ_{1k})^\frac{m}{2} \leq (||\Sigma_n||_{op}||Z_{1k}||^2)^\frac{m}{2} \leq \tau^\frac{r}{2}||Z_{1k}||^m$ for $m \geq 1$
    \item $|W_{nk}|^2 \leq  \dfrac{2}{n^2}(||X_{1k}||^4 + ||X_{2k}||^4) \leq \dfrac{2\tau^2}{n^2}(||Z_{1k}||^4 + ||Z_{2k}||^4)$
    \item $|W_{nk}|^4 \leq  \dfrac{8}{n^4}(||X_{1k}||^8 + ||X_{2k}||^8) \leq \dfrac{8\tau^4}{n^4}(||Z_{1k}||^8 + ||Z_{2k}||^8)$
\end{enumerate}

Recall that $Z_{1k}$ is the $k^{th}$ column of $Z_1$ and $z_{jk}^{(1)}$ represents the $j^{th}$ element of $Z_{1k}$. Let $M_4>0$ be such that for $r=1,2$, $\mathbb{E}|z_{ij}^{(r)}|^4 \leq M_4$. This exists since the entries of $Z_1, Z_2$ have uniform bound on moments of order $4+\eta_0$. So, we have

\begin{align*}
    \mathbb{E}||Z_{1k}||^4 = \mathbb{E} \bigg(\sum_{j=1}^p|z_{jk}^{(1)}|^2\bigg)^2 = \mathbb{E} \bigg(\sum_{j=1}^p|z_{jk}^{(1)}|^4 + \sum_{j \neq l}|z_{jk}^{(1)}|^2|z_{lk}^{(1)}|^2\bigg) \leq pM_4 + p(p-1)
\end{align*}
and, by Assumptions \ref{A123}, $\mathbb{E}|z_{jk}^{(1)}|^m \leq n^{ma}$, for $m \geq 1$, with $0 < a < 1/4$, 
\begin{align*}
    \mathbb{E}||Z_{1k}||^8 = \mathbb{E} \bigg(\sum_{j=1}^p|z_{jk}^{(1)}|^2\bigg)^4 =& \mathbb{E} \bigg(\sum_{j=1}^p|z_{jk}^{(1)}|^8 + \sum_{j \neq l}|z_{jk}^{(1)}|^6|z_{lk}^{(1)}|^2 + \sum_{j \neq l}|z_{jk}^{(1)}|^4|z_{lk}^{(1)}|^4 \bigg)\\ \leq& p \times
    (M_4 n^{12a})^{1/2} + p(p-1) \times (M_4 n^{8a})^{1/2} \times 1 + p(p-1) \times M_4^2\\
    \leq& p(M_4 n^{12a})^{1/2} + p^2(M_4 n^{8a})^{1/2} + p^2M_4^2
\end{align*}
where the bounds in the second last line follows from $\mathbb{E}|z_{jk}^{(1)}|^8 \leq (\mathbb{E}|z_{jk}^{(1)}|^4 \mathbb{E}|z_{jk}^{(1)}|^{12})^{1/2}$, and 
$\mathbb{E}|z_{jk}^{(1)}|^6 \leq (\mathbb{E}|z_{jk}^{(1)}|^4 \mathbb{E}|z_{jk}^{(1)}|^8)^{1/2}$,
by Cauchy-Schwarz inequality.
Therefore, combining everything, we get
\begin{align*}
    \mathbb{E}|W_{nk}|^2 \leq& \frac{4\tau^2}{n^2}\bigg(pM_4 + p(p-1)\bigg) \leq 4\tau^2c_n^2(M_4/p +1) \text{, recall } c_n = p/n\\
    \mathbb{E}|W_{nk}|^4 \leq& \frac{16\tau^4}{n^4}\bigg(c_nM_4^{1/2} n^{1+6a}+ c_n^2M_4^{1/2} n^{2+4a}+p^2M_4^2\bigg) = O\bigg(\max\{\frac{1}{n^{3-6a}}, \frac{1}{n^{2-4a}}\}\bigg)
\end{align*}
Since $0 < a < 1/4$, we have $\min\{2-4a,3-6a\} = 2-4a > 1$. Using these in (\ref{concInequality}) and by Borel Cantelli Lemma, we have $|h_n(z) - \mathbb{E}h_n(z)| \xrightarrow{a.s.} 0$. The other result follows similarly.
\end{proof}

\begin{lemma}\label{boundedAwayFromZero}
    Recall definitions of $h_n(z)$ (\ref{defining_hn}) and $v_n(z)$ (\ref{defining_vn}). Under \textbf{A1} of Assumptions \ref{A123}, for $z \in \mathbb{C}_L$ and sufficiently large $n$,  $\Re(c_nh_n(z)) \geq K_0$ a.s., $\Re(c_n\mathbb{E}h_n(z)) \geq K_0$ and $|v_n(z)| \leq 1/K_0$ a.s. where $K_0 > 0$ depends on $z, c, \tau$ and $H$.
\end{lemma}
\begin{proof}

We have $||\Sigma_n||_{op} \leq \tau$. Since $F^{\Sigma_n}$ and $H$ have a compact support $[0, \tau]$ and $F^{\Sigma_n} \xrightarrow{d} H$ a.s., we get
$\int_0^\tau\lambda dF^{\Sigma_n}(\lambda) \xrightarrow{} \int_0^\tau\lambda dH(\lambda) > 0$ since $H \neq \delta_0$. Therefore,
\begin{align*}\label{B.4}\tag{B.4}
\frac{1}{n}\operatorname{trace}(\Sigma_n) = c_n\int_0^\tau\lambda dF^{\Sigma_n}(\lambda) \xrightarrow{} c\int_0^\tau\lambda dH(\lambda) > 0    
\end{align*}

Let $z = -u + \mathbbm{i}v$ with $u > 0$. Let $a_{ij}$ represent the $ij^{th}$ element of $A:= P^*\Sigma_nP$ where $S_n = P\Lambda P^*$ with $\Lambda = \operatorname{diag}(\{\mathbbm{i}\lambda_j\}_{j=1}^p)$ being a diagonal matrix containing the purely imaginary eigenvalues of $S_n$. Then,
    \begin{align*}
        c_nh_n = \frac{1}{n}\operatorname{trace}\{\Sigma_nQ\}
        = \frac{1}{n}\operatorname{trace}\{P^*\Sigma_nP(\Lambda - zI_p)^{-1}\}
         =\frac{1}{n}\sum_{j=1}^p \frac{a_{jj}}{\mathbbm{i}\lambda_j  - z}
    \end{align*}

For any $\delta > 0$, we have
$$||S_n||_{op} = ||\frac{1}{n}X_1X_2^* - \frac{1}{n}X_2X_1^*||_{op} \leq 2\sqrt{||\frac{1}{n}X_1X_1^*||_{op}}\sqrt{||\frac{1}{n}X_2X_2^*||_{op}} \leq 2||\Sigma_n||_{op}\bigg(1+\sqrt{\frac{p}{n}}\bigg)^2 + \delta$$

Let $B = 4\tau(1+\sqrt{c})^2$. Then $\mathbb{P}(|\lambda_j| > B \hspace{2mm} i.o.) = 0$.

\footnote{$\operatorname{sgn}(x)$ is the sign function}
Define $B^* :=\left\{\begin{matrix}
    -B\operatorname{sgn}(v) & \text{ if } v \neq 0 \\
    B &  \text{            if } v = 0
\end{matrix} \right.$

Then $(\lambda_j - v)^2 \leq (B^* - v)^2$. Therefore,
\begin{align*}
    \Re(c_nh_n) =& \frac{1}{n}\sum_{j=1}^p\frac{a_{jj}u}{(\lambda_j - v)^2 + u^2}\\
    \geq & \frac{1}{n}\sum_{j=1}^p \frac{a_{jj}u}{(B^* - v)^2 + u^2}\\
    =& \frac{u}{(B^* - v)^2 + u^2} \bigg(\frac{1}{n}\sum_{j=1}^p a_{jj}\bigg)\\
    =& \frac{u}{(B^* - v)^2 + u^2} \bigg(\frac{1}{n}\operatorname{trace}(\Sigma_n)\bigg) \text{, as } \operatorname{trace}(A) = \operatorname{trace}(\Sigma_n)\\
    \xrightarrow{} & \frac{u}{(B^* - v)^2 + u^2}\bigg(c\int_0^\tau\lambda dH(\lambda)\bigg) := K_0
    > 0 \text{ from } (\ref{B.4})
\end{align*} 

Therefore for sufficiently large $n$, $\Re(c_nh_n) > 0$ a.s. In conjunction with Lemma \ref{ConcentrationOfSnH}, we also get $\Re(c_n\mathbb{E}h_n) \geq K_0 > 0$ for large $n$. Moreover, for large $n$, $c_nh_n \neq \pm \mathbbm{i}$ a.s. and hence, the quantity $v_n$ is well defined almost surely. Noting that for $z \in \mathbb{C}$ with $\Re(z) \neq 0$,
\begin{align*}
    &\left|\frac{1}{1 + z^2}\right|
    = \left|\frac{1}{2\mathbbm{i}}\frac{\mathbbm{i} + z + \mathbbm{i} - z}{1 + z^2}\right|
    \leq \frac{1}{2}\bigg(\frac{1}{|\mathbbm{i}-z|} + \frac{1}{|\mathbbm{i}+z|}\bigg)
    \leq \frac{1}{2}\bigg(\frac{1}{|\Re(z)|} + \frac{1}{|\Re(z)|}\bigg) = \frac{1}{|\Re(z)|}
\end{align*}

we therefore conclude that for sufficiently large $n$, we must have 
$$|v_n| = \left|\frac{1}{1 + (c_nh_n)^2}\right| \leq \frac{1}{\Re(c_nh_n)} \leq \frac{1}{K_0} \text{ a.s. }$$
\end{proof}

\begin{lemma}\label{rhoContinuity}
Suppose $\{X_n, Y_n\}_{n=1}^\infty$ are complex random variables with $\Re(X_n), \Re(Y_n) \geq B$ for some $B > 0$ and $|X_n - Y_n| \xrightarrow{a.s.} 0$. Then, $|\rho(X_n) - \rho(Y_n)| \xrightarrow{a.s.} 0$.
\end{lemma}
\begin{proof}
The result is clear from the below string of inequalities.
\begin{align*}
    &|\rho(X_n) - \rho(Y_n)| \leq \bigg|\frac{1}{\mathbbm{i}+X_n} - \frac{1}{\mathbbm{i}+Y_n}\bigg| + \bigg|\frac{1}{\mathbbm{-i}+X_n}-\frac{1}{\mathbbm{-i}+Y_n}\bigg| \leq \frac{2|X_n -Y_n|}{\Re(X_n)\Re(Y_n)} \leq \frac{2|X_n-Y_n|}{B^2}
\end{align*}

We mention a two direct implications of this result which will be used later on. Under Assumptions \ref{A123},
$|h_n - \mathbbm{E}h_n| \xrightarrow{a.s.} 0$ by Lemma \ref{ConcentrationOfSnH} and by Lemma \ref{boundedAwayFromZero}, for sufficiently large $n$, we have $\Re(c_nh_n) \geq K_0 > 0$ where $K_0$ depends on $c, z, \tau$ and $H$. Hence, $\Re(\mathbb{E}c_nh_n) \geq K_0 > 0$ for large $n$. Therefore we have 

\begin{align} \label{rhoContinuity1}\tag{B.5}
    |\rho(c_nh_n) - \rho(c_n\mathbb{E}h_n)| \xrightarrow{a.s.} 0
\end{align}

Moreover, by Theorem \ref{DeterministicEquivalent}, we have $|h_n - \Tilde{h}_n| \xrightarrow{a.s.} 0$. Therefore, we have $\Re(c_n\Tilde{h}_n) \geq K_0 > 0$ as well for large $n$. Thus, we also get
\begin{align} \label{rhoContinuity2}\tag{B.6}
    |\rho(c_n\Tilde{h}_n) - \rho(c_n\mathbb{E}h_n)| \xrightarrow{} 0
\end{align}
\end{proof}

\begin{lemma}\label{normBound}
    Under Assumptions \ref{A123}, the operator norms of the matrices $\Bar{Q}(z), \Bar{\Bar{Q}}(z)$ defined in Theorem \ref{DeterministicEquivalent} and (\ref{definingQBarBar}) respectively are bounded by $1/|\Re(z)|$ for $z \in \mathbb{C}_L$.
\end{lemma}
\begin{proof}
    Let $\Sigma_n = P\Lambda P^*$ where $\Lambda = \operatorname{diag}(\{\lambda_j\}_{j=1}^p)$ with $\lambda_j > 0$. Then, 
\begin{align*}
    \Bar{Q}(z) = (-zPP^* + \rho(c_n\mathbb{E}h_n)P\Lambda P^*)^{-1}
    = P(-zI_p + \rho(c_n\mathbb{E}h_n)\Lambda)^{-1}P^*
\end{align*}

Now note that for any $1 \leq j \leq p$, we have 
\begin{align*}
\bigg|\dfrac{1}{-z + \lambda_j\rho(c_n\mathbb{E}h_n)}\bigg|
\leq & \frac{1}{|\Re(-z + \lambda_j\rho(c_n\mathbb{E}h_n))|}
\leq \dfrac{1}{|\Re(z)|} \text{, using } (\ref{realOf_Qbar_EV_inverse})
\end{align*}
This proves that $||\Bar{Q}(z)||_{op} \leq 1/|\Re(z)|$. For the other result note that,
\begin{align*}
    \Bar{\Bar{Q}}(z) =& (-zPP^* + \rho(c_n\Tilde{h}_n)P\Lambda P^*)^{-1}
    = P(-zI_p + \rho(c_n\Tilde{h}_n)\Lambda)^{-1}P^*
\end{align*}
We have $|\rho(c_n\Tilde{h}_n) - \rho(c_n\mathbb{E}h_n)| \rightarrow 0$ from Lemma \ref{rhoContinuity} and by Lemma \ref{boundedAwayFromZero}, for sufficiently large $n$, $\Re(\rho(c_n\mathbb{E}h_n)) > 0$. Therefore, we have $||\Bar{\Bar{Q}}(z)||_{op} \leq {1}/{|\Re(z)|}$.
\end{proof}

\begin{lemma}\label{hn_tilde_tilde2}
Under Assumptions \ref{A123} and $z \in \mathbb{C}_L$, we have $|\Tilde{h}_n(z) - \Tilde{\Tilde{h}}_n(z)| \rightarrow 0$    
\end{lemma}
\begin{proof}
From definitions (\ref{definingHnTilde}), (\ref{definingHnTilde2})
\begin{align*}
    |\Tilde{h}_n(z) - \Tilde{\Tilde{h}}_n(z)|
    =& \frac{1}{p}|\operatorname{trace}\{\Sigma_n (\Bar{Q} - \Bar{\Bar{Q}})\}|\\
    =& \frac{1}{p}|\operatorname{trace}\{\Sigma_n \Bar{Q} [\rho(c_n\Tilde{h}_n) - \rho(c_n\mathbb{E}h_n)]\Sigma_n \Bar{\Bar{Q}}\}| \text{, by resolvent identity (\ref{R123})}\\
    \leq& \bigg(\frac{1}{p}\operatorname{trace}(\Sigma_n)\bigg) |\rho(c_n\Tilde{h}_n) - \rho(c_n\mathbb{E}h_n)| \times ||\Bar{Q}\Sigma_n \Bar{\Bar{Q}}||_{op} \text{, by (R5) of } (\ref{R123})\\
    \leq & C |\rho(c_n\Tilde{h}_n) - \rho(c_n\mathbb{E}h_n))| \frac{\tau}{\Re^2 (z)} \text{, for large } n \text{ and using Lemma \ref{normBound}}\\ 
    =& |\rho(c_n\Tilde{h}_n) - \rho(c_n\mathbb{E}h_n)| \frac{C\tau}{\Re^2 (z)}
\end{align*}

Now, by (\ref{rhoContinuity2}), we finally have $|\Tilde{h}_n(z) - \Tilde{\Tilde{h}}_n(z)| \xrightarrow{a.s.}0$.
\end{proof}

\begin{lemma}\label{uniformConvergence}
Under Assumptions \ref{A123}, the quantities $c_{1j},c_{2j}, d_{1j}, d_{2j}, F_j(r,s), r,s \in \{1,2\}, v_n, m_n$ as defined throughout the proof of Theorem \ref{DeterministicEquivalent} satisfy the following results.
\begin{align*}
    &\underset{1\leq j\leq n}{\max }|c_{1j} - v_n| \xrightarrow{a.s.} 0 \hspace{18mm} \underset{1\leq j\leq n}{\max }|d_{1j} - v_n|\xrightarrow{a.s.} 0\\
    &\underset{1\leq j\leq n}{\max }|c_{2j} - c_nh_nv_n|\xrightarrow{a.s.} 0 \hspace{10mm} \underset{1\leq j\leq n}{\max }|d_{2j} - c_nh_nv_n|\xrightarrow{a.s.} 0\\
    & \underset{1\leq j\leq n}{\max }|F_j(r,r) - m_n| \xrightarrow{a.s.} 0, r \in \{1,2\}\\
    & \underset{1\leq j\leq n}{\max }|F_j(r,s)| \xrightarrow{a.s.} 0  \text{ where } r \neq s, r,s \in \{1,2\} 
\end{align*}  
\end{lemma}
\begin{proof}
For $n \in \mathbb{N}$ and $1 \leq j \leq n$, define 
\begin{align*}\label{defining_Aj}\tag{B.7}
A_j = \Sigma_n^\frac{1}{2}Q_{-j}\Sigma_n^\frac{1}{2}    
\end{align*}

Recall the definition of $E_j(r,s)$ (\ref{definingEjrs}). For $r \in \{1,2\}$, let $z_{jn} = Z_{rj}, 1 \leq j \leq n$ . We have $||A_j||_{op} \leq {\tau}/{|\Re(z)|}$. Then $\{z_{jn}: 1 \leq j \leq n\}_{n=1}^\infty$ and $A_j$ satisfy the conditions of Lemma \ref{quadraticForm}. Thus we have 
\begin{align*}\label{B.8}\tag{B.8}
 \underset{1\leq j \leq n}{\max} |E_j(r,r) - \dfrac{1}{n}\operatorname{trace}(A_j)| = \underset{1\leq j \leq n}{\max} \left|E_j(r,r) - \dfrac{1}{n}\operatorname{trace}\{\Sigma_n Q_{-j}\}\right| \xrightarrow{a.s.} 0   
\end{align*}

From Lemma \ref{Rank2Perturbation}, $|\frac{1}{n}\operatorname{trace}\{\Sigma_n(Q- Q_{-j})\}| \xrightarrow{a.s.} 0$. Observing that $c_nh_n = \dfrac{1}{n}\operatorname{trace}\{\Sigma_nQ\}$ we get
\begin{align*}\label{E_j_rr}\tag{B.9}
    \underset{1\leq j \leq n}{\max} \left|E_j(r,r) - c_nh_n\right| \xrightarrow{a.s.} 0
\end{align*}

With $\{z_{jn} = \frac{1}{\sqrt{2}}(Z_{1j} + Z_{2j}): 1 \leq j \leq n, n \in \mathbb{N}\}$ and $A_j$ from (\ref{defining_Aj}), Lemma \ref{quadraticForm} and Lemma \ref{Rank2Perturbation} gives us 
\begin{align*}\label{B.10}\tag{B.10}
    &\underset{1\leq j \leq n}{\max} \left|\frac{1}{2}(E_j(1,1) - c_nh_n) + \frac{1}{2}(E_j(2,2) - c_nh_n) + \frac{1}{2}(E_j(1,2)+E_j(2,1))\right| \xrightarrow{a.s.} 0
\end{align*}

Finally using $\{z_{jn} = \frac{1}{\sqrt{2}}(Z_{1j} + \mathbbm{i}Z_{2j}): 1 \leq j \leq n, n \in \mathbb{N}\}$ and $A_j$ from (\ref{defining_Aj}), Lemma \ref{quadraticForm} and Lemma \ref{Rank2Perturbation} gives us 
\begin{align*}\label{B.11}\tag{B.11}
\underset{1\leq j \leq n}{\max} \left|\frac{1}{2}(E_j(1,1) - c_nh_n) + \frac{1}{2}(E_j(2,2) - c_nh_n) + \frac{1}{2}(E_j(1,2)-E_j(2,1))\right| \xrightarrow{a.s.} 0 
\end{align*}

From (\ref{E_j_rr}), (\ref{B.10}) and (\ref{B.11}) and using Lemma \ref{lA.5}, we get 
\begin{itemize}
    \item $\underset{1\leq j \leq n}{\max} |E_j(1,2)+E_j(2,1)| \xrightarrow{a.s.} 0$
    \item $\underset{1\leq j \leq n}{\max} |E_j(1,2)-E_j(2,1)| \xrightarrow{a.s.} 0$\end{itemize}
A further application of Lemma \ref{lA.5} gives us 
\begin{align*}\label{E_j_rs}\tag{B.12}
\underset{1\leq j \leq n}{\max}|E_j(1,2)|\xrightarrow{a.s.}0 \hspace{6mm}\text{ and } \hspace{6mm} \underset{1\leq j \leq n}{\max}|E_j(2,1)| \xrightarrow{a.s.}0
\end{align*}

Note that by Lemma \ref{normBound}, 
\begin{align*}\label{QbarMQ_bound}\tag{B.13}
||\Bar{Q}M_nQ_{-j}||_{op} \leq ||\Bar{Q}||_{op}||M_n||_{op}||Q_{-j}||_{op} \leq  \frac{B}{\Re ^2(z)}
\end{align*}

Therefore, repeating the same arguments presented through (\ref{E_j_rr})-(\ref{E_j_rs}) (replacing $Q_{-j}$ with $\Bar{Q}M_nQ_{-j}$ throughout), we get the following uniform almost sure convergence results.
\begin{itemize}
    \item $\underset{1\leq j\leq n}{\max }|F_j(r,r) - m_n| \xrightarrow{a.s.} 0, r \in \{1,2\}$
    \item $\underset{1\leq j\leq n}{\max }|F_j(r,s)| \xrightarrow{a.s.} 0  \text{ where } r \neq s, r,s \in \{1,2\} $
\end{itemize}

We now prove the result related to $c_{1j}$ defined in (\ref{defining_cj}). To show $\underset{1 \leq j \leq n}{\max}|c_{1j} - v_n| \xrightarrow{a.s.} 0$,  define for $1 \leq j \leq n$, 
\begin{description}
    \item[1] $A_{jn} = 1 - E_j(1,2)$
    \item[2] $B_{jn} = Den(j)$ \text{, defined in} (\ref{defining_cj})
    \item[3] $C_{jn} = 1$ and $D_{jn} = v_n$
\end{description}
Due to Lemma \ref{boundedAwayFromZero}, we see that $A_{jn}, B_{jn}, C_{jn}, D_{jn}$ satisfy the conditions of Lemma \ref{lA.6}. Therefore, we have the result associated with $c_{1j}$. The results for $c_{2j}, d_{1j}, d_{2j}$ can be established by similar arguments. 
\end{proof}

\subsection{Proof of Theorem \ref{DeterministicEquivalent}}\label{ProofDeterministicEquivalent}

\begin{proof}
Let $z \in \mathbb{C}_L$. Define $F(z) := \bigg(\Bar{Q}(z)\bigg)^{-1}$. From the resolvent identity (\ref{R123}), we have 
\begin{align*}\label{B.14}\tag{B.14}
Q - \Bar{Q} = Q \bigg(F + zI_p - \dfrac{1}{n}\sum_{j=1}^{n}(X_{1j}X_{2j}^{*} - X_{2j}X_{1j}^{*})\bigg) \Bar{Q}    
\end{align*}

Using the above, we get
\begin{align*}\label{T1_minu_T2_expansion}\tag{B.15}
& \frac{1}{p}\operatorname{trace}\{(Q - \Bar{Q})M_n\}\\
=&\frac{1}{p}\operatorname{trace} \{Q(F + zI_p)\Bar{Q}M_n \} - \frac{1}{p} \operatorname{trace} \{Q \bigg(\sum_{j=1}^{n}\dfrac{1}{n} (X_{1j}X_{2j}^{*} - X_{2j}X_{1j}^{*})\bigg) \Bar{Q} M_n \} \\
=&\frac{1}{p}\operatorname{trace} \{(F + zI_p)\Bar{Q}M_nQ \} - \frac{1}{p} \operatorname{trace} \{\bigg(\sum_{j=1}^{n}\dfrac{1}{n} (X_{1j}X_{2j}^{*} - X_{2j}X_{1j}^{*})\bigg) \Bar{Q} M_n Q\} \\
= & \underbrace{\frac{1}{p}\operatorname{trace} \{(F + zI_p)\Bar{Q}M_nQ \}}_{{Term}_1} - \underbrace{\frac{1}{p} \sum_{j=1}^{n}\dfrac{1}{n} (X_{2j}^{*}\Bar{Q}M_nQX_{1j} - X_{1j}^{*}\Bar{Q}M_nQX_{2j})}_{{Term}_2} 
\end{align*}

To establish ${Term}_1 - {Term}_2 \xrightarrow{a.s.} 0$, we define the following.
\begin{align*}\label{definingEjrs}\tag{B.16}
    &E_j(r,s) := \frac{1}{n}X_{rj}^*Q_{-j}X_{sj} = \frac{1}{n}Z_{rj}^*\Sigma_n^\frac{1}{2}Q_{-j}\Sigma_n^\frac{1}{2}Z_{sj} \text{ for } r,s \in \{1,2\}, 1\leq j \leq n
\end{align*}
\begin{align*}\label{definingFjrs}\tag{B.17}
        &F_j(r,s) := \frac{1}{n}X_{rj}^*\Bar{Q}M_nQ_{-j}X_{sj} \text{ for } r,s \in \{1,2\}, 1\leq j \leq n
\end{align*}
\begin{align*}\label{defining_vn}\tag{B.18}
        &v_n(z) := \dfrac{1}{1 + (c_nh_n(z))^2}
\end{align*}
\begin{align*}\label{defining_mn}\tag{B.19}
        &m_n(z) := \frac{1}{n}\operatorname{trace}\{\Sigma_n\Bar{Q}M_nQ\} 
\end{align*}

Simplifying ${Term}_2$ using Lemma \ref{Rank2Woodbury}, with $A = Q_{-j}(z)$ (see (\ref{notations})), $u = \frac{1}{\sqrt{n}}X_{1j}$ and $v = \frac{1}{\sqrt{n}}X_{2j}$, we get

\begin{align*}\label{defining_cj}\tag{B.20}
        \frac{1}{\sqrt{n}} QX_{1j} =& Q_{-j} \bigg(\frac{1}{\sqrt{n}}X_{1j} c_{1j} + \frac{1}{\sqrt{n}}X_{2j} c_{2j}\bigg)\\ 
        \text{where } c_{1j} =& (1 - E_j(1,2))Den(j); \hspace{5mm} c_{2j} =E_j(1,1)Den(j)\\
        Den(j) =& \bigg((1 - E_j(1,2))(1 + E_j(2,1)) + E_j(1,1)E_j(2,2)\bigg)^{-1} 
\end{align*}
and
\begin{align*}\label{defining_dj}\tag{B.21}
        \frac{1}{\sqrt{n}} QX_{2j} =& Q_{-j} \bigg(\frac{1}{\sqrt{n}}X_{2j} d_{1j} - \frac{1}{\sqrt{n}}X_{1j} d_{2j}\bigg)\\ 
        \text{where } d_{1j} =& (1 + E_j(2,1))Den(j);\hspace{5mm} d_{2j} = E_j(2,2)Den(j)
\end{align*}

Using (\ref{defining_cj}) and (\ref{defining_dj}), $Term_2$ of (\ref{T1_minu_T2_expansion}) can be simplified as follows.
\begin{align*}\label{T2_simplification}\tag{B.22}
&{Term}_2
        =\frac{1}{p}\sum_{j=1}^{n}\dfrac{1}{n} (X_{2j}^{*}\Bar{Q}M_nQX_{1j} - X_{1j}^{*}\Bar{Q}M_nQX_{2j}) \\
        =&\sum_{j=1}^{n}\dfrac{1}{p\sqrt{n}} X_{2j}^*\Bar{Q}M_n \bigg(\frac{1}{\sqrt{n}} Q X_{1j}\bigg) -  \sum_{j=1}^{n}\dfrac{1}{p\sqrt{n}} X_{1j}^*\Bar{Q}M_n \bigg( \frac{1}{\sqrt{n}}QX_{2j}\bigg)\\
        =& \sum_{j=1}^{n}\dfrac{1}{p\sqrt{n}} X_{2j}^*\Bar{Q}M_n Q_{-j} \bigg(\frac{X_{1j}c_{1j} + X_{2j}c_{2j}}{\sqrt{n}}\bigg) - \sum_{j=1}^{n}\dfrac{1}{p\sqrt{n}} X_{1j}^*\Bar{Q}M_n Q_{-j}\bigg( \frac{X_{2j}d_{1j} - X_{1j}d_{2j}}{\sqrt{n}}\bigg)\\
        =& \frac{1}{p}\sum_{j=1}^{n}\Bigg[\bigg(c_{1j}F_j(2,1) + c_{2j}F_j(2,2)\bigg) - \bigg(d_{1j}F_j(1,2) - d_{2j}F_j(1,1)\bigg)\Bigg]  \text{ using definition } (\ref{definingFjrs})
\end{align*}

To proceed further, we need the limiting behaviour of $c_{1j}, c_{2j}, d_{1j}, d_{2j}, F_j(r,s), r,s \in \{1,2\}$ for $1 \leq j \leq n$. This is established in Lemma \ref{uniformConvergence} and the summary of results is given below.
\begin{equation*}\label{uniformResults}\tag{B.23}
\left\{ \begin{aligned} 
    &\underset{1\leq j\leq n}{\max }|c_{1j} - v_n| \xrightarrow{a.s.} 0 \hspace{5mm}
    \underset{1\leq j\leq n}{\max }|d_{1j} - v_n|\xrightarrow{a.s.} 0\\
    &\underset{1\leq j\leq n}{\max }|c_{2j} - c_nh_nv_n|\xrightarrow{a.s.} 0 \hspace{5mm}
    \underset{1\leq j\leq n}{\max }|d_{2j} - c_nh_nv_n|\xrightarrow{a.s.} 0\\
    & \underset{1\leq j\leq n}{\max }|F_j(r,r) - m_n| \xrightarrow{a.s.} 0, r \in \{1,2\}\\
    & \underset{1\leq j\leq n}{\max }|F_j(r,s)| \xrightarrow{a.s.} 0  \text{ where } r \neq s, r,s \in \{1,2\} 
\end{aligned} \right.
\end{equation*}

Note that \begin{enumerate}
    \item For sufficiently large $n$, $|v_n|$ is bounded above by Lemma \ref{boundedAwayFromZero} 
    \item By (\ref{h_nBound}), $|h_n| \leq C/|\Re(z)|$ 
    \item Using (R5) of \ref{R123} and (\ref{QbarMQ_bound}), for sufficiently large n,
\begin{align*}\label{mn_bound}\tag{B.24}
|m_n| = \bigg|\frac{1}{n}\operatorname{trace}\{\Sigma_n\Bar{Q}M_nQ\}\bigg| \leq \bigg(\frac{1}{n}\operatorname{trace}(\Sigma_n)\bigg) ||\Bar{Q}M_nQ||_{op} \leq \frac{BC}{\Re ^2(z)}    
\end{align*}

\end{enumerate}

From (\ref{uniformResults}), the above bounds and applying Lemma \ref{lA.6}, we get the following results 
\begin{enumerate}
    \item $\underset{1 \leq j \leq n}{\max}|c_{1j}F_j(2,1)| \xrightarrow{a.s.} 0$; \hspace{5mm} $\underset{1 \leq j \leq n}{\max}|d_{1j}F_j(1,2)| \xrightarrow{a.s.} 0$
    \item $\underset{1 \leq j \leq n}{\max}|c_{2j}F_j(2,2) - c_nh_nv_nm_n| \xrightarrow{a.s.} 0$; \hspace{5mm} $\underset{1 \leq j \leq n}{\max}|d_{2j}F_j(1,1) - c_nh_nv_nm_n| \xrightarrow{a.s.} 0$    
\end{enumerate}

With the above results, Lemma \ref{lA.7} applied on (\ref{T2_simplification}) gives 
\begin{align*}\label{T2_simplification2}\tag{B.25}
|Term_2 - 2h_nv_nm_n| \xrightarrow{a.s.} 0    
\end{align*}

Now note that 
\begin{align*}\tag{B.26}\label{T2_expansion}
    2h_nv_nm_n =& \frac{n}{p} \dfrac{2c_nh_n}{1 + (c_nh_n)^2}\frac{1}{n}\operatorname{trace}\{\Sigma_n \Bar{Q}M_nQ\} \text{, by definitions } (\ref{defining_vn}), (\ref{defining_mn}) \\
    =& \frac{1}{p} \bigg[\frac{1}{\mathbbm{i} + c_nh_n} + \frac{1}{-\mathbbm{i} + c_nh_n}\bigg]\operatorname{trace}\{\Sigma_n\Bar{Q}M_nQ\}\\
    =&\frac{1}{p} \operatorname{trace}\{\rho(c_nh_n)\Sigma_n\Bar{Q}M_nQ\}
\end{align*}
where the last equality follows from definition \ref{definingRho}. Finally from (\ref{rhoContinuity1}) and (\ref{mn_bound}), we get
\begin{align*}\label{interim1}\tag{B.27}
\bigg|\frac{1}{p} \operatorname{trace}\{\rho(c_nh_n)\Sigma_n\Bar{Q}M_nQ\} - \frac{1}{p} \operatorname{trace}\{\rho(c_n\mathbb{E}h_n)\Sigma_n\Bar{Q}M_nQ\}\bigg| \xrightarrow{a.s.} 0    
\end{align*}

Combining (\ref{T2_simplification2}), (\ref{T2_expansion}) and (\ref{interim1}), we get
\begin{align*}
    &|Term_2 - \frac{1}{p} \operatorname{trace}\{\rho(c_n\mathbb{E}h_n)\Sigma_n\Bar{Q}M_nQ\}| \xrightarrow{a.s.} 0\\
    \implies & |Term_2 - \frac{1}{p}\operatorname{trace}\{[zI_p -zI_p + \rho(c_n\mathbb{E}h_n)\Sigma_n]\Bar{Q}M_nQ\}| \xrightarrow{a.s.} 0\\
    \implies & |Term_2 - \frac{1}{p} \operatorname{trace}\{(F(z) + zI_p)\Bar{Q}M_nQ\}| \xrightarrow{a.s.} 0\\
    \implies& |Term_2 - Term_1| \xrightarrow{a.s.} 0
\end{align*}
This concludes the proof.
\end{proof}

\subsection{Proof of Theorem \ref{ExistenceA123}}\label{ProofExistenceA123}

\begin{proof}
By Lemma \ref{CompactConvergence}, every sub-sequence of $\{h_n(z)\}$ has a further sub-sequence that converges uniformly in each compact subset of $\mathbb{C}_L$. Let $h_0(z)$ be one such sub-sequential limit corresponding to the sub-sequence $\{h_{k_m}(z)\}_{m=1}^\infty$. For simplicity, denote  $g_m = h_{k_m}, \Tilde{g}_m = \Tilde{h}_{k_m}, \Tilde{\Tilde{g}}_m = \Tilde{\Tilde{h}}_{k_m}, d_m = c_{k_m}$ and $G_m = F^{\Sigma_{k_m}}$ for $m \in \mathbb{N}$. Thus we have $g_m \xrightarrow{a.s.} h_0$. By Theorem \ref{DeterministicEquivalent}, we have, 

$$|\Tilde{g}_m - h_0| = |\Tilde{h}_{k_m} - h_0| \leq |\Tilde{h}_{k_m} - h_{k_m}| + |h_{k_m} - h_0| \rightarrow 0$$

Therefore, $\Tilde{g}_m \rightarrow h_0$. By Lemma \ref{hn_tilde_tilde2}, we have
\begin{align*}\label{gm_tilde_tilde2}\tag{B.28}
&\Tilde{g}_m(z) - \Tilde{\Tilde{g}}_m(z) \rightarrow 0\\
\implies &\displaystyle \Tilde{g}_m(z)- \int\dfrac{\lambda dG_m(\lambda)}{-z + \lambda\rho(d_m\Tilde{g}_m(z))} \longrightarrow 0\\
\implies &\Tilde{g}_m(z) -\int\dfrac{\lambda d\{G_m(\lambda)-H(\lambda)\}}{-z +\lambda\rho(d_m\Tilde{g}_m(z))} -
\int\dfrac{\lambda dH(\lambda)}{-z + \lambda\rho(d_m\Tilde{g}_m(z))} \longrightarrow 0
\end{align*}

For large $m$, the common integrand in the second and third terms of (\ref{gm_tilde_tilde2}) can be bounded above as follows.
\begin{align}\label{integrandBound}\tag{B.29}
    &\bigg|\frac{\lambda}{-z + \lambda\rho(d_m\Tilde{g}_m)}\bigg|
    \leq \frac{|\lambda|}{|\Re(-z + \lambda\rho(d_m\Tilde{g}_m))|}
    \leq \frac{|\lambda|}{|\Re(\lambda\rho(d_m\Tilde{g}_m))|} = \frac{1}{\Re(\rho(d_m\Tilde{g}_m))
    } \xrightarrow{} \frac{1}{\Re(\rho(ch_0))}
\end{align}

The limit in (\ref{integrandBound}) follows because of the following argument. First note that $\Re(ch_0) > 0$. To see this, note that $g_m \xrightarrow{a.s.} h_0, d_m \rightarrow c$ and by Lemma \ref{boundedAwayFromZero}, for sufficiently large m, $\Re(d_mg_m) \geq K_0(c,z,H, \tau) > 0$ a.s. Therefore, $\Re(ch_0) > 0$.
Secondly, $d_m\Tilde{g}_m = c_{k_m}\Tilde{h}_{k_m} \rightarrow ch_0$. By continuity of $\rho$ at $ch_0$, we have $\rho(d_m\Tilde{g}_m) \rightarrow \rho(ch_0)$. Finally by (\ref{realOfRho}), $\Re(\rho(ch_0)) = \Re(ch_0)\rho_2(ch_0) > 0$.

So the second term of (\ref{gm_tilde_tilde2}) can be made arbitrarily small as $G_m \xrightarrow{d} H$. Applying D.C.T. in the third term of (\ref{gm_tilde_tilde2}), we get 
\begin{align*}\label{subsequentialLimit}\tag{B.30}
h_0(z) = \displaystyle\int\dfrac{\lambda dH(\lambda)}{-z + \lambda\rho(ch_0(z))}
\end{align*}

Thus any subsequential limit ($h_0(z) \in \mathbb{C}_R$) satisfies (\ref{h_main_eqn}). By Lemma \ref{Uniqueness}, all sub-sequential limits must be the same, say $h^\infty(z)$. This implies $h_n(z) \xrightarrow{a.s.} h^\infty(z)$ where the convergence is uniform in each compact subset of $\mathbb{C}_L$. Therefore, the limit $h^\infty$ must be analytic in $\mathbb{C}_L$ itself.

To complete the proof, we need to prove that $s_n(z) \xrightarrow{a.s.} s^\infty(z)$. Here we define an intermediate quantity 
\begin{align*}\label{defining_sn_tilde}\tag{B.31}
\Tilde{s}_n(z) := \frac{1}{z}\bigg(\frac{2}{c_n}-1\bigg) + \frac{1}{\mathbbm{i}c_nz}\bigg(\frac{1}{\mathbbm{i} +c_n \Tilde{h}_n(z)} - \frac{1}{-\mathbbm{i} +c_n \Tilde{h}_n(z)}\bigg)    
\end{align*}

It is clear that $\Tilde{s}_n(z) \rightarrow s^\infty(z)$ since $c_n \rightarrow c$ and $\Tilde{h}_n(z) \rightarrow h^\infty(z)$. So it is sufficient to show $s_n(z) - \Tilde{s}_n(z) \xrightarrow{a.s.} 0$. We will utilise use relationship between the resolvent and the co-resolvent for this. The co-resolvent of $S_n$ is given by 
$$\Tilde{Q}(z) := \bigg(\dfrac{1}{n}\begin{bmatrix}
    X_2^*\\
    X_1^*\\
\end{bmatrix}[X_1:-X_2] - zI_{2n}\bigg)^{-1}$$ 

First we simplify the expression a bit. Let $A = \frac{1}{\sqrt{n}}[X_{1}:-X_{2}]$ and $B = \frac{1}{\sqrt{n}}[X_{2}:X_{1}]^*$. Then $Q(z) = (AB - zI_p)^{-1}$ and $\Tilde{Q}(z) = (BA - zI_{2n})^{-1}$. Observe that

\begin{align*}\label{defining_Qtilde}\tag{B.32}
(BQA - I_{2n})(BA - zI_{2n}) =& zI_{2n} \\
\implies \Tilde{Q}(z) = \frac{1}{z}BQA - \frac{1}{z}I_{2n}
=& \dfrac{1}{nz}\begin{bmatrix}
    X_2^*\\
    X_1^*\\
\end{bmatrix}Q[X_1:-X_2] - \dfrac{1}{z}I_{2n}\\
=& \frac{1}{z}\begin{bmatrix}
    \frac{1}{n}X_2^*QX_1 - I_n & -\frac{1}{n}X_2^*QX_2\\
    \frac{1}{n}X_1^*QX_1 & -\frac{1}{n}X_1^*QX_2 - I_n
\end{bmatrix}\
=: \frac{1}{z}\begin{bmatrix}
    U & *\\
    * & V\\
\end{bmatrix}    
\end{align*}

We focus only on the two diagonal blocks $U$ and $V$ of $\Tilde{Q}$ (\ref{defining_Qtilde}). Then for $1 \leq j \leq n$, 
\begin{align*}\label{zAjj}\tag{B.33}
U_{jj}
=&\frac{1}{n}X_{2j}^*QX_{1j} - 1\\
=& \frac{1}{n}X_{2j}^*Q_{-j}(c_{1j} X_{1j} + c_{2j} X_{2j}) - 1 \text{, where } c_{1j}, c_{2j} \text{ are as per } (\ref{defining_cj})\\
=& c_{1j} E_j(2,1) + c_{2j}E_j(2,2) - 1
\end{align*}

From Lemma \ref{uniformConvergence}, we have $\underset{1\leq j\leq n}{\max}|E_j(2,1)| \xrightarrow{a.s.} 0$, $\underset{1\leq j\leq n}{\max}|E_j(2,2) - c_nh_n| \xrightarrow{a.s.} 0$,\\    
$\underset{1 \leq j \leq n}{\max}|c_{1j} - v_n| \xrightarrow{a.s.} 0$ and $\underset{1 \leq j \leq n}{\max}|c_{2j} -c_nh_nv_n| \xrightarrow{a.s.} 0$. Using Lemma \ref{lA.6}, (\ref{h_nBound}) and Lemma \ref{boundedAwayFromZero}, we have 
\begin{itemize}
    \item $|c_{1j}E_j(2,1)| \xrightarrow{a.s.} 0$
    \item $|c_{2j}E_j(2,2) - (c_nh_n)^2v_n| \xrightarrow{a.s.} 0$
\end{itemize}

Using the above results in (\ref{zAjj}) and by Lemma \ref{lA.5}, we get 
\begin{align*}
&\underset{1\leq j \leq n}{\max}|U_{jj} - (0 + (c_nh_n)^2v_n - 1)| \xrightarrow{a.s.} 0\\
\implies & \underset{1\leq j \leq n}{\max}\bigg|U_{jj} - \dfrac{(c_nh_n)^2}{1 + (c_nh_n)^2} + 1\bigg| \xrightarrow{a.s.} 0\\
\implies & \underset{1\leq j \leq n}{\max}|U_{jj} +(1+(c_nh_n)^2)^{-1}| \xrightarrow{a.s.} 0
\end{align*}

 Now from Theorem \ref{DeterministicEquivalent}, we have $|h_n - \Tilde{h}_n| \xrightarrow{a.s.} 0$ and from Lemma \ref{boundedAwayFromZero}, for sufficiently large $n$, $|v_n| = |1 + (c_nh_n)^2|^{-1}$ is bounded which implies $|1 + (c_n\Tilde{h}_n)^2|^{-1}$ is bounded. Moreover by (\ref{h_nBound}), we have $|h_n| \leq C/|\Re(z)|$ which implies $|\Tilde{h}_n| \leq C/|\Re(z)|$. Therefore,

\begin{align}\tag{B.34}
    \bigg|\dfrac{1}{1 + (c_nh_n)^2} - \dfrac{1}{1 + (c_n\Tilde{h}_n)^2}\bigg| \leq \dfrac{c_n^2|h_n - \Tilde{h}_n|(|h_n| + |\Tilde{h}_n|)}{|1 + (c_nh_n)^2||1 + (c_n\Tilde{h}_n)^2|} \xrightarrow{} 0
\end{align}

Hence, we have 
$\underset{1 \leq j \leq n}{\max}|U_{jj} + (1 + c_n^2 \Tilde{h}_n^2)^{-1}| \xrightarrow{a.s.} 0$. Similarly, $\underset{1 \leq j \leq  n}{\max} |V_{jj} + (1 + c_n^2 \Tilde{h}_n^2)^{-1}| \xrightarrow{a.s.} 0$. Therefore using Lemma \ref{lA.7},
\begin{align*}\label{Ujj_Vjj}\tag{B.35}
&\frac{1}{2n}\bigg|\sum_{j=1}^{n} \bigg(U_{jj} + V_{jj} + 2(1 + (c_n^2\Tilde{h}_n^2))^{-1}\bigg)\bigg| \xrightarrow{a.s.} 0\\
\implies& \bigg|\frac{1}{2n}\operatorname{trace}(\Tilde{Q}) + z^{-1}(1 + (c_n^2\Tilde{h}_n^2))^{-1}\bigg| \xrightarrow[]{} 0
\end{align*}

Using the identity linking the trace of the resolvent with that of the co-resolvent, we get
\begin{align*}
& \frac{1}{2n}\operatorname{trace}(\Tilde{Q}) = \frac{1}{2n}\operatorname{trace}(Q) + \frac{p-2n}{2nz}\\
\implies & \bigg|\frac{1}{z}\frac{1}{1 + c_n^2 \Tilde{h}_n^2} + \frac{c_n}{2}s_n(z) + \frac{1}{z}\bigg(\frac{c_n}{2} - 1\bigg)\bigg| \xrightarrow{a.s.} 0\\
\implies & \bigg|s_n(z) + \frac{1}{z}\bigg(1 - \frac{2}{c_n}\bigg) + \frac{2}{c_nz}\bigg(\frac{1}{1 + c_n^2 \Tilde{h}_n^2(z)}\bigg)\bigg| \xrightarrow{a.s.} 0\\
\implies & \bigg|s_n(z) - \dfrac{1}{z}\bigg(\dfrac{2}{c_n}-1\bigg) - \dfrac{1}{\mathbbm{i}c_nz}\bigg(\dfrac{1}{\mathbbm{i} +c_n\Tilde{h}_n(z)} - \dfrac{1}{-\mathbbm{i} +c_n\Tilde{h}_n(z)}\bigg)\bigg| \xrightarrow{a.s.} 0\\
\implies & |s_n(z) - \Tilde{s}_n(z)| \xrightarrow{a.s.} 0 \text{, using the definition of } \Tilde{s}_n \text{ from } (\ref{defining_sn_tilde})
\end{align*}

Therefore, $s_n(z) \xrightarrow{a.s.} s^\infty(z)$. From (\ref{h_nBound}), for sufficiently large $n$, $|h_n| \leq {C}/{|\Re(z)|}$ for $z \in \mathbb{C}_L$. Thus for $y > 0$, $|h^\infty(-y)| \leq {C}/{|y|}$ and $\underset{y \rightarrow \infty}{\lim}h^\infty(-y) = 0$. This implies that 
\begin{align*}\tag{B.36}
\underset{y \rightarrow +\infty}{\lim}ys^\infty(-y) = 1 - \frac{2}{c} - \underset{y \rightarrow \infty}{\lim}\frac{1}{\mathbbm{i}c}\bigg(\frac{1}{\mathbbm{i} +ch^\infty(-y)} - \frac{1}{-\mathbbm{i} +ch^\infty(-y)}\bigg) = 1    
\end{align*}

Since $s^\infty(.)$ satisfies the necessary and sufficient condition from Proposition \ref{GeroHill}, it is a Stieltjes transform of some probability distribution $F^\infty$ and $F^{S_n} \xrightarrow{d} F^\infty$ a.s. So we have proved (2)-(4) of Section \ref{ProofSketch} under \textbf{A1-A2} of Section \ref{A123}.
\end{proof}

\subsection{Results related to Proof of Existence under General Conditions}

\begin{lemma}\label{3.4.1}
    $h^\tau \rightarrow h^\infty$, $s^\tau \rightarrow s^\infty$  as $\tau \rightarrow \infty$
\end{lemma}
\begin{proof}    
Since Theorem \ref{t3.1} holds for $\Tilde{U}_n$, we have $F^{\Tilde{U}_n} \xrightarrow{d} F^{\tau}$ for some LSD $F^{\tau}$ and for $z \in \mathbb{C}_L$, there exists functions $s^\tau(z)$ and $h^\tau(z)$ satisfying (\ref{s_main_eqn}) and (\ref{h_main_eqn}) with $H^\tau$ replacing $H$ and mapping $\mathbb{C}_L$ to $\mathbb{C}_R$ and analytic on $\mathbb{C}_L$. We have to show existence of analogous quantities for the sequence $\{F^{S_n}\}_{n=1}^\infty$.

 First assume that H has a bounded support. If $\tau > \sup \operatorname{supp}(H)$, then $H^{\tau}(t) = H(t)$ and $H(\tau) = 1$. By the uniqueness property from Lemma \ref{Uniqueness}, $h^{\tau}(z)$ must be the same for all large $\tau$. Hence $s^{\tau}(z)$ and in turn $F^{\tau}(.)$ must also be the same for all large $\tau$.  Denote this common LSD by $F^\infty$ and the common value of $h^\tau$ and $s^\tau$ by $h^\infty$ and $s^\infty$ respectively.
 This proves the case when H has a bounded support. 

Now we analyse the case where H has unbounded support. We need to show there exist functions $h^\infty$, $s^\infty$ that satisfy equations (\ref{s_main_eqn}) and (\ref{h_main_eqn}) and an LSD $F^\infty$ serving as the limit for the ESDs of $\{S_n\}_{n=1}^\infty$.

 We will show that $\mathcal{H} = \{h^\tau: \tau > 0\}$ forms a normal family. Following arguments similar to those used in the proof of Lemma \ref{CompactConvergence}, let $K \subset \mathbb{C}_L$ be an arbitrary compact subset. Then $u_0 > 0$ where $u_0 := \inf\{|\Re(z)|: z \in K\}$. For arbitrary $z \in K$, using (R5) of \ref{R123} and (T4) of Theorem \ref{t3.1}, for sufficiently large $n$, we have
    \begin{align}\label{hn_tau_Bound}\tag{B.37}
        |h_n^\tau(z)| =& \frac{1}{p}|\operatorname{trace}\{\Sigma_n^\tau Q\}| \leq \bigg(\frac{1}{p}\operatorname{trace}(\Sigma_n^\tau)\bigg) ||Q||_{op} 
        \leq \frac{C}{|\Re(z)|} \leq \frac{C}{u_0}
    \end{align}

By Theorem \ref{ExistenceA123}, for any $\tau > 0$, $h^\tau(z)$ is the uniform limit of $h_n^\tau(z) := \frac{1}{n}\operatorname{trace}\{\Sigma_n^\tau Q(z)\}$. Therefore, for $z \in K$,    
\begin{align*}\label{h_tauBound}\tag{B.38}
|h^\tau(z)|\leq \frac{C}{|\Re(z)|} \leq \frac{C}{u_0} 
\end{align*}
Therefore as a consequence of \textit{Montel's theorem}, any subsequence of $\mathcal{H}$ has a further convergent subsequence that converges uniformly on compact subsets of $\mathbb{C}_L$. Let $\{h^{\tau_m}(z)\}_{m=1}^\infty$ be such a sequence with $h_0(z)$ as the subsequential limit where $\tau_m \rightarrow \infty$ as $m \rightarrow \infty$.
By Lemma \ref{boundedAwayFromZero}, $\Re(h^{\tau_m}(z)) \geq K_0(z,c,\tau_m, H^{\tau_m}) > 0$ which implies that $\Re(h_0(z)) \geq 0$ $\forall z \in \mathbb{C}_L$.

\begin{claim}\label{Claim}
We must have $\Re(h_0(z)) > 0$ for all $z \in \mathbb{C}_L$. For a proof of this claim, see \ref{subsec:ProofOfClaim}.    
\end{claim}

By (\ref{realOfRho}), and the fact that $\Re(h_0) > 0$, we have $\Re(\rho(ch_0)) = \Re(ch_0)\rho_2(ch_0)> 0$. Therefore by continuity of $\rho$ at $ch_0$, 
\begin{align*}\label{integrandBound3}\tag{B.39}
\bigg|\frac{\lambda}{-z + \lambda\rho(ch^{\tau_m})}\bigg|
    \leq \frac{|\lambda|}{|\Re(-z + \lambda\rho(ch^{\tau_m}))|}
    \leq \frac{|\lambda|}{|\Re(\lambda\rho(ch^{\tau_m}))|} = \frac{1}{\Re(\rho(ch^{\tau_m}))} \xrightarrow{} \frac{1}{\Re(\rho(ch_0))}  < \infty  
\end{align*}
as $m \rightarrow \infty$. Now by Theorem \ref{ExistenceA123}, $(h^{\tau_m}, H^{\tau_m})$ satisfy the below equation.

\begin{align*}
   h^{\tau_m}(z) =& \int \dfrac{\lambda dH^{\tau_m}(\lambda)}{-z + \lambda\rho(ch^{\tau_m})}
  = \int \dfrac{\lambda d\{H^{\tau_m}(\lambda) - H(\lambda)\}}{-z + \lambda\rho(ch^{\tau_m})} + \int \dfrac{\lambda dH(\lambda)}{-z +\lambda\rho(ch^{\tau_m})}\\
\end{align*}
Note that the first term of the last expression can be made arbitrarily small as the integrand is bounded by (\ref{integrandBound3}) and $H^{\tau_m} \xrightarrow{d} H$. The same bound on the integrand also allows us to apply D.C.T. in the second term thus giving us
\begin{align*}
  &\underset{m \rightarrow \infty}{\lim} h^{\tau_m}(z) = \underset{m \rightarrow \infty}{\lim} \int \dfrac{\lambda dH(\lambda)}{-z +\lambda\rho(ch^{\tau_m})}\\
  \implies &  h_0(z) = \int \dfrac{\lambda dH(\lambda)}{-z +\lambda\rho(ch_0(z))}\label{subsequentialLimit2}\tag{B.40}
\end{align*}

Now $\{\tau_m\}_{m=1}^\infty$ is a further subsequence of an arbitrary subsequence and $\{h^{\tau_m}(z)\}$ converges to $h_0(z) \in \mathbb{C}_R$ that satisfies (\ref{h_main_eqn}). By Theorem \ref{Uniqueness}, all these subsequential limits must be the same and that $\{h^\tau(z)\}$ must converge uniformly to this common limit which we denote by $h^\infty(z)$. The uniform convergence of analytic functions $h^\tau$ in each compact subset of $\mathbb{C}_L$ imply that the limit $h^\infty$ must be analytic in $\mathbb{C}_L$.

Notice that $s^\tau(z) \rightarrow s^\infty(z)$ as $\tau \rightarrow \infty$ since 
\begin{align*}\label{s_tau_limit}\tag{B.41}
    \underset{\tau \rightarrow \infty}{\lim}s^\tau(z) 
    =&\underset{\tau \rightarrow \infty}{\lim} \frac{1}{z}\bigg(\frac{2}{c}-1\bigg) + \frac{1}{\mathbbm{i}cz}\bigg(\frac{1}{\mathbbm{i} +ch^\tau(z)} - \frac{1}{-\mathbbm{i} +ch^\tau(z)}\bigg)\\
    =& \frac{1}{z}\bigg(\frac{2}{c}-1\bigg) + \frac{1}{\mathbbm{i}cz}\bigg(\frac{1}{\mathbbm{i} +ch^\infty(z)} - \frac{1}{-\mathbbm{i} +ch^\infty(z)}\bigg)
    =s^\infty(z)
\end{align*}

From (\ref{h_tauBound}), $|h^\tau(z)| \leq {C}/{|\Re(z)|}$. Thus, $|h^\infty(z)| \leq {C}/{|\Re(z)|}$ implying that $\underset{y \rightarrow \infty}{\lim}h^\infty(-y) = 0$. Therefore, 
$$\underset{y \rightarrow +\infty}{\lim}ys^\infty(-y) = \bigg(1 - \frac{2}{c}\bigg) - \underset{y \rightarrow \infty}{\lim}\frac{1}{\mathbbm{i}c}\bigg(\frac{1}{\mathbbm{i} +ch^\infty(-y)} - \frac{1}{-\mathbbm{i} +  ch^\infty(-y)}\bigg)= 1$$ So, we have established that \begin{itemize}
    \item $h^\tau \rightarrow h^\infty$ and $s^\tau \rightarrow s^\infty$
    \item $h^\infty$ satisfies (\ref{h_main_eqn}) and is analytic on $\mathbb{C}_L$
    \item $s^\infty$ satisfies the conditions of Proposition \ref{GeroHill} for a Stieltjes Transform (of a probability measure on the imaginary axis)
\end{itemize}
\end{proof}

\begin{lemma}\label{3.4.2}
    $||F^{S_n} - F^{T_n}||_{im} \xrightarrow{a.s.} 0$
\end{lemma} 
\begin{proof}
Using (R1) and (R3) of Section \ref{R123}, we get
\begin{align*}
  ||F^{S_n} - F^{T_n}||_{im} \leq& \dfrac{1}{p}\operatorname{rank}(S_n - T_n)\\
  =& \frac{1}{p}\operatorname{rank}\bigg(\Lambda_n(Z_1Z_2^* - Z_2Z_1^*)\Lambda_n - \Lambda_n^\tau(Z_1Z_2^* - Z_2Z_1^*)\Lambda_n^\tau\bigg)\\
  \leq& \dfrac{2}{p}\operatorname{rank}(\Lambda_n - \Lambda_n^\tau)\\
  =& \dfrac{2}{p}\operatorname{rank}(\Sigma_n - \Sigma_n^\tau)\\
  =& 2(1 - F^{\Sigma_n}(\tau)) \xrightarrow{n \rightarrow \infty} 2(1 - H(\tau)) \xrightarrow{\text{ as } \tau \rightarrow \infty} 0
\end{align*}
Here $\tau$ approaches $\infty$ only through continuity points of $H$.

\end{proof}

\begin{lemma} \label{3.4.3}
    $||F^{T_n} - F^{U_n}||_{im} \xrightarrow{a.s.} 0$, $||F^{U_n}- F^{\Tilde{U}_n}||_{im} \xrightarrow{a.s.} 0$
\end{lemma}
\begin{proof}
We have $T_n - U_n = \frac{1}{n}\Lambda_n^\tau(Z_1Z_2^* - \Hat{Z}_1\Hat{Z}_2^*)\Lambda_n^\tau - \frac{1}{n}\Lambda_n^\tau(Z_2Z_1^* - \Hat{Z}_2\Hat{Z}_1^*)\Lambda_n^\tau$. Therefore using (R1) and (R3) of Section \ref{R123}, we get
\begin{align*}\label{Tn_Un}\tag{B.42}
        ||F^{T_n} - F^{U_n}||_{im} \leq & \dfrac{1}{p}\operatorname{rank} (T_n - U_n) \\ 
        \leq&\frac{1}{p} \operatorname{rank}(Z_1Z_2^* - \Hat{Z}_1\Hat{Z}_2^*) + \frac{1}{p} \operatorname{rank}(Z_2Z_1^* - \Hat{Z}_2\Hat{Z}_1^*)\\
        \leq& \frac{2}{p}\bigg(\operatorname{rank}(Z_1 - \Hat{Z}_1) + \operatorname{rank} (Z_2 - \Hat{Z}_2)\bigg)
\end{align*}

Define $I_{ij}^{(k)} := \mathbbm{1}_{\{z_{ij}^{(k)} \neq \Hat{z}_{ij}^{(k)}\}} = \mathbbm{1}_{\{|z_{ij}^{(k)}| > B_n\}}$ where $B_n$ is defined in Assumption \ref{A123}. Using (R2) of Section \ref{R123}, we get $\operatorname{rank}(Z_k - \Hat{Z}_k) \leq \sum_{ij} I_{ij}^{(k)}$. Noting that 
\begin{align*}
 &\mathbb{P}(I_{ij}^{(k)} = 1)
=\mathbb{P}(|z_{ij}^{(k)}| > B_n)
\leq \dfrac{\mathbb{E}|z_{ij}^{(k)}|^{4+\eta_0}}{B_n^{4+\eta_0}}
\leq \dfrac{M_{4+\eta_0}}{n^{a(4+\eta_0)}}
\end{align*}

we get,
\begin{align*}
\dfrac{1}{p}\sum_{i,j}\mathbb{P}(I_{ij}^{(k)} = 1)
\leq \frac{1}{p} \times \dfrac{npM_{4+\eta_0}}{n^{a(4+\eta_0)}}
=\dfrac{M_{4+\eta_0}}{n^{(4+\eta_0)(a - \frac{1}{4+\eta_0})}} \rightarrow 0    
\end{align*}

Also, we have $\operatorname{Var}I_{ij}^{(k)} \leq \mathbb{P}(I_{ij}^{(k)}=1)$. For arbitrary $\epsilon > 0$, we must  have $\sum_{i,j}\operatorname{Var}I_{ij}^{(k)} \leq p\epsilon/2$ for large enough $n$. Finally we use Bernstein's Inequality to get the following bound.
\begin{align*}
    \mathbb{P}\bigg(\dfrac{1}{p}\sum_{i,j}I_{ij}^{(k)} > \epsilon\bigg)
    \leq & \mathbb{P}\bigg(\sum_{i,j}(I_{ij}^{(k)} - \mathbb{P}(I_{ij}^{(k)} = 1)) > \dfrac{p\epsilon}{2}\bigg)\\
    \leq & 2\exp{\bigg(-\dfrac{p^2\epsilon^2/4}{2(p\epsilon/2 + \sum_{i,j} \operatorname{Var}I_{ij}^{(k)})}\bigg)}\\
    \leq & 2\exp{\bigg(-\dfrac{p^2\epsilon^2/4}{2(p\epsilon/2 + p\epsilon/2)}\bigg)} = 2\exp{\bigg(-\dfrac{p\epsilon}{8}\bigg)}
\end{align*}
By Borel Cantelli lemma, $\frac{1}{p}\sum_{ij}I_{ij}^{(k)} \xrightarrow{a.s.}0$ and thus $\frac{1}{p}\operatorname{rank}(Z_k-\hat{Z}_k) \xrightarrow{a.s.} 0$. Combining this with (\ref{Tn_Un}), we have $||F^{T_n} - F^{U_n}||_{im} \xrightarrow{a.s.} 0$.

For the other result, define $\check{Z}_k = (\check{z}^{(k)}_{ij}) := (z_{ij}^{(k)}I_{ij}^{(k)})$ for $k \in \{1,2\}$. Then, 

\begin{align*}\label{rank_convergence}\tag{B.43}
\dfrac{1}{p}\operatorname{rank}\check{Z}_k = \dfrac{1}{p}\operatorname{rank}(Z_k -\hat{Z}_k)\leq \dfrac{1}{p}\sum_{i,j}\mathbb{P}(I_{ij}^{(k)} = 1) \xrightarrow{a.s.} 0    
\end{align*}

Finally we see that, \begin{align*}
    ||F^{U_n} - F^{\Tilde{U}_n}||_{im}
    \leq & \frac{1}{p} \operatorname{rank}(U_n - \Tilde{U}_n)\\
    \leq & \frac{1}{p}\operatorname{rank}(\Lambda_n^\tau(\Hat{Z}_1\Hat{Z}_2^* - \Tilde{Z}_1\Tilde{Z}_2^*)\Lambda_n^\tau) + \frac{1}{p}\operatorname{rank}\Lambda_n^\tau(\Hat{Z}_2\Hat{Z}_1^* - \Tilde{Z}_2\Tilde{Z}_1^*)\Lambda_n^\tau\\
    \leq & \frac{1}{p}\operatorname{rank}(\Hat{Z}_1\Hat{Z}_2^* - \Tilde{Z}_1\Tilde{Z}_2^*) + \frac{1}{p}\operatorname{rank}(\Hat{Z}_2\Hat{Z}_1^* - \Tilde{Z}_2\Tilde{Z}_1^*)\\
    \leq & \frac{2}{p}\operatorname{rank}(\hat{Z}_1 - \Tilde{Z}_1) + \frac{2}{p}\operatorname{rank}(\Hat{Z}_2 - \Tilde{Z}_2) \text{, using (R3) of Section } \ref{R123} \\
    = & \frac{2}{p}\operatorname{rank}(\mathbb{E}\Hat{Z}_1) + \frac{2}{p}\operatorname{rank}(\mathbb{E}\Hat{Z}_2)\\
    =& \frac{2}{p}\operatorname{rank}(\mathbb{E}\check{Z}_1) + \frac{2}{p}\operatorname{rank}(\mathbb{E}\Check{Z}_2) \text{, since } \textbf{0} = \mathbb{E}Z_k = \mathbb{E}\hat{Z}_k + \mathbb{E}\check{Z}_k\\
    & \longrightarrow 0 \text{, using } (\ref{rank_convergence}) 
\end{align*}
\end{proof}

\subsubsection{Proof of Claim \ref{Claim}}\label{subsec:ProofOfClaim}
\begin{proof}
Suppose not, then $\exists z_0 \in \mathbb{C}_L$ with $\Re(h_0(z_0)) = 0$. Either, $h_0$ is non-constant in which case by the Open Mapping Theorem, $h_0(\mathbb{C}_L)$ is an open set containing $h_0(z_0)$ which is purely imaginary. This implies that there exists $z_1 \in \mathbb{C}_L$, $\Re(h_0(z_1)) < 0$ which is a contradiction. 

The other case is that $h_0$ is constant in which case. For some $\zeta \in \mathbb{R}$, let $h_0(z) = \mathbbm{i}\zeta, z \in \mathbb{C}_L$. Note that for any $\tau > 0$, using the fact that $\rho(\overline{z}) =\overline{\rho(z)}$ (see the remark immediately following (\ref{realOfRho})), we get
\begin{align*}
\overline{h^\tau(z)} = \int \frac{\lambda dH(\lambda)}{-\overline{z} + \lambda \overline{\rho(ch^\tau(z))}} = \int \frac{\lambda dH(\lambda)}{-\overline{z} + \lambda \rho(c\overline{h^\tau(z)})} = h^\tau(\overline{z})
\end{align*}
The last equality follows from Theorem \ref{Uniqueness} since $\overline{h^\tau(z)} \in \mathbb{C}_R$ and satisfies (\ref{h_main_eqn}) with $\overline{z} \in \mathbb{C}_L$ instead of $z$.

Therefore, $-\mathbbm{i}\zeta = \overline{h_0(z)} = \underset{m \rightarrow \infty}{\lim}\overline{h^{\tau_m}(z)} = \underset{m \rightarrow \infty}{\lim}h^{\tau_m}(\overline{z}) = h_0(\overline{z}) = \mathbbm{i}\zeta$ so that $\zeta = 0$ and in turn $h_0(z) = 0$. 

Fix $z = -u + \mathbbm{i}v$ with $u > 0$. Recalling $I_1, I_2$ as defined in (\ref{realOf_h}), we have,
\begin{align*}
    &\Re(h^{\tau_m}) = c\Re(h^{\tau_m})\rho_2(ch^{\tau_m})I_2(h^{\tau_m}, H^{\tau_m}) + u I_1(h^{\tau_m}, H^{\tau_m})\\
    \implies& \underset{m \rightarrow \infty}{\lim} I_1(h^{\tau_m}, H^{\tau_m})= 0 \text{, using } (\ref{lessThanOne}) \text{ and } u > 0\\
    \implies & \underset{m \rightarrow \infty}{\lim}\int_0^{\infty}\dfrac{\lambda dH^{\tau_m}(\lambda)}{|-z +\lambda\rho(ch^{\tau_m})|^2} = 0 \label{rightBound}\tag{B.44}
\end{align*}

For arbitrary $M > 0$, choose $m \in \mathbb{N}$ such that $\tau_m > M$. Then noting the relationship between $H$ and $H^{\tau_m}$, we have 
\begin{align*}\tag{B.45}\label{tempResult}
    &\int_0^M\dfrac{\lambda dH(\lambda)}{|-z + \lambda\rho(ch^{\tau_m})|^2}
    \leq      \int_0^{\tau_m}\dfrac{\lambda dH(\lambda)}{|-z +\lambda\rho(ch^{\tau_m})|^2}
    = \int_0^{\infty}\dfrac{\lambda dH^{\tau_m}(\lambda)}{|-z +\lambda\rho(ch^{\tau_m})|^2}
\end{align*}

Since $\Re(\rho(ch^{\tau_m}) = \rho_2(ch^{\tau_m})\Re(ch^{\tau_m}) > 0$, we have for $0 \leq \lambda \leq M$,
\begin{align*}\label{integrandBound2}\tag{B.46}
    \frac{|\lambda|}{|-z + \lambda\rho(ch^{\tau_m})|^2}\leq \frac{|\lambda|}{(\Re(-z + \lambda\rho(ch^{\tau_m})))^2} \leq \frac{M}{u^2}
\end{align*}

From (\ref{tempResult}) and (\ref{rightBound}), we get
\begin{align*}\label{lim_is_zero_arbtry_M}\tag{B.47}
0 \leq \underset{m \rightarrow \infty}{\lim} \displaystyle \int_0^M\dfrac{\lambda dH(\lambda)}{|-z +\lambda\rho(ch^{\tau_m})|^2} \leq \underset{m \rightarrow \infty}{\lim}\int_0^{\infty}\dfrac{\lambda dH^{\tau_m}(\lambda)}{|-z +\lambda\rho(ch^{\tau_m})|^2} = 0    
\end{align*}

Now applying D.C.T. (because of \ref{integrandBound2}) on the first term in (\ref{tempResult}) and using (\ref{lim_is_zero_arbtry_M}) we get
\begin{align*}
0 = \underset{m \rightarrow \infty}{\lim} \displaystyle \int_0^M\dfrac{\lambda dH(\lambda)}{|-z +\lambda\rho(ch^{\tau_m})|^2} = \int_0^M\dfrac{\lambda dH(\lambda)}{|-z +\lambda\rho(0)|^2} = \frac{1}{|z|^2}\int_0^M \lambda dH(\lambda) 
\end{align*}

Since $M > 0$ is arbitrary, it follows that 
$\int_0^\infty \lambda dH(\lambda) = 0$, which implies that $H\{0\} = 1$. This contradicts the assumption that $H$ is not degenerate at 0, and therefore proves the claim that $\Re(h_0(z)) > 0$.    
\end{proof}

\begin{lemma} \label{Continuity}
    The solution to (\ref{h_main_eqn}) has a continuous dependence on H, the distribution function.
\end{lemma}
\begin{proof}
 For a fixed $c > 0$ and $z \in \mathbb{C}_L$, let $h, \ubar{h}$ be the unique numbers in $\mathbb{C}_R$ corresponding to distribution functions $H$ and $\ubar{H}$ respectively that satisfy (\ref{h_main_eqn}). Following \cite{PaulSilverstein2009}, we have 
\begin{align*}
    h - \ubar{h}
    = & \int\dfrac{\lambda dH(\lambda)}{-z + \lambda\rho(ch)} - \int\dfrac{\lambda d\ubar{H}(\lambda)}{-z + \lambda\rho(c\ubar{h})}\\
    =& \underbrace{\int\dfrac{\lambda d\{H(\lambda) - \ubar{H}(\lambda)\}}{-z +\lambda\rho(ch)}}_{:=T_1} + \int\dfrac{\lambda d\ubar{H}(\lambda)}{-z +\lambda\rho(ch)} - \int\dfrac{\lambda d\ubar{H}(\lambda)}{-z + \lambda\rho(c\ubar{h})}\\
    =& T_1 + \int\dfrac{\lambda^2(\rho(c\ubar{h})- \rho(ch))}{(-z + \lambda\rho(ch))(-z + \lambda\rho(c\ubar{h}))}d\ubar{H}(\lambda)\\
    =& T_1 + \int\dfrac{\dfrac{\lambda^2c(h-\ubar{h})}{(\mathbbm{i} +ch)(\mathbbm{i} +c\ubar{h})} + \dfrac{\lambda^2c(h-\ubar{h})}{(-\mathbbm{i} +ch)(-\mathbbm{i} +c\ubar{h})}}{(-z + \lambda\rho(ch))(-z + \lambda\rho(c\ubar{h}))}d\ubar{H}(\lambda)\\
    =& T_1 + (h-\ubar{h}) \underbrace{\int\dfrac{\dfrac{\lambda^2c}{(\mathbbm{i} +ch)(\mathbbm{i} +c\ubar{h})} + \dfrac{\lambda^2c}{(-\mathbbm{i} +ch)(-\mathbbm{i} +c\ubar{h})}}{(-z + \lambda\rho(ch))(-z + \lambda\rho(c\ubar{h}))}d\ubar{H}(\lambda)}_{:=\gamma}\\
    =& T_1 + (h-\ubar{h})\gamma
\end{align*}
Note that $\Re(\rho(ch)) = \rho_2(ch)\Re(ch) > 0$ and the integrand in $T_1$ is bounded by $1/\Re(\rho(ch))$. So by making $\ubar{H}$ closer to $H$, $T_1$ can be made arbitrarily small. Now, if we can show that $|\gamma| < 1$, this will essentially prove the continuous dependence of the solution to (\ref{h_main_eqn}) on H. 

\begin{align*}
    \gamma =& \underbrace{\int\dfrac{\dfrac{\lambda^2c}{(\mathbbm{i} +ch)(\mathbbm{i} +c\ubar{h})}}{(-z + \lambda\rho(ch))(-z + \lambda\rho(c\ubar{h}))}d\ubar{H}(\lambda)}_{:=G_1} + \underbrace{\int\dfrac{ \dfrac{\lambda^2c}{(-\mathbbm{i} +ch)(-\mathbbm{i} +c\ubar{h})}}{(-z + \lambda\rho(ch))(-z \lambda\rho(c\ubar{h}))}d\ubar{H}(\lambda)}_{:=G_2}\\
    =& G_1 + G_2
\end{align*}                                     

By $\Ddot{H}$older's Inequality we have,
\begin{align*}
    |G_1| &\leq \displaystyle \sqrt{\underbrace{\int\dfrac{c\lambda^2|\mathbbm{i} +ch|^{-2}d\ubar{H}(\lambda)}{|-z + \lambda\rho(ch)|^2}}_{:= P_1}}
    \sqrt{\underbrace{\int\dfrac{c\lambda^2|\mathbbm{i} +c\ubar{h}|^{-2}d\ubar{H}(\lambda)}{|-z + \lambda\rho(c\ubar{h})|^2}}_{:=P_2}} = \sqrt{P_1 \times P_2}
\end{align*}
From the definitions used in (\ref{realOf_h}), we have $|P_2| = c|\mathbbm{i} +c\ubar{h}|^{-2}I_2(\ubar{h}, \ubar{H})$
and 
\begin{align*}
    |P_1| =& c|\mathbbm{i} +ch|^{-2}\int\dfrac{\lambda^2 d\ubar{H}(\lambda)}{|-z + \lambda\rho(ch)|^2}\\
    & = c|\mathbbm{i} +ch|^{-2}\bigg(\underbrace{\int\dfrac{\lambda^2 d\{\ubar{H}(\lambda)-H(\lambda)\}}{|-z + \lambda\rho(ch)|^2}}_{:=K_1} + \int\dfrac{\lambda^2 dH(\lambda)}{|-z + \lambda\rho(ch)|^2}\bigg)\\
    & =c|\mathbbm{i} +ch|^{-2}K_1 + c|\mathbbm{i} +ch|^{-2}I_2(h, H)\\
    & < \epsilon + c|\mathbbm{i} +ch|^{-2}I_2(h, H)
\end{align*}

for some arbitrarily small $\epsilon > 0$. The last inequality follows since the integrand in $K_1$ is bounded by ${|\Re(\rho(ch))|^{-2}}$, we can arbitrarily control the first term by taking $\ubar{H}$ sufficiently close to $H$ in the Levy metric. The argument for bounding $|G_2|$ is exactly the same. 

Therefore we have $|G_1| < \sqrt{\epsilon + c|\mathbbm{i} +ch|^{-2}I_2(h, H)}\sqrt{c|\mathbbm{i} +c\ubar{h}|^{-2}I_2(\ubar{h}, \ubar{H})}$. \\
Similarly, we get $|G_2| < \sqrt{\epsilon + c|-\mathbbm{i} +ch|^{-2}I_2(h, H)}\sqrt{c|-\mathbbm{i} +c\ubar{h}|^{-2}I_2(\ubar{h}, \ubar{H})}$.

Thus, using the inequality $\sqrt{ac} +\sqrt{bd} \leq \sqrt{a+b}\sqrt{c+d}$ with equality iff $a=b=c=d=0$, we have \begin{align*}
    &|G_1| + |G_2|\\
    < & \sqrt{\epsilon + c|\mathbbm{i} +ch|^{-2}I_2(h, H)}\sqrt{c|\mathbbm{i} +c\ubar{h}|^{-2}I_2(\ubar{h}, \ubar{H})} +\\
    &\sqrt{\epsilon + c|-\mathbbm{i} +ch|^{-2}I_2(h, H)}\sqrt{c|-\mathbbm{i} +c\ubar{h}|^{-2}I_2(\ubar{h}, \ubar{H})}\\
    \leq & \sqrt{2\epsilon + (c|\mathbbm{i} +ch|^{-2} + c|-\mathbbm{i} +ch|^{-2})I_2(h, H)}\sqrt{(c|\mathbbm{i} +c\ubar{h}|^{-2} + c|-\mathbbm{i} +c\ubar{h}|^{-2})I_2(\ubar{h}, \ubar{H})}\\
    =& \sqrt{2\epsilon + c\rho_2(ch)I_2(h,H)}\sqrt{c\rho_2(c\ubar{h})I_2(\ubar{h}, \ubar{H})}
\end{align*}

From (\ref{lessThanOne}), we have $c\rho_2(ch)I_2(h, H) < 1$ and $c\rho_2(c\ubar{h})I_2(\ubar{h}, \ubar{H}) < 1$. By choosing $\epsilon > 0$ arbitrarily small, we finally have $|\gamma| = |G_1 + G_2| \leq |G_1| + |G_2| < 1$ for $\ubar{H}$ sufficiently close to H. This completes the proof.
\end{proof}

\section{Proofs related to Section \ref{sec:identity_covariance}}
\subsection{Results related to the density of the LSD in Section \ref{sec:identity_covariance}}
\begin{lemma}\label{SilvChoiResult2}
If a certain sequence $\{z_n\}_{n=1}^\infty \subset \mathbb{C}_L$ with $z_n \rightarrow \mathbbm{i}x$ satisfies $\underset{n \rightarrow \infty}{\lim} s_F(z_n) = \overline{s} \in \mathbb{C}_R$, then 
$s_F^0(x) := \underset{\mathbb{C}_L \ni z \rightarrow \mathbbm{i}x}{\lim} s_F(z)$ is well defined and equals $\overline{s}$.
\end{lemma}

\begin{proof}

Consider the unique pair $(z, s_F(z))$ for $z \in \mathbb{C}_L$. Define the function,
    $$z_F:s_F(\mathbb{C}_L) \rightarrow \mathbb{C}_L \hspace{5mm}  z_F(s) := \frac{1}{s}\bigg(\frac{2}{c}-1\bigg) + \frac{1}{\mathbbm{i}cs}\bigg(\frac{1}{\mathbbm{i}+cs} - \frac{1}{-\mathbbm{i}+cs}\bigg)$$
We can extend the domain of $z_F$ to the set $\mathbb{C}\backslash\{0, \pm {\mathbbm{i}}/{c}\}$ where it is analytic. Note that on $s_F(\mathbb{C}_L)$, $z_F$ coincides with the inverse mapping of $s_F$. Clearly $z_F$ is continuous at $\overline{s}$ as $\overline{s} \in \mathbb{C}_R \implies \overline{s} \not \in \{0, \pm \mathbbm{i}/c\}$. Therefore, $z_F(\overline{s}) = z_F(\underset{n \rightarrow \infty}{\lim}s_F(z_n)) = \underset{n \rightarrow \infty}{\lim}z_F(s_F(z_n)) = \underset{n \rightarrow \infty}{\lim}z_n = \mathbbm{i}x$.

Let $\{z_{1n}\}_{n=1}^\infty \subset \mathbb{C}_L$ be any another sequence such that $z_{1n} \rightarrow \mathbbm{i}x$. Since $\overline{s} \in \mathbb{C}_R$, we can choose an arbitrarily small $\epsilon$ such that $0 < \epsilon < \Re(\overline{s})$ and define $B := B(s_0; \epsilon)$\footnote{$B(x;r)$ indicates the open ball of radius $r$ centred at $x \in \mathbb{C}$}. $z_F$ being analytic and non-constant,  $z_F(B)$ is open by the Open Mapping Theorem and $\mathbbm{i}x \in z_F(B)$. So, for large $n$, $z_{1n} \in z_F(B)$. For these $z_{1n}$, there exists $s_{1n} \in B$ such that $z_F(s_{1n}) = z_{1n}$. By Theorem \ref{Uniqueness}, we must have $s_F(z_{1n}) = s_{1n} \in B$. Since $\epsilon > 0$ was arbitrary, the result follows.
\end{proof}

\begin{lemma}\label{DistributionParameters}
    For the quantities defined in (\ref{supportParameters}) and $\Tilde{r},\Tilde{q},\Tilde{d}$ defined in (\ref{RQD}), the following results hold.
    \begin{description}
        \item[1] $L_c < U_c$ \vspace{-2mm}
        \item[2] $d(x) < 0$ on $S_c$ and $d(x) \geq 0$ on $S_c^c\backslash\{0\}$
        \item[3] For $x\neq 0$, $r(x) = \mathbbm{i}\operatorname{sgn}(x)\bigg(-\dfrac{r_1}{|x|}  + \dfrac{r_3}{|x|^3}\bigg)$ and $q(x) = q_0 - \dfrac{q_2}{x^2}$
        \item[4] For $x\neq 0$, $d(x) = r^2(x) + q^3(x)$
    \end{description}
\end{lemma}
\begin{proof}
Consider the polynomial $g(x) = d_0x^4 - d_2x^2 + d_4$. Reparametrizing $y = x^2$, the two roots in $y$ are given by $R_{\pm}$ ((1) of \ref{supportParameters}). We start with the fact for any $c \in (0, \infty)$, the discriminant term is positive since 
\begin{align*}\label{discr_positive}\tag{C.1}
    d_2^2 - 4d_0d_4 = \bigg(\frac{4c+1}{9c^4}\bigg)^3 > 0    
\end{align*}
    
Now note that $\forall c \in (0, \infty)$, $R_{+}$ is positive for all values of c. In fact we have $$R_{+} = \frac{d_2 +\sqrt{d_2^2 - 4 d_0 d_4}}{2d_0} = \frac{1}{2}\bigg((2c^2+10c-1) + (4c+1)^{\frac{3}{2}}\bigg) > 0$$ 
    
However, $R_{-}$ is positive depending on the value of c. Note that
    \begin{align*}
            R_{-}=&\dfrac{d_2 - \sqrt{d_2^2 - 4 d_0 d_4}}{2d_0} > 0 \\
            \Leftrightarrow& d_2 > \sqrt{d_2^2 - 4d_0d_4} > 0 \text{, since } d_0 = 1/27c^2 > 0 \\
            \Leftrightarrow& 4d_0d_4 > 0 \Leftrightarrow d_4 > 0 \Leftrightarrow 1 - 2/c > 0 \Leftrightarrow c > 2
    \end{align*}
For $0 < c \leq 2$, $R_{-} \leq 0 < R_{+} \implies L_c < U_c$. For $c > 2$, we have $d_2 > 0$ and using (\ref{discr_positive}) implies 
\begin{align*}
\sqrt{d_2^2 - 4d_0d_4} < d_2 \implies \frac{d_2 - \sqrt{d_2^2 - 4d_0d_4}}{2d_0} < \frac{d_2 + \sqrt{d_2^2 - 4d_0d_4}}{2d_0}
\implies & R_- < R_+ \implies L_c < U_c    
\end{align*}

Therefore, $\forall c > 0$ $(L_c, U_c)$ is a valid interval in $\mathbb{R}$. This proves the first result.

Since $d_0 = 1/(27c^6) > 0$ $\forall$ $c > 0$, the polynomial $g(x)$ is a parabola (in $x^2$) with a convex shape. When $c > 2$, we have $0 < R_- < R_+$. In this case, $g(x) = 0$ when $x^2 = R_\pm$ and $g(x) < 0$ when $x^2 \in (R_{-}, R_{+})$. Thus $\forall x \in (-\sqrt{R_{+}}, -\sqrt{R_{-}}) \cup (\sqrt{R_{-}}, \sqrt{R_{+}})=S_c$, we have $g(x) < 0$. Similarly, for $0 < c \leq 2$, $g(x) < 0$ $\forall x \in (-\sqrt{R_{+}}, 0)\cup(0, \sqrt{R_{+}}) = S_c$. Therefore, for any $c > 0$, we have $g(x) < 0$ on the set $S_c$.    $g(x) \geq 0$ on $S_c^c \backslash\{0\}$ follows from the convexity of $g(\cdot)$ in $x^2$. This establishes the second result.

Let $x \neq 0$ and $\epsilon > 0$. Consider $z = -\epsilon + \mathbbm{i}x$. 
Using the definition of $R(z)$, $Q(z)$ from (\ref{RQD}),
 \begin{align*}
    &r(x) = \underset{\epsilon \downarrow 0}{\lim}R(-\epsilon + \mathbbm{i}x) = \underset{\epsilon \downarrow 0}{\lim} \dfrac{r_1}{-\epsilon + \mathbbm{i}x} + \dfrac{r_3}{(-\epsilon + \mathbbm{i}x)^3}
    = \frac{r_1}{\mathbbm{i}x} + \frac{r_3}{(\mathbbm{i}x)^3}
    = \mathbbm{i}\operatorname{sgn}(x)\bigg(-\frac{r_1}{|x|} + \frac{r_3}{|x|^3}\bigg) \label{value_of_rx}\tag{C.2}\\
    & \text{ and, }\\
    & q(x) = \underset{\epsilon \downarrow 0}{\lim}Q(-\epsilon + \mathbbm{i}x) = \underset{\epsilon \downarrow 0}{\lim} \bigg(q_0 + \frac{q_2}{(-\epsilon + \mathbbm{i}x)^2}\bigg) = q_0 - \frac{q_2}{x^2}
\end{align*}

This proves the third result. For the final result, note that $q_0^3 = d_0$, $d_2 = 3q_0^2q_2 + r_1^2$, $d_4 = 3q_0q_2^2 + 2r_1r_3$ and $q_2^3 + r_3^2 = 0$. Therefore for $x \neq 0$, we have
\begin{align*}
    r^2(x)+q^3(x)
    =& -\bigg(-\frac{r_1}{|x|} + \frac{r_3}{|x|^3}\bigg)^2 + \bigg(q_0 - \frac{q_2}{x^2}\bigg)^3\\
    =& q_0^3 + \frac{-3q_0^2q_2 - r_1^2}{x^2} + \frac{3q_0q_2^2 + 2r_1r_3}{x^4} + \frac{q_2^3 +r_3^2}{x^6}\\
    =& d_0 - \frac{d_2}{x^2} + \frac{d_4}{x^4} = d(x)
\end{align*}
\end{proof}

We state the following result (Theorem 2.2 of \cite{SilvChoi}) without proof. This result will be used to establish the continuity of the density function.
\begin{lemma}\label{continuousOnClosure}
    Let X be an open and bounded subset of $\mathbb{R}^n$, let Y be an open and bounded subset of $\mathbb{R}^m$, and let $f:\overline{X} \rightarrow Y$ be a function continuous on X. If, for all $x_0 \in \partial X$, $\underset{x \in X \rightarrow x_0}{\lim}f(x) = f(x_0)$, then f is continuous on all of $\overline{X}$.
\end{lemma}

\subsection{Proof of Theorem \ref{DensityDerivation}}\label{ProofDensityDerivation}
\begin{proof}
The density of $\Tilde{F}$ at $x \in \mathbb{R}$ (if it exists) is the same as that of $F$ at $\mathbbm{i}x$. To check for existence (and consequently derive the value), we employ the following strategy. We first show that $\underset{\epsilon \downarrow 0}{\lim}\Re(s_F(-\epsilon + \mathbbm{i}x))$ exists. Then by Lemma \ref{SilvChoiResult2}, the conditions of Proposition \ref{SilvChoiResult1} are satisfied implying existence of density at $x_0$. The value of the density is then extracted by using the formula in (\ref{inversion_density}).

Recall the definition of $r(x)$ and $q(x)$ from (\ref{supportParameters}). We will first show that for $x \in S_c$,
\begin{align*}\label{abs_of_rx}\tag{C.3}
-\frac{r_1}{|x|} + \frac{r_3}{|x|^3} > 0
\end{align*}

For $0 < c \leq 2$, we have $0 = L_c < U_c$ and from (\ref{RQD}),
\begin{align*}
  \frac{r_3}{r_1} = \frac{(c-2)^3}{9(c+1)} < 0  
\end{align*}
Thus $x \in S_c \implies x^2 > 0 > \dfrac{r_3}{r_1}$. 

For $c>2$, we have 
\begin{align*}
  0 < \frac{r_3}{r_1} = \frac{(c-2)^3}{9(c+1)} < \frac{1}{2}((2c^2 + 10c - 1) - (4c+1)^\frac{3}{2}) = L_c^2  
\end{align*}

Thus, $0 < \dfrac{r_3}{r_1} < L_c^2 < U_c^2$. Therefore $x \in S_c \implies x^2 > \dfrac{r_3}{r_1}$. In either case, since $r_1<0$ we have 
\begin{align*}\label{tempResults2}\tag{C.4}
    x^2 > \frac{r_3}{r_1} \implies r_1x^2 < r_3 \implies -\frac{r_1}{|x|} +\frac{r_3}{|x|^3} > 0 \implies |r(x)| = \mathbbm{i}\operatorname{sgn}(x)r(x)
\end{align*}
where the last equality follows from (\ref{value_of_rx}).

Having established this, we are now in a position to derive the value of the density. Without loss of generality, choose $x \in S_c$ such that $x > 0$.  We can do this since the limiting distribution is symmetric about 0 from Section \ref{Section4.5}. Consider $z = -\epsilon + \mathbbm{i}x$. The roots of (\ref{5B}) are given in (\ref{RootFormula}) in terms of quantities $S_0(z), T_0(z)$ that satisfy (\ref{S0T0}). Using (\ref{tempResults2}) and Lemma \ref{DistributionParameters}, we get
\begin{align*}\label{tempResults3}\tag{C.5}
& |r(x)|^2 > (\mathbbm{i}\operatorname{sgn}(x)r(x))^2 - q^3(x) = -(r^2(x) +q^3(x)) =-d(x) > 0\\
\implies& |r(x)| > \sqrt{-d(x)}
\end{align*}

Therefore $V_+(x) > V_-(x) > 0$. Now, let $s_0 := \mathbbm{i}(V_+)^\frac{1}{3}$ and $t_0 := -q(x)/s_0$ (note that $s_0 \neq 0$). Since $q(x) =q_0 - {q_2}/{x^2} > 0$ as $q_0 > 0, q_2 < 0$, both $s_0$ and $t_0$ are purely imaginary. First of all, observe that
\begin{align*}
    V_+(x)V_-(x) = |r(x)|^2 - (\sqrt{-d(x)})^2=-r^2(x) + d(x) = q^3(x)
\end{align*}
Therefore, we get 
\begin{align*}
    t_0^3 = -\frac{q^3(x)}{s_0^3} = \frac{V_+(x)V_-(x)}{\mathbbm{i}V_+(x)} = -\mathbbm{i}V_-(x)
\end{align*}
Finally we observe that $s_0, t_0$ satisfy the below relationship.
\begin{itemize}
    \item $s_0^3+t_0^3 = 2r(x) = \underset{\epsilon \downarrow 0}{\lim}2R(-\epsilon + \mathbbm{i}x) = \underset{\epsilon \downarrow 0}{\lim}\bigg(S_0^3(-\epsilon + \mathbbm{i}x) + T_0^3(-\epsilon + \mathbbm{i}x)\bigg)$
    \item $s_0t_0 = -q(x) = -\underset{\epsilon \downarrow 0}{\lim}Q(-\epsilon + \mathbbm{i}x) = \underset{\epsilon \downarrow 0}{\lim}\bigg(S_0(-\epsilon + \mathbbm{i}x)T_0(-\epsilon + \mathbbm{i}x)\bigg)$
\end{itemize}

From the above it turns out that
\begin{align*}
\bigg\{\underset{\epsilon \downarrow 0}{\lim}S_0^3(-\epsilon + \mathbbm{i}x), \underset{\epsilon \downarrow 0}{\lim}T_0^3(-\epsilon + \mathbbm{i}x)\bigg\} =\{s_0^3, t_0^3\}    
\end{align*}

This leaves us with the following three possibilities.
 \begin{align*}
     \bigg\{\underset{\epsilon \downarrow 0}{\lim}S_0(-\epsilon + \mathbbm{i}x), \underset{\epsilon \downarrow 0}{\lim} T_0(-\epsilon + \mathbbm{i}x)\bigg\} =\{s_0, t_0\} \text{ or } \{\omega_1 s_0, \omega_2 t_0\} \text{ or } \{\omega_2 s_0, \omega_1 t_0\}
 \end{align*}
 
Fortunately, the nature of (\ref{RootFormula}) is such that all three choices lead to the same set of roots denoted by $\{m_j(-\epsilon + \mathbbm{i}x)\}_{j=1}^3$.
Using (\ref{RootFormula}) and shrinking $\epsilon$ to 0, we find in the limit
\begin{equation*}
    \left\{ \begin{aligned} 
M_1(x) := \underset{\epsilon \downarrow 0}{\lim}m_1(-\epsilon + \mathbbm{i}x) &= -\dfrac{1-2/c}{3\mathbbm{i}x} + s_0 + t_0\\
M_2(x) := \underset{\epsilon \downarrow 0}{\lim}m_2(-\epsilon + \mathbbm{i}x) &= -\dfrac{1-2/c}{3\mathbbm{i}x} + \omega_1s_0 + \omega_2t_0\\
M_3(x) := \underset{\epsilon \downarrow 0}{\lim}m_3(-\epsilon + \mathbbm{i}x) &= -\dfrac{1-2/c}{3\mathbbm{i}x} + \omega_2s_0 + \omega_1t_0
\end{aligned} \right.
\end{equation*}

We have $\underset{\epsilon \downarrow 0}{\lim} \Re  \bigg(\dfrac{2/c-1}{3\mathbbm{i}x}\bigg)= 0$ and $\Re(s_0) = 0 = \Re(t_0)$. Therefore, $\Re(M_1(x)) = 0$.
Focusing on the second root,
\begin{align*}
     \Re(M_2(x))
    =\Re(\omega_1s_0 + \omega_2t_0
    =& \Re \bigg(-\frac{s_0+t_0}{2} + \mathbbm{i}\frac{\sqrt{3}}{2}(s_0 - t_0)\bigg)\\
    =& \frac{\sqrt{3}}{2}\Im(t_0 - s_0) = \frac{\sqrt{3}}{2}\bigg((V_-)^\frac{1}{3} -  (V_+)^\frac{1}{3}\bigg) < 0
\end{align*}
and similarly, 
\begin{align*}
     \Re(M_3(x))
    =\Re(\omega_2s_0 + \omega_1t_0)
    =& \Re \bigg(-\frac{s_0+t_0}{2} - \mathbbm{i}\frac{\sqrt{3}}{2}(s_0 - t_0)\bigg)\\
    =& \frac{\sqrt{3}}{2}\Im(s_0 - t_0) = \frac{\sqrt{3}}{2}\bigg((V_+)^\frac{1}{3} - (V_-)^\frac{1}{3}\bigg) > 0
\end{align*}

To summarise till now, we evaluated the roots of (\ref{5B}) at a sequence of complex numbers $-\epsilon + \mathbbm{i}x$ in the left half of the argand plane close to the point $\mathbbm{i}x$ on the imaginary axis. This leads to three sequences of roots $\{m_j(-\epsilon + \mathbbm{i}x)\}_{j=1}^3$ of which only one has real part converging to a positive number. Therefore, for $x \in S_c \cap \mathbb{R}_+$, $s_F(-\epsilon + \mathbbm{i}x) \rightarrow M_3(x)$ as $\epsilon \downarrow 0$ by Theorem \ref{Uniqueness}. So, from (\ref{inversion_density}) and the symmetry about 0, the density at $x \in S_c$ is
\begin{align*}
    f_c(x) &= \dfrac{1}{\pi}\underset{\epsilon \downarrow 0}{\lim} \Re  (s_F(-\epsilon + \mathbbm{i}x)) = \frac{\sqrt{3}}{2\pi}\bigg(V_+(x))^\frac{1}{3} -V_-(x)^\frac{1}{3}\bigg)
\end{align*}

Now we evaluate the density when $x \in S_c^c\backslash\{0\}$. Without loss of generality, let $x > 0$ since the distribution is symmetric about 0. From Lemma \ref{DistributionParameters}, $d(x) \geq 0$ in this case. Noting that $r(x) = -\mathbbm{i}|r(x)|$ from (\ref{tempResults2}), define $s_0 := (\sqrt{d(x)} -\mathbbm{i}|r(x)|)^\frac{1}{3}$ be any cube root and $t_0 := -{q(x)}/{s_0}$. Note that $s_0 \neq 0$ since $d(x) \geq 0$ and $|r(x)| > 0$.
Then, \begin{align*}
    t_0^3 = -\frac{q^3(x)}{s_0^3}=-\frac{d(x)-r^2(x)}{s_0^3} = -\frac{(\sqrt{d(x)} -\mathbbm{i}|r(x)|)(\sqrt{d(x)} + \mathbbm{i}|r(x)|)}{\sqrt{d(x)} - \mathbbm{i}|r(x)|} = -\sqrt{d(x)} - \mathbbm{i}|r(x)|
\end{align*}
Therefore, we have 
\begin{align*}
 s_0^3+t_0^3 = 2r(x) = \underset{\epsilon \downarrow 0}{\lim}2R(-\epsilon + \mathbbm{i}x); \hspace{5mm} s_0t_0 = -q(x) = -\underset{\epsilon \downarrow 0}{\lim}Q(-\epsilon + \mathbbm{i}x)
\end{align*}

Therefore using (\ref{RootFormula}) to find the three roots of (\ref{5B}) and shrinking $\epsilon > 0$ to 0, we get in the limit 

\begin{equation*}
    \left\{ \begin{aligned} 
M_1(x) := \underset{\epsilon \downarrow 0}{\lim}m_1(-\epsilon + \mathbbm{i}x) &= -\dfrac{1-2/c}{3x} + s_0 + t_0\\
M_2(x) := \underset{\epsilon \downarrow 0}{\lim}m_2(-\epsilon + \mathbbm{i}x) &= -\dfrac{1-2/c}{3x} + \omega_1s_0 + \omega_2t_0\\
M_3(x) := \underset{\epsilon \downarrow 0}{\lim}m_3(-\epsilon + \mathbbm{i}x) &= -\dfrac{1-2/c}{3x} + \omega_2s_0 + \omega_1t_0
\end{aligned} \right.
\end{equation*}

Observe that $\bigg|\sqrt{d(x)}-\mathbbm{i}|r(x)|\bigg|^2 = d(x) + |r(x)|^2 = d(x) - r^2(x) = q^3(x)$. Therefore,
$$t_0 = -\dfrac{q(x)}{s_0} = -\dfrac{q(x)\overline{s_0}}{|s_0|^2}
    = -\dfrac{q(x)\overline{s_0}}{|(\sqrt{d(x)} -\mathbbm{i} |r(x)|)|^\frac{2}{3}}
    = -\dfrac{q(x)\overline{s_0}}{|q(x)^3|^\frac{1}{3}}
    = -\dfrac{q(x)\overline{s_0}}{q(x)} = -\overline{s_0}$$
using the fact that $q(x) > 0$ for $x \neq 0$. Therefore, $\Re(s_0)=-\Re(t_0)$ and $\Im(s_0) =\Im(t_0)$. In particular, $s_0 + t_0 = 2\mathbbm{i}\Im(s_0)$ and $s_0 - t_0 = 2\Re(s_0)$.
This leads to the following observations.
\begin{align*}
\Re  (M_1(x)) =& \Re \{s_0+t_0\} = 0\\
\Re  (M_2(x)) =& \Re  \{-\frac{1}{2}(s_0+t_0) + \mathbbm{i}\frac{\sqrt{3}}{2}(s_0 - t_0)\} = 0\\
\Re  (M_3(x)) =& \Re  \{-\frac{1}{2}(s_0+t_0) - \mathbbm{i}\frac{\sqrt{3}}{2}(s_0 - t_0)\} = 0    
\end{align*}

So when $x \in S_c^c\backslash\{0\}$, all three roots (in particular, the one that agrees with the Stieltjes transform) of (\ref{5B}) at $z = -\epsilon + \mathbbm{i}x$ have real component shrinking to 0 as $\epsilon \downarrow 0$. Therefore, by (\ref{inversion_density}) and the symmetry about $0$
\begin{equation*}
    \begin{split}
    f_c(x) &= -\dfrac{1}{\pi}\underset{\epsilon \downarrow 0}{\lim} \Re  (s_F(-\epsilon + \mathbbm{i}x)) = 0\\
    \end{split}
\end{equation*}
So, the density is positive on $S_c$ and zero on $S_c^c\backslash\{0\}$. 

Finally we check if the density can exist at $x=0$ for $0 < c < 2$. For this we evaluate $L:= \underset{\epsilon \downarrow 0}{\lim} \Re(s_F(-\epsilon))$.
\begin{align*}
    &\frac{1}{s_F(-\epsilon)} = -(-\epsilon) + \frac{1}{\mathbbm{i} +cs_F(-\epsilon)} + \frac{1}{-\mathbbm{i} +  s_F(-\epsilon)}\\
    \implies &\frac{1}{L} = \frac{1}{\mathbbm{i} +cL} + \frac{1}{-\mathbbm{i} +cL}  = \frac{2cL}{1 + c^2L^2}\\
    \implies &2cL^2 = 1 + c^2L^2\\
    \implies & \underset{\epsilon \downarrow 0}{\lim} s_F(-\epsilon)= \frac{1}{\sqrt{2c-c^2}}  \text{, considering the positive root as } s_F\\
    & \text{ is a Stieltjes Transform of a measure on the imaginary axis}
\end{align*}

Therefore, when $0 < c < 2$, $f_c(0) = \dfrac{1}{\pi\sqrt{2c-c^2}}$.

Consider the case $0 < c < 2$. We saw that $f_c(x) = 0$ for $x \in S_c^c$. So, we need to show the continuity of $f_c$ in $S_c$. When $0<c<2$, $\underset{\epsilon \downarrow 0}{\lim}\Re(s_F(-\epsilon + \mathbbm{i}x))$ exists for all $x \in \mathbb{R}$. In particular, when $x \in S_c$, $\underset{\epsilon \downarrow 0}{\lim}\Re(s_F(-\epsilon + \mathbbm{i}x)) > 0$. For an arbitrary $x_0 \in S_c$, take an open bounded set $E \subset \mathbb{C}_L$ and choose $K > 0$ such that
$$\mathbbm{i}x_0 \in (-\mathbbm{i}K, \mathbbm{i}K) \subset \partial E$$ Then the below function
$$s_F^0: \overline{E} \rightarrow \mathbb{R}; \hspace{5mm} s_F^0(z_0) = \underset{E \ni z \rightarrow z_0}{\lim}\Re(s_F(z))$$ 
is well defined due to Lemma \ref{SilvChoiResult2}. It
is continuous on $E$ due to the continuity of $\Re(s_F)$ on $\mathbb{C}_L$ and satisfies the conditions of Lemma \ref{continuousOnClosure} by construction. Hence, the continuity of $s_F^0$ and of $f_c$ at $x_0$ is immediate.

Now consider the case when $c \geq 2$. As before, we only need to show the continuity of $f_c$ at an arbitrary $x_0 \in S_c$. Note that $x_0$ cannot be 0 as $0 \not \in S_c$. We already proved that $\underset{\epsilon \downarrow 0}{\lim}\Re(s_F(-\epsilon + \mathbbm{i}x_0)) > 0$. Construct an open bounded set $E \subset \mathbb{C}_L$ such that 
$$\bigg(-\frac{3\mathbbm{i}|x_0|}{2}, -\frac{\mathbbm{i}|x_0|}{2}\bigg) \cup \bigg(\frac{\mathbbm{i}|x_0|}{2}, \frac{3\mathbbm{i}|x_0|}{2}\bigg) \subset \partial E$$
A similar argument establishes the continuity of $f_c$ at $x_0 \neq 0$.
\end{proof}

\subsection{Ancillary results related to the Stieltjes Transform in Section \ref{sec:identity_covariance}}
For the results in this subsection, we denote $\mathbb{Q}_1:= \{u + \mathbbm{i}v: u > 0, v > 0\}$ and $\mathbb{Q}_2 := \{-u+\mathbbm{i}v: u > 0, v > 0\}$

\begin{proposition}\label{prop_C.4}
    For $j \in \{1,2,3\}$, $m_j(z)$ cannot be purely imaginary for $z \in \mathbb{C}_L$ and for $z \in \mathbb{Q}_2$, $m_j(z)$ cannot be purely real.
\end{proposition}

\begin{proof}
    First of all $m_j(z)$ cannot be equal to 0 because then we have $1 = 0$ from (\ref{5B}).
    
   Suppose $\exists z = u + \mathbbm{i}v \in \mathbb{C}_L$ such that $\Re(m_j(z)) = 0$ for some $j \in \{1,2,3\}$. Let $m_j(z) = \mathbbm{i}m$ where $m \neq 0$. Since $m_j(z)$ satisfies (\ref{5A}),
    \begin{align*}
            &\dfrac{1}{m_j(z)} = -z + \dfrac{1}{\mathbbm{i} +cm_j(z)} + \dfrac{1}{-\mathbbm{i} +cm_j(z)}\\
            \implies &u + \mathbbm{i}v = -\dfrac{1}{\mathbbm{i}m}+ \dfrac{1}{\mathbbm{i} +  \mathbbm{i}cm} + \dfrac{1}{-\mathbbm{i} +  \mathbbm{i}cm}\\
            \implies &u = -\mathbbm{i}v -\dfrac{1}{\mathbbm{i}m}+ \dfrac{1}{\mathbbm{i} +  \mathbbm{i}cm} + \dfrac{1}{-\mathbbm{i} +  \mathbbm{i}cm}
    \end{align*}
Note that $m \neq \pm 1/c$ from the discussion following (\ref{5A}). So the RHS of the above equation is purely imaginary but the LHS is purely real. This implies that $u = 0$ which leads to a contradiction since $z \in \mathbb{C}_L$.

For the other part, suppose $\exists z = u + \mathbbm{i}v \in \mathbb{Q}_2$ such that $\Im(m_j(z)) = 0$ for some $j \in \{1,2,3\}$. Let $m_j(z) = m$ where $m \in \mathbb{R}\backslash\{0\}$. Since $m_j(z)$ satisfies (\ref{5B}), we have
\begin{equation*}
    \begin{split}
            & c^2zm^3 + (c^2-2c)m^2 +zm + 1 = 0\\
            \implies & (c^2zm^3 + zm) + (c^2-2c)m^2+1=0\\
            \implies & m(u+\mathbbm{i}v)(c^2m^2 + 1) + (c^2-2c)m^2+1=0\\
            \implies & um(c^2m^2+1) + (c^2-2c)m^2+1 =-\mathbbm{i}vm(c^2m^2+1)\\
            \implies & vm(c^2m^2+1) = 0\\
    \end{split}
\end{equation*}
This leads to a contradiction as $v \neq 0$, $m \neq 0$ and clearly $c^2m^2+1 \neq 0$
\end{proof}

\begin{proposition}\label{continuity_in_c}
    Keeping $z \in \mathbb{C}_L$ fixed, the Stieltjes transform s(z, c) is continuous in $c > 0$. 
\end{proposition} 

\begin{proof}

From Theorem \ref{t3.1}, one of the roots of (\ref{5B}) is a Stieltjes transform and hence analytic. Let us denote this functional root as $s(z, c)$. Let $z_0 \in \mathbb{C}_L$ be fixed and $\{c_n\}_{n=1}^\infty \subset \mathbb{R}_{+}$ be such that $\underset{n \rightarrow \infty}{\lim}c_n = c_0 \in (0, \infty)$. Let $s_n := s(z_0, c_n)$ be the corresponding functional root (with positive real component) of (\ref{5B}) with $c_n$ instead of $c$. Then $\forall n \in \mathbb{N}$, $s_n$ satisfies
\begin{equation*}\tag{C.6}\label{C.6}
    c_n^2z_0s_n^3 + (c_n^2-2c_n)s_n^2 +z_0s_n + 1 = 0
\end{equation*}
Now, $|s_n|\leq 1/|\Re(z_0)|$ on account of being a Stieltjes transform at $z_0$ of a probability measure on the imaginary axis. Therefore, every subsequence has a further subsequence that converges. For such a convergent subsequence $\{s_{n_k}\}_{k=1}^\infty$, let the sub-sequential limit be $l(z_0)$. Then $\Re(l(z_0)) = \underset{n \rightarrow \infty}{\lim}\Re(s_n(z_0)) \geq 0$. We claim that $\Re(l(z_0)) > 0$. Taking limit as $n \rightarrow \infty$ in (\ref{C.6}), we get 
$$c_0^2z_0l^3(z_0) + (c_0^2-2c_0)l^2(z_0) + z_0l(z_0) + 1 = 0$$

Thus $l(z_0)$ satisfies (\ref{5B}) with $c_0, z_0$ in place of $c, z$ respectively. Therefore, by Proposition \ref{prop_C.4}, we must have $\Re(l(z_0))>0$. By Theorem (\ref{Uniqueness}), we conclude that any subsequential limit of $\{s_n\}_{n=1}^\infty$ must be the same and satisfies (\ref{5B}) with $c_0, z_0$ instead of $c, z$ respectively. Therefore, $\underset{n \rightarrow \infty}{\lim}s_n = \underset{n \rightarrow \infty}{\lim}s(z_0, c_n) = s(z_0, c)$. Therefore, $s(z, c)$ is continuous in $c$.
\end{proof}

\begin{proposition}\label{Q2_to_Q1}
For $z \in \mathbb{Q}_2$, there always exists one root of (\ref{5B}) that lies in $\mathbb{Q}_1$. In fact this functional root is the required Stieltjes Transform.
\end{proposition}

\begin{proof}
We first prove the result for $c=2$. In this case, $r_1 = -1/8; r_3 = 0; q_0 = 1/12; q_2 = 0$ from (\ref{RQD}). Fix an arbitrary $z \in \mathbb{Q}_2$. Then $\theta_z := \operatorname{Arg}(z) \in (\frac{\pi}{2}, \pi)$. By (\ref{RQD}), we have
\begin{align*}
    Q(z) =& q_0 + \frac{q_2}{z^2} = q_0\label{Qz} \tag{C.7}\\
    R(z) =& \dfrac{r_1}{z} = \dfrac{r_1}{|z|}e^{-\mathbbm{i}\theta_z}= -\dfrac{1}{8|z|}(\cos{\theta_z} - \mathbbm{i}\sin{\theta_z}) = -\dfrac{1}{8|z|} (\cos{\theta_z} - \mathbbm{i}\sin{\theta_z}) \in \mathbb{Q}_1  \label{Rz} \tag{C.8}\\
    D(z) =& R^2(z) + Q^3(z) = \frac{r_1^2}{|z|^2}e^{-2\mathbbm{i}\theta_z} + q_0^3 = \bigg(q_0^3 + \frac{r_1^2\operatorname{cos}(2\theta_z)}{|z|^2}\bigg) - \mathbbm{i}\bigg(\frac{r_1^2\operatorname{sin}(2\theta_z)}{|z|^2}\bigg) \in \mathbb{Q}_1 \text{ or } \mathbb{Q}_2  \label{Dz} \tag{C.9}
\end{align*}

(\ref{Rz}) and (\ref{Dz}) follows since $\cos(\theta_z) < 0$, $\sin{(\theta_z)} > 0$ and $\sin{(2\theta_z)} < 0$. Therefore, 
$\operatorname{Arg}(D(z)) := \theta_D \in (0, \pi)$. Since $\theta_D/2 \in (0, \pi/2)$, it is clear that
\begin{align*}\label{rootDz}\tag{C.10}
\sqrt{D(z)} = \sqrt{|D(z)|}e^{\frac{1}{2}\mathbbm{i}\theta_D} \in \mathbb{Q}_1    
\end{align*}
Finally, let us define
 \begin{align*}\label{Vplusminus}\tag{C.11}
 &V_{\pm}(z) := R(z) \pm \sqrt{D(z)}; \hspace{10mm} \theta_{\pm} := \operatorname{Arg(V_\pm(z))}\\
 \implies& V_{\pm}=|V_{\pm}|e^{\mathbbm{i}\theta_{\pm}}
 \end{align*}

For simplicity, we will henceforth refer to $R, Q, D, V_{\pm}$ without the explicit dependence on z. As $R$ and $\sqrt{D}$ both are in $\mathbb{Q}_1$, we have the following results using (\ref{Vplusminus}),
\begin{align*}\label{tempResults1}\tag{C.12}
V_{+} \in \mathbb{Q}_1; \hspace{5mm} |V_{+}| > |V_{-}|; \hspace{5mm} \theta_{+} \in (0, \pi/2)
\end{align*}

To derive the roots of (\ref{5B}), we define two quantities below that satisfy (\ref{S0T0}).
\begin{align*}\label{choiceOfS0T0}\tag{C.13}
s_0 :=& |V_{+}|^{\frac{1}{3}}e^{\frac{1}{3}\mathbbm{i}\theta_+}; \hspace{10mm}
t_0 := -Q/s_0 = -q_0|V_+|^{-\frac{1}{3}}e^{-\frac{1}{3}\mathbbm{i}\theta_+} \text{, where } q_0 = 1/12
\end{align*}

Then $s_0$ is cube root of $V_+$ by definition. Moreover, $t_0$ is a cube root of $V_{-}$ since 
\begin{align*}\label{T0cube_equal_Vminus}\tag{C.14}
 &t_0^3 = \frac{-Q^3}{s_0^3} = \frac{(R + \sqrt{R^2+Q^3})(R-\sqrt{R^2+Q^3})}{R + \sqrt{R^2 + Q^3}} = V_{-} \text{, note } s_0^3 = V_+ \in \mathbb{Q}_1 (\neq 0)\\
 \implies& |t_0| = |V_-|^\frac{1}{3}
\end{align*}
Note that from (\ref{tempResults1}), (\ref{choiceOfS0T0}) and (\ref{T0cube_equal_Vminus}), we have $|s_0| > |t_0|$. Also (\ref{choiceOfS0T0}) implies $\operatorname{Arg}(s_0) = - \operatorname{Arg}(t_0)$. As intended, $s_0, t_0$ satisfy
 $$s_0^3 + t_0^3=V_+ + V_- = 2R; \hspace{10mm} s_0t_0=-Q$$

\textbf{First Root:}
By (\ref{RootFormula}), the first root $m_1 = s_0 + t_0$. Since $\theta_+ \in (0, \pi/2)$, (\ref{choiceOfS0T0}) implies $\Im(s_0), \Im(t_0) > 0 \implies \Im(s_0 + t_0) > 0$. Also, $\Re(s_0) > 0, \Re(t_0) < 0$ but since $|s_0|>|t_0|$ and $\operatorname{Arg}(s_0) = -\operatorname{Arg}(t_0)$ from (\ref{choiceOfS0T0}), we conclude $\Re(s_0 + t_0) > 0$ and thus $m_1 \in \mathbb{Q}_1$.

\textbf{Second Root:}
By (\ref{RootFormula}), the second root is 
\begin{align*}\label{secondRoot}\tag{C.15}
m_2 =& \omega_1s_0 + \omega_2t_0\\
=& |V_+|^\frac{1}{3}e^{\frac{\mathbbm{i}(2\pi+\theta_+)}{3}} + |V_-|^\frac{1}{3}e^{\frac{\mathbbm{i}(4\pi-\theta_+)}{3}}\\
=& (|V_+|^\frac{1}{3} - |V_-|^\frac{1}{3})e^{\frac{\mathbbm{i}(2\pi+\theta_+)}{3}}
\end{align*}
Since $\theta_+ \in (0, \pi/2)$, we have $\cos(\frac{2\pi+\theta_+}{3}) < 0$ and $\sin(\frac{2\pi+\theta_+}{3}) > 0$. Thus we have $m_2 \in \mathbb{Q}_2$.\\
\textbf{Third Root:}
By (\ref{RootFormula}), the last root is 
\begin{align*}\label{thirdRoot}\tag{C.16}
m_3 =& \omega_2s_0 + \omega_1t_0\\
=& |V_+|^\frac{1}{3}e^{\frac{\mathbbm{i}(4\pi+\theta_+)}{3}} + |V_-|^\frac{1}{3}e^{\frac{\mathbbm{i}(2\pi-\theta_+)}{3}}\\
=& (|V_+|^\frac{1}{3} - |V_-|^\frac{1}{3})e^{\frac{\mathbbm{i}(4\pi+\theta_+)}{3}}
\end{align*}
Since $\theta_+ \in (0, \pi/2)$, we have $\cos(\frac{4\pi+\theta_+}{3}) < 0$ and $\sin(\frac{4\pi+\theta_+}{3}) < 0$. Thus we have $m_3 \in \mathbb{Q}_3$.

Thus for $c=2$, it is clear that $s_F = m_1$ is the required Stieltjes Transform. When $c \neq 2$, by the continuity of $s(z, c)$ in $c$ from Proposition \ref{continuity_in_c}, for values of $c$ ``close" to $2$, $s(z, c) \in \mathbb{Q}_1$ will still hold keeping $z \in \mathbb{Q}_2$ fixed. However, we claim that $s(z, c)$ must live in $\mathbb{Q}_1$. Because if for a certain $c \neq 2$, say $\Re(s(z, c)) < 0$, then by Intermediate Value Theorem, there exists $c_1, c_2, c_3 > 0$ such that $\Re(s(z, c_1)) > 0$, $\Re(s(z, c_2)) < 0$ and $\Re(s(z, c_3)) = 0$. But this contradicts Proposition \ref{prop_C.4}. We arrive at another contradiction if we assume $\Im(s(z, c)) < 0$ for some $c \neq 2$. Therefore for all values of $c > 0$, $s(z, c) \in \mathbb{Q}_1$ for $z \in \mathbb{Q}_2$. 
\end{proof}

\section{Proofs related to Section \ref{sec:Hermitian}}\label{sec:AppendixD}

In this section we provide a proof of Theorem \ref{t6.1}. We denote $S_n^+ = \frac{1}{n}\Sigma_n^\frac{1}{2}(Z_1Z_2^* + Z_2Z_1^*)\Sigma_n^\frac{1}{2}$ as $S_n$ for simplicity. The central objects our work such as $Q(z), s_n(z), h_n(z)$ are derived after incorporating this change. The proof sketch is exactly the same as in Section \ref{ProofSketch}. As stated earlier, most of the proofs are nearly identical to the ones we did for the skew-Hermitian version. The most important change is that bounds of important quantities are now characterized in terms of their imaginary components instead of their real components. Moreover, unless explicitly mentioned, the domain of all functions such as $Q, s_n, h_n$ will now be $\mathbb{C}^+$ instead of $\mathbb{C}_L$. Also, by Stieltjes Transform, we mean the transform of a measure on the real line.

We define a function analogous to the function $\rho(\cdot)$ (defined in (\ref{definingRho})) that we used throughout Section \ref{sec:general_covariance}.
\begin{align*}
    \sigma(z) = \frac{1}{1 + z} + \frac{1}{-1 + z} \text{, } z \not \in \{1, -1\}\label{defining_sigma}\tag{D.1}
\end{align*}
\begin{align*}
    \sigma_2(z) = \frac{1}{|1 + z|^2} + \frac{1}{|-1 + z|^2} \text{, } z \not \in \{1, -1\}\label{defining_sigma2}\tag{D.2}
\end{align*}
Then as before for $z \not \in \{1,-1\}$, we have
\begin{align*}\label{ImaginaryOfSigma}\tag{D.3}
    \Im(\sigma(z)) = \frac{\Im(\overline{1+z})}{|1+z|^2}+\frac{\Im(\overline{1-z})}{|1-z|^2} = -\Im(z)\sigma_2(z)
\end{align*}

REMARK: Note that $\sigma(-\overline{z}) = -\overline{\sigma(z)}$ and $\sigma$ is analytic in any open set not containing $\pm 1$. Also $\sigma_2(z) > 0$ in its domain. Now we prove the unique solvability of (\ref{6.2}).

\subsection{Proof of Uniqueness}
\begin{theorem}\label{Uniqueness_Im}
There exists at most one solution to the following equation within the class of functions that map $\mathbb{C}^+$ to $\mathbb{C}^+$.
    $$h(z) = \displaystyle \int \dfrac{\lambda dH(\lambda)}{-z + \dfrac{\lambda}{1 +ch(z)} + \dfrac{\lambda}{-1 +ch(z)}} = \int \dfrac{\lambda dH(\lambda)}{-z + \lambda \sigma(ch(z))}$$
where H is any probability distribution function such that $\operatorname{supp}(H) \subset \mathbb{R}_{+}$ and $H \neq \delta_0$.
\end{theorem}

\begin{proof}
    Suppose for some $z = u + \mathbbm{i}v \in \mathbb{C}^+$, $\exists$ $h_1, h_2 \in \mathbb{C}^+$ such that for $j \in \{1,2\}$, we have $$h_j = \displaystyle \int \dfrac{\lambda dH(\lambda)}{-z + \lambda \sigma(ch_j)}$$
    Further let $\Re(h_j) = h_{j1}, \Im(h_j) = h_{j2}$ where $h_{j2} > 0$ by assumption for $j \in \{1,2\}$. Using (\ref{ImaginaryOfSigma}), we have  
\begin{align*}
        & h_{j2} = \Im(h_j)
         = \displaystyle\int \dfrac{\lambda\Im(\overline{-z + \lambda\sigma(ch_j)})dH(\lambda)}{|-z + \lambda \sigma(ch_j)|^2} = \int\dfrac{v\lambda + \lambda^2 [\sigma_2(ch_j)\Im(ch_j)]}{|-z + \lambda\sigma(ch_j)|^2}dH(\lambda)\\
\implies & h_{j2} = vJ_1(h_j, H) + ch_{j2}\sigma_2(ch_j)J_2(h_j, H) \label{ImOf_h}\tag{D.4}\\           
\text {where } J_k(h_j, H) :=& \int \dfrac{\lambda^k dH(\lambda)}{|-z + \lambda \sigma(ch_j)|^2} \text{ for } k \in \{1,2\}
\end{align*}

Note that $J_k(h_j, H) > 0, k \in \{1,2\}$ due to the conditions on H. Since $h_{j2} > 0$ and $v > 0$, using (\ref{ImOf_h}), we must have 
\begin{align*}\tag{D.5}\label{lessThanOne_Im}
 c\sigma_2(ch_j)J_2(h_j, H) < 1
\end{align*}
    Then we have
\begin{equation*}
    \begin{split}
         h_1 - h_2
        =& \displaystyle\int \dfrac{(\sigma(ch_2) - \sigma(ch_1))\lambda^2}{[-z + \lambda\sigma(ch_1)][-z + \lambda\sigma(ch_2)]}dH(\lambda)\\
        =& (h_1-h_2)\displaystyle\int \dfrac{\dfrac{c\lambda^2}{(1 +ch_1)(1 +ch_2)} + \dfrac{c\lambda^2}{(-1 +ch_1)(-1 +ch_2)}}{[-z + \lambda\sigma(ch_1)][-z + \lambda\sigma(ch_2)]}dH(\lambda)\\
    \end{split}
\end{equation*}

By $\Ddot{H}$older's inequality, we have $|h_1-h_2| \leq |h_1-h_2|(T_1 + T_2)$ where 

\begin{align*}
T_1 =& \displaystyle\sqrt{\int\dfrac{c|1 +ch_1|^{-2}\lambda^2dH(\lambda)}{|-z + \lambda\sigma(ch_1)|^2}}\sqrt{\int\dfrac{c|1 +ch_2|^{-2}\lambda^2dH(\lambda)}{|-z + \lambda\sigma(ch_2)|^2}}\\    
=& \sqrt{c|1 +ch_1|^{-2}J_2(h_1, H)}\sqrt{c|1 +ch_2|^{-2}J_2(h_2, H)}\\
&\text{and}\\
T_2 =& \displaystyle\sqrt{\int\dfrac{c|-1 +ch_1|^{-2}\lambda^2dH(\lambda)}{|-z + \lambda\sigma(ch_1)|^2}}\sqrt{\int\dfrac{c|-1 +ch_2|^{-2}\lambda^2dH(\lambda)}{|-z +\lambda\sigma(ch_2)|^2}}\\
=& \sqrt{c|-1 +ch_1|^{-2}J_2(h_1, H)}\sqrt{c|-1 +ch_2|^{-2}J_2(h_2, H)}
\end{align*}
Noting that for $a,b,c,d \geq 0$, $\displaystyle\sqrt{ac} + \sqrt{bd} \leq \sqrt{a+b}\sqrt{c+d}$ with equality if and only if $a=b=c=d=0$. So we have
\begin{align*}
    &T_1 + T_2\\
    =& \sqrt{c|1 +ch_1|^{-2}J_2(h_1, H)}\sqrt{c|1 +ch_2|^{-2}J_2(h_2, H)} +  \sqrt{c|-1 +ch_1|^{-2}J_2(h_1, H)}\sqrt{c|-1 +ch_2|^{-2}J_2(h_2, H)}\\
    \leq& \sqrt{(c|1 +ch_1|^{-2} + c|-1 +ch_1|^{-2})J_2(h_1, H)}\sqrt{(c|1 +ch_2|^{-2} + c|-1 +ch_2|^{-2})J_2(h_2, H)}\\
    =& \sqrt{c\sigma_2(ch_1)J_2(h_1, H)}\sqrt{c\sigma_2(ch_2)J_2(h_2, H)}
    <  1 \text{, using } (\ref{lessThanOne_Im})
\end{align*}
This implies that $|h_1 - h_2|< |h_1 - h_2|$ which is a contradiction thus proving the uniqueness of $h(z) \in \mathbb{C}^+$.
\end{proof}
The proof of continuous dependence of $h$ on the distribution function $H$ is given in Lemma \ref{Continuity_Im}.\\
\subsection{Existence of Solution}

As before, we first prove the theorem under Assumptions \ref{A123} and follow it up for the general case. Establishing the necessary properties of $s_n(\cdot)$ and $h_n(\cdot)$ in Lemma \ref{CompactConvergence_Im} and Lemma \ref{ConcentrationOfSnH_Im}, we construct a sequence of deterministic matrices $\Bar{Q}(z) \in \mathbb{C}^{p \times p}$ satisfying $|\frac{1}{p}\operatorname{trace}(Q(z)-\Bar{Q}(z))| \xrightarrow{a.s.} 0$ in Theorem \ref{DeterministicEquivalent_Im}. Finally, we prove the existence of (the) solution to (\ref{6.2}) in Theorem \ref{ExistenceA123_Im}. 

\begin{lemma}\label{CompactConvergence_Im}
    \textbf{(Compact Convergence)} 
    $\{s_n(z)\}_{n=1}^\infty$ and $\{h_n(z)\}_{n=1}^\infty$ form a normal family, i.e. every subsequence has a further subsequence that converges uniformly on compact subsets of $\mathbb{C}^+$. 
\end{lemma}
\begin{proof}
By Montel's theorem, it is sufficient to show that $s_n$ and $h_n$ are uniformly bounded on every compact subset of $\mathbb{C}^+$. Let $K \subset \mathbb{C}^+$ be an arbitrary compact subset. Define $v_0 := \inf\{|\Im(z)|: z \in K\}$. Noting that $v_0 > 0$, for arbitrary $z \in K$, we have 
$$|s_n(z)| = \frac{1}{p}|\operatorname{trace}(Q(z))| \leq \frac{1}{|\Im(z)|} \leq \frac{1}{v_0}$$
    and using by (R5) of \ref{R123} and (T4) of Theorem \ref{t6.1}, we get for sufficiently large $n$,
    \begin{align}\label{h_nBound_Im}\tag{D.6}
        |h_n(z)| =& \frac{1}{p}|\operatorname{trace}(\Sigma_nQ)| \leq \bigg(\frac{1}{p}\operatorname{trace}(\Sigma_n)\bigg) ||Q(z)||_{op} 
        \leq \frac{C}{|\Im(z)|} \leq \frac{C}{v_0}
    \end{align}
\end{proof}
 
\begin{theorem}\label{DeterministicEquivalent_Im}
Let $M_n \in \mathbb{C}^{p \times p}$ be a sequence of deterministic matrices with $||M_n||_{op} \leq B$ for some $B \geq 0$. For $z \in \mathbb{C}^+$,
    $\frac{1}{p}\operatorname{trace}\{(Q(z) - \Bar{Q}(z))M_n\} \xrightarrow{a.s.} 0$
where $\Bar{Q}(z) := \bigg(-zI_p + \sigma(c_n\mathbb{E}h_n(z))\Sigma_n\bigg)^{-1}$
\end{theorem}

\textbf{REMARK}: Let $z = u + \mathbbm{i}v$ with $v > 0$. By Lemma \ref{boundedAwayFromZero_Im}, $\Im(c_n\mathbb{E}h_n(z)) \geq K_0 > 0$ a.s. for sufficiently large $n$, where $K_0$ depends only on $z, c, \tau$ and $H$. So for large $n$, $\sigma(c_n\mathbb{E}h_n(z))$ is well defined. Expressing $\Sigma_n = P\Lambda P^*$ with $\Lambda = \operatorname{diag}(\{\lambda_j\}_{j=1}^p)$, the $j^{th}$ eigenvalue of $\Bar{Q}(z)$ is $e_j:= (-z + \lambda_j\sigma(c_n\mathbb{E}h_n))^{-1}$. Then, for sufficiently large $n$ using (\ref{ImaginaryOfSigma}), we observe that
\begin{align*}\label{ImOf_Qbar_EV_inverse}\tag{D.7}
\Im(e_j^{-1}) = \Im(-z + \lambda_j\sigma(c_n\mathbb{E}h_n)) = -\Im(z) - \lambda_j\Im(c_n\mathbb{E}h_n)\sigma_2(c_n\mathbb{E}h_n) \leq -v < 0    
\end{align*}
In particular, $(-zI_p + \sigma(c_n\mathbb{E}h_n(z))\Sigma_n)$ is invertible for large $n$. The proof of Theorem \ref{DeterministicEquivalent_Im} can be found in Section \ref{sec:ProofDeterministicEquivalent_Im}. 

\begin{definition}
    For $z \in \mathbb{C}^+$ with $u >0$, with $\Bar{Q}(z)$ as defined in Theorem \ref{DeterministicEquivalent_Im}, we define the following
     \begin{align}
        \Tilde{h}_n(z) :=&\frac{1}{p}\operatorname{trace}\{\Sigma_n \Bar{Q}(z)\} = \displaystyle\int\dfrac{\lambda dF^{\Sigma_n}(\lambda)}{-z + \lambda \sigma(c_n\mathbb{E}h_n)} \label{definingHnTilde_Im}\tag{D.8}\\
        \Bar{\Bar{Q}}(z) :=& \bigg(-zI_p + \sigma(c_n\Tilde{h}_n(z))\Sigma_n\bigg)^{-1} \label{definingQBarBar_Im}\tag{D.9}\\
        \Tilde{\Tilde{h}}_n(z) :=& \frac{1}{p}\operatorname{trace}\{\Sigma_n \Bar{\Bar{Q}}(z)\} = \displaystyle\int\dfrac{\lambda dF^{\Sigma_n}(\lambda)}{-z + \lambda \sigma(c_n\Tilde{h}_n(z))} \label{definingHnTilde2_Im}\tag{D.10}
    \end{align}
\end{definition}

\begin{theorem}\label{ExistenceA123_Im} \textbf{(Existence of Solution)}
Under Assumptions \ref{A123}, for $z \in \mathbb{C}^+$, $h_n(z) \xrightarrow{a.s.} h^\infty(z) \in \mathbb{C}^+$ which satisfies (\ref{6.2}). Moreover, $s_n(z) \xrightarrow{a.s.} s^\infty(z)$ where $\underset{y \rightarrow \infty}{\lim}\mathbbm{i}ys^\infty(\mathbbm{i}y) = -1$ and 
$$s^\infty(z) = \frac{1}{z}\bigg(\frac{2}{c}-1\bigg) + \frac{1}{cz}\bigg(\frac{1}{-1 +ch^\infty(z)} - \frac{1}{1 + ch^\infty(z)}\bigg)$$
\end{theorem}
The proof is given in Section \ref{sec:ProofExistenceA123_Im}.

\subsection{Proof of Existence of solution under General Conditions}\label{ExistenceGeneral_Im}
In this section, we will prove (2)-(4) of Section \ref{ProofSketch} under the general conditions of Theorem \ref{t6.1}. To achieve this, we will create a sequence of matrices \textit{similar} to $\{S_n\}_{n=1}^\infty$ but satisfying \textbf{A1-A2} of Section \ref{A123}. The outline of the proof is given below with necessary details split into individual modules.
\begin{description}
    \item[Step1]  For a p.s.d.  matrix A and $\tau > 0$, let $A^\tau$ represent the matrix obtained by replacing all eigenvalues greater than $\tau$ with 0 in the spectral decomposition of A. For a distribution function $H$ supported on $\mathbb{R}_+$, define $H^{\tau}(t) := \mathbbm{1}_{\{t > \tau\}} + (H(t) + 1 - H(\tau))\mathbbm{1}_{\{0 \leq t \leq \tau\}}$. $H^\tau$ is a distribution function that transfers all mass of $H$ beyond $\tau$ to the point $0$.
    \item[Step2] Denote $\Lambda_n := \Sigma_n^\frac{1}{2}$ and $\Lambda_n^\tau := (\Sigma_n^\tau)^\frac{1}{2}$
    \item[Step3] We have $S_n := \frac{1}{n}\Lambda_n(Z_1Z_2^* + Z_2Z_1^*)\Lambda_n$
    \item[Step4] Define $T_n := \frac{1}{n}\Lambda_n^{\tau}(Z_1Z_2^* + Z_2Z_1^*)\Lambda_n^{\tau}$. $\Lambda_n^\tau$ satisfies \textbf{A1} of Section \ref{A123}.
    \item[Step5] Recall that for $k \in \{1, 2\}$, we have $Z_k = ((z_{ij}^{(k)})) \in \mathbb{C}^{p \times n}$. With $B_n = n^a$ as in \textbf{A2} of Assumptions \ref{A123}, define $\hat{Z}_k := (\hat{z}_{ij}^{(k)})_{ij}$ with $\hat{z}_{ij}^{(k)} = z_{ij}^{(k)}\mathbbm{1}_{\{|z_{ij}^{(k)}| \leq B_n\}}$. Now, let $U_n := \frac{1}{n}\Lambda_n^\tau(\hat{Z_1}\hat{Z_2}^* + \hat{Z_2}\hat{Z_1}^*)\Lambda_n^\tau$.
    \item[Step6] Construct $\Tilde{U}_n:=\frac{1}{n}\Lambda_n^\tau(\Tilde{Z_1}\Tilde{Z_2}^* + \Tilde{Z_2}\Tilde{Z_1}^*)\Lambda_n^\tau$ where $\Tilde{Z}_k = \hat{Z}_k - \mathbb{E}\hat{Z}_k$. Then, $\Sigma_n^\tau$ satisfies \textbf{A1} and $\Tilde{Z}_k$, $k=1,2$ satisfy \textbf{A2} of Assumptions \ref{A123}. Let $s_n(z), t_n(z), u_n(z), \Tilde{u}_n(z)$ be the the Stieltjes transforms of $F^{S_n}, F^{T_n},$ $F^{U_n}, F^{\Tilde{U}_n}$ respectively.
    \item[Step7] By Theorem \ref{ExistenceA123_Im}, $F^{\Tilde{U}_n} \xrightarrow{a.s.} F^\tau$ for some $F^\tau$ which is characterised by a pair $(h^\tau, s^\tau)$ satisfying (\ref{6.1}) and (\ref{6.2}) with $H^\tau$ instead of $H$. In particular, $|\Tilde{u}_n(z) - s^\tau(z)| \xrightarrow{a.s.} 0$ by the same theorem.
    \item[Step8] Next we show that $h^\tau$ converges to some non-random limit as $\tau \rightarrow \infty$. Using Montel's Theorem, we are able to show that any arbitrary subsequence of $\{h^\tau\}$ has a further subsequence $\{h^{\tau_m}\}_{m=1}^\infty$ that converges uniformly on compact subsets (of $\mathbb{C}^+$) as $m \rightarrow \infty$. Each subsequential limit will be shown to belong to $\mathbb{C}^+$ and satisfy (\ref{6.2}). Hence by Theorem \ref{Uniqueness_Im}, all these subsequential limits must be the same which we denote by $h^\infty$. Therefore, $h^\tau \xrightarrow{\tau \rightarrow \infty} h^\infty$.
    \item[Step9] We derive $s^\infty$ from $h^\infty$ using (\ref{6.1}) and show that $s^\infty$ satisfies the condition in Theorem 1 of \cite{GeroHill03} (quoted in Proposition \ref{GeroHill} of this paper) to be a Stieltjes transform of a measure over the imaginary axis. So, there exists some distribution $F^\infty$ corresponding to $s^\infty$. 
   \textbf{Step8} and \textbf{Step9} are shown explicitly in Lemma \ref{l.D1}.
    \item[Step10]  We have
$$|s_n(z) - s^\infty(z)| \leq |s_n(z) - t_n(z)| + |t_n(z)-u_n(z)| + |u_n(z)-\Tilde{u}_n(z)| + |\Tilde{u}_n(z)-s^\tau(z)| + |s^\tau(z) - s^\infty(z)|$$
We will show that each term on the RHS goes to 0 as $n, \tau \rightarrow \infty$. \begin{itemize}
    \item  From Lemma \ref{l.D2} and (\ref{Levy_vs_uniform}), $L(F^{S_n}, F^{T_n}) \leq ||F^{S_n} - F^{T_n}|| \xrightarrow{a.s.} 0$
    \item From Lemma \ref{l.D3} and (\ref{Levy_vs_uniform}), $L(F^{T_n}, F^{U_n}) \leq ||F^{T_n} - F^{U_n}|| \xrightarrow{a.s.} 0$ 
    \item From Lemma \ref{l.D3} and (\ref{Levy_vs_uniform}), $L(F^{U_n}, F^{\Tilde{U}_n}) \leq ||F^{U_n} - F^{\Tilde{U}_n}|| \xrightarrow{a.s.} 0$
    \item Application of Lemma \ref{lD.6} on the above three items gives $|s_n(z) - t_n(z)| \xrightarrow{a.s.} 0$, $|t_n(z) - u_n(z)| \xrightarrow{a.s.} 0$ and $|u_n(z) - \Tilde{u}_n(z)| \xrightarrow{a.s.} 0$ respectively.
    \item From \textbf{Step7}, we have $|\Tilde{u}_n(z) - s^\tau(z)| \xrightarrow{a.s.} 0$
    \item $|s^\tau(z) - s^\infty(z)| \rightarrow 0$ is shown in Lemma \ref{l.D1}.
\end{itemize}
\item[Step11] Hence  $s_n(z) \xrightarrow{a.s.} s^\infty(z)$ which is a Stieltjes transform. Therefore, $F^{S_n} \xrightarrow{a.s.} F^\infty$ where $F^\infty$ is characterised by $(h^\infty, s^\infty)$ which satisfy (\ref{6.1}) and (\ref{6.2}). This concludes the proof of Theorem \ref{t6.1} in the general case.
\end{description}

\subsection{A few preliminary results}\label{sec:preliminaries_Im}

\begin{lemma}\label{lD.6}
    Let $\{F_n, G_n\}_{n=1}^\infty$ be sequences of distribution functions on $\mathbb{R}$ with $s_{F_n}(z), s_{G_n}(z)$ denoting their respective Stieltjes transforms at $z \in \mathbb{C}^+$. If $L(F_n, G_n) \rightarrow 0$, then $|s_{F_n}(z) - s_{G_n}(z)| \rightarrow 0$.
\end{lemma}
\begin{proof}
    The proof is similar to that of Lemma \ref{lA.1}.
\end{proof}

\begin{lemma} \label{Rank2Woodbury_Im}
Let $B \in \mathbb{C}^{p \times p}$ be of the form $B = A - zI_p$ for some Hermitian matrix $A$ and $z \in \mathbb{C}^+$. For vectors $u, v \in \mathbb{C}^p$, define $\langle u, v\rangle := u^*B^{-1}v$. Then,
\begin{description}
    \item[1] $(B+ uv^{*} + vu^{*})^{-1}u = B^{-1}(\alpha_1 u + \beta_1 v);\alpha_1 = (1 + \langle u,v\rangle)D(u,v); \beta_1 = -\langle u,u\rangle D(u,v)$
    \item[2] $(B+ uv^{*} + vu^{*})^{-1}v = B^{-1}(\alpha_2 u + \beta_1 v);\alpha_1 = (1 + \langle v,u\rangle)D(u,v); \beta_1 = -\langle v,v\rangle D(u,v)$
\end{description}
where $D(u,v) =  \bigg((1 + \langle u,v\rangle)(1 + \langle v,u \rangle)) -  \langle u,u\rangle\langle v,v\rangle\bigg)^{-1}$
\end{lemma}
\begin{proof}
Clearly, $B$ cannot have zero as eigenvalue. So $\langle u,v \rangle$ is well defined.
Let $P = B$, $Q=[u:v]$ and $R=[v: u]$. Note that $\operatorname{det}(I_2 + R^*P^{-1}Q)^{-1} = D(u,v)$. So, $D(u,v)$ is well-defined. Finally, observing that $B + uv^*+ vu^* = P + QR^*$, we use (\ref{ShermanMorrison}) to get
\begin{align*}
    &(B + uv^* + vu^*)^{-1}[u:v]\\
    =& B^{-1}[u:v]\bigg(I_2 - (I_2 + R^*P^{-1}Q)^{-1}R^*P^{-1}Q\bigg)\\
    =& B^{-1}[u:v] \bigg(I_2 -{\begin{bmatrix}
        1+\langle v,u \rangle & \langle v,v \rangle\\
        \langle u,u \rangle & 1+\langle u,v \rangle
    \end{bmatrix}}^{-1}\begin{bmatrix}
        \langle v,u \rangle &  \langle v,v \rangle\\
        \langle u,u \rangle &  \langle u,v \rangle\\        
    \end{bmatrix}\bigg)\\
    =&  B^{-1}[u:v] \bigg(I_2 - D(u,v)\begin{bmatrix}
        1+\langle u,v \rangle & -\langle v,v \rangle\\
        -\langle u,u \rangle & 1+\langle v,u \rangle
    \end{bmatrix}\begin{bmatrix}
        \langle v,u \rangle &  \langle v,v \rangle\\
        \langle u,u \rangle &  \langle u,v \rangle\\        
    \end{bmatrix}\bigg)\\
    =& D(u,v)B^{-1}[u:v] \begin{bmatrix}
        \alpha_1 & \beta_2\\
        \beta_1 & \alpha_2
    \end{bmatrix}
\end{align*}

\end{proof}

\begin{lemma}\label{Tightness_Im}
    $\{F^{S_n}\}_{n=1}^\infty$ is a tight sequence.
\end{lemma}
\begin{proof}
The proof is nearly identical to that of Lemma \ref{tightness}.     
\end{proof}

\begin{lemma}\label{Rank2Perturbation_Im}
Let $M_n \in \mathbb{C}^{p \times p}$ be a sequence of deterministic matrices with bounded operator norm, i.e. $||M_n||_{op} < B$ for some $B \geq 0$. Under Assumptions \ref{A123}, for $z \in \mathbb{C}^+$ and sufficiently large $n$, we have $$\underset{1 \leq j \leq n}{\max}|\operatorname{trace}\{M_nQ(z)\} - \operatorname{trace}\{M_nQ_{-j}(z)\}| \leq \dfrac{4cC B}{\Im^2 (z)} \text{ a.s.}$$
Consequently, $\underset{1 \leq j \leq n}{\max}|\frac{1}{p}\operatorname{trace}\{M_n(Q - Q_{-j})\}| \xrightarrow{a.s.} 0$
\end{lemma}

\begin{lemma}\label{ConcentrationOfSnH_Im}
\textbf{Concentration of Stieltjes Transforms}
    Under Assumptions \ref{A123} for $z \in \mathbb{C}^+$, we have $|s_n(z) - \mathbb{E}s_n(z)| \xrightarrow{a.s.} 0$ and $|h_n(z) - \mathbb{E}h_n(z)| \xrightarrow{a.s.} 0$ .
\end{lemma}

\begin{lemma}\label{boundedAwayFromZero_Im}
    Let $z \in \mathbb{C}^+$. Recall definitions of $h_n(z)$ (\ref{defining_hn}) and $v_n(z)$ (\ref{defining_vn_Im}). Under \textbf{A1} of Assumptions \ref{A123}, for sufficiently large $n$,  $\Im(c_nh_n(z)) \geq K_0$ a.s., $\Im(c_n\mathbb{E}h_n(z)) \geq K_0$ and $|v_n(z)| \leq 1/K_0$ a.s. where $K_0 > 0$ depends on $z, c, \tau$ and $H$.
\end{lemma}
\begin{proof}

We have $||\Sigma_n||_{op} \leq \tau$. Since $F^{\Sigma_n}$ and $H$ have a compact support $[0, \tau]$ and $F^{\Sigma_n} \xrightarrow{d} H$ a.s., we get
$\int_0^\tau\lambda dF^{\Sigma_n}(\lambda) \xrightarrow{} \int_0^\tau\lambda dH(\lambda) > 0$ since $H \neq \delta_0$. Therefore,
\begin{align*}\label{D.11}\tag{D.11}
\frac{1}{n}\operatorname{trace}(\Sigma_n) = c_n\int_0^\tau\lambda dF^{\Sigma_n}(\lambda) \xrightarrow{} c\int_0^\tau\lambda dH(\lambda) > 0    
\end{align*}

Let $z = u + \mathbbm{i}v$ with $v > 0$. Let $a_{ij}$ represent the $ij^{th}$ element of $A:= P^*\Sigma_nP$ where $S_n = P\Lambda P^*$ with $\Lambda = \operatorname{diag}(\{\lambda_j\}_{j=1}^p)$ being a diagonal matrix containing the real eigenvalues of $S_n$. Then,
    \begin{align*}
        c_nh_n = \frac{1}{n}\operatorname{trace}\{\Sigma_nQ\}
        = \frac{1}{n}\operatorname{trace}\{P^*\Sigma_nP(\Lambda - zI_p)^{-1}\}
         =\frac{1}{n}\sum_{j=1}^p \frac{a_{jj}}{\lambda_j  - z}
    \end{align*}

For any $\delta > 0$, we have
$$||S_n||_{op} = ||\frac{1}{n}X_1X_2^* + \frac{1}{n}X_2X_1^*||_{op} \leq 2\sqrt{||\frac{1}{n}X_1X_1^*||_{op}}\sqrt{||\frac{1}{n}X_2X_2^*||_{op}} \leq 2||\Sigma_n||_{op}\bigg(1+\sqrt{\frac{p}{n}}\bigg)^2 + \delta$$

Let $B = 4\tau(1+\sqrt{c})^2$. Then $\mathbb{P}(|\lambda_j| > B \hspace{2mm} i.o.) = 0$.

Define $B^* :=\left\{\begin{matrix}
    -B\operatorname{sgn}(u) & \text{ if } u \neq 0\\
    B &  \text{            if } u = 0
\end{matrix} \right.$

Then $(\lambda_j - u)^2 \leq (B^* - u)^2$. Therefore,
\begin{align*}
    \Im(c_nh_n) =& \frac{1}{n}\sum_{j=1}^p\frac{a_{jj}v}{(\lambda_j - u)^2 + v^2}\\
    \geq & \frac{1}{n}\sum_{j=1}^p \frac{a_{jj}v}{(B^* - u)^2 + v^2}\\
    =& \frac{v}{(B^* - u)^2 + v^2} \bigg(\frac{1}{n}\sum_{j=1}^p a_{jj}\bigg)\\
    =& \frac{v}{(B^* - u)^2 + v^2} \bigg(\frac{1}{n}\operatorname{trace}(\Sigma_n)\bigg) \text{, as } \operatorname{trace}(A) = \operatorname{trace}(\Sigma_n)\\
    \xrightarrow{} & \frac{v}{(B^* - u)^2 + v^2}\bigg(c\int_0^\tau\lambda dH(\lambda)\bigg) := K_0
    > 0 \text{ from } (\ref{D.11})
\end{align*} 

Therefore for sufficiently large $n$, $\Im(c_nh_n) \geq K_0 > 0$ a.s. In conjunction with Lemma \ref{ConcentrationOfSnH_Im}, we also get $\Im(c_n\mathbb{E}h_n) > 0$ for large $n$. Moreover, for large $n$, $c_nh_n \neq \pm 1$ a.s. and hence, the quantity $v_n$ is well defined almost surely. Noting that for $z \in \mathbb{C}$ with $\Im(z) \neq 0$,
\begin{align*}
    &\left|\frac{1}{1 - z^2}\right|
    = \left|\frac{1}{2}\frac{1 + z + 1 - z}{1 - z^2}\right|
    \leq \frac{1}{2}\bigg(\frac{1}{|1-z|} + \frac{1}{|1+z|}\bigg)
    \leq \frac{1}{2}\bigg(\frac{1}{|\Im(z)|} + \frac{1}{|\Im(z)|}\bigg) = \frac{1}{|\Im(z)|}
\end{align*}

we therefore conclude that for sufficiently large $n$, we must have 
$$|v_n| = \left|\frac{1}{1 - (c_nh_n)^2}\right| \leq \frac{1}{\Im(c_nh_n)} \leq \frac{1}{K_0} \text{ a.s. }$$
\end{proof}

\begin{lemma}\label{sigmaContinuity}
Suppose $\{X_n, Y_n\}_{n=1}^\infty$ are complex random variables with $\Im(X_n), \Im(Y_n) \geq B$ for some $B > 0$ and $|X_n - Y_n| \xrightarrow{a.s.} 0$. Then, $|\sigma(X_n) - \sigma(Y_n)| \xrightarrow{a.s.} 0$.
\end{lemma}
\begin{proof}
The result is clear from the below string of inequalities.
\begin{align*}
    &|\sigma(X_n) - \sigma(Y_n)| \leq \bigg|\frac{1}{1+X_n} - \frac{1}{1+Y_n}\bigg| + \bigg|\frac{1}{-1+X_n}-\frac{1}{-1+Y_n}\bigg| \leq \frac{2|X_n -Y_n|}{\Im(X_n)\Im(Y_n)} \leq \frac{2|X_n-Y_n|}{B^2}
\end{align*}

A direct implication of this result is as follows. Under Assumptions \ref{A123},
$|h_n - \mathbbm{E}h_n| \xrightarrow{a.s.} 0$ by Lemma \ref{ConcentrationOfSnH_Im} and by Lemma \ref{boundedAwayFromZero_Im}, for sufficiently large $n$, we have $\Im(c_nh_n) \geq K_0 > 0$ where $K_0$ depends on $c, z, \tau$ and $H$. Hence, $\Im(\mathbb{E}c_nh_n) \geq K_0 > 0$ for large $n$. Therefore we have 

\begin{align} \label{sigmaContinuity1}\tag{D.12}
    |\sigma(c_nh_n) - \sigma(c_n\mathbb{E}h_n)| \xrightarrow{a.s.} 0
\end{align}

Moreover, by Theorem \ref{DeterministicEquivalent_Im}, we have $|h_n - \Tilde{h}_n| \xrightarrow{a.s.} 0$. Therefore, we have $\Im(c_n\Tilde{h}_n) \geq K_0 > 0$ as well for large $n$. Thus, we also get
\begin{align} \label{sigmaContinuity2}\tag{D.13}
    |\sigma(c_n\Tilde{h}_n) - \sigma(c_n\mathbb{E}h_n)| \xrightarrow{a.s.} 0
\end{align}
\end{proof}

\begin{lemma}\label{normBound_Im}
    Under Assumptions \ref{A123}, for $z \in \mathbb{C}^+$ the operator norms of the matrices $\Bar{Q}(z), \Bar{\Bar{Q}}(z)$ defined in Theorem \ref{DeterministicEquivalent_Im} and (\ref{definingQBarBar_Im}) respectively are bounded by $1/|\Im(z)|$.
\end{lemma}
\begin{lemma}\label{hn_tilde_tilde2_Im}
Under Assumptions \ref{A123}, for $z \in \mathbb{C}^+$, we have $|\Tilde{h}_n(z) - \Tilde{\Tilde{h}}_n(z)| \rightarrow 0$    
\end{lemma}

\begin{lemma}\label{uniformConvergence_Im}
Under Assumptions \ref{A123}, the quantities $c_{1j},c_{2j}, d_{1j}, d_{2j}, F_j(r,s), r,s \in \{1,2\}, v_n, m_n$ as defined throughout the proof of Theorem \ref{DeterministicEquivalent} satisfy the following results.
\begin{align*}
    &\underset{1\leq j\leq n}{\max }|c_{1j} - v_n| \xrightarrow{a.s.} 0 \hspace{18mm} \underset{1\leq j\leq n}{\max }|d_{1j} - v_n|\xrightarrow{a.s.} 0\\
    &\underset{1\leq j\leq n}{\max }|c_{2j} - c_nh_nv_n|\xrightarrow{a.s.} 0 \hspace{10mm} \underset{1\leq j\leq n}{\max }|d_{2j} - c_nh_nv_n|\xrightarrow{a.s.} 0\\
    & \underset{1\leq j\leq n}{\max }|F_j(r,r) - m_n| \xrightarrow{a.s.} 0, r \in \{1,2\}\\
    & \underset{1\leq j\leq n}{\max }|F_j(r,s)| \xrightarrow{a.s.} 0  \text{ where } r \neq s, r,s \in \{1,2\} 
\end{align*}  
\end{lemma}

The proofs of some of the above lemmas have been left out as they mimic those of their respective counterparts in Appendix B.

\subsection{Proof of Theorem \ref{DeterministicEquivalent_Im}}\label{sec:ProofDeterministicEquivalent_Im}
\begin{proof}
Let $z \in \mathbb{C}^+$. Define $F(z) := \bigg(\Bar{Q}(z)\bigg)^{-1}$. From the resolvent identity (\ref{R123}), we have 
\begin{align*}\label{D.14}\tag{D.14}
Q - \Bar{Q} = Q \bigg(F + zI_p - \dfrac{1}{n}\sum_{j=1}^{n}(X_{j1}X_{j2}^{*} + X_{j2}X_{j1}^{*})\bigg) \Bar{Q}    
\end{align*}

Using the above, we get
\begin{align*}\label{T1_minu_T2_expansion_Im}\tag{D.15}
& \frac{1}{p}\operatorname{trace}\{(Q - \Bar{Q})M_n\}\\
=&\frac{1}{p}\operatorname{trace} \{Q(F + zI_p)\Bar{Q}M_n \} - \frac{1}{p} \operatorname{trace} \{Q \bigg(\sum_{j=1}^{n}\dfrac{1}{n} (X_{1j}X_{2j}^{*} + X_{2j}X_{1j}^{*})\bigg) \Bar{Q} M_n \} \\
=&\frac{1}{p}\operatorname{trace} \{(F + zI_p)\Bar{Q}M_nQ \} - \frac{1}{p} \operatorname{trace} \{\bigg(\sum_{j=1}^{n}\dfrac{1}{n} (X_{1j}X_{2j}^{*} + X_{2j}X_{1j}^{*})\bigg) \Bar{Q} M_n Q\} \\
= & \underbrace{\frac{1}{p}\operatorname{trace} \{(F + zI_p)\Bar{Q}M_nQ \}}_{{Term}_1} - \underbrace{\frac{1}{p} \sum_{j=1}^{n}\dfrac{1}{n} (X_{2j}^{*}\Bar{Q}M_nQX_{1j} + X_{1j}^{*}\Bar{Q}M_nQX_{2j})}_{{Term}_2} 
\end{align*}

To establish ${Term}_1 - {Term}_2 \xrightarrow{a.s.} 0$, we need to change the definition of $v_n$ (earlier defined in (\ref{defining_vn})) as follows. The definition of $E_j(r,s), F_j(r,s)$ and $m_n$ remain the same.
\begin{align*}\label{defining_vn_Im}\tag{D.16}
        &v_n(z) := \dfrac{1}{1 - (c_nh_n(z))^2}
\end{align*}

Simplifying ${Term}_2$ using Lemma \ref{Rank2Woodbury_Im}, with $A = Q_{-j}(z)$ (see (\ref{notations})), $u = \frac{1}{\sqrt{n}}X_{1j}$ and $v = \frac{1}{\sqrt{n}}X_{2j}$, we get

\begin{align*}\label{defining_cj_Im}\tag{D.17}
        \frac{1}{\sqrt{n}} QX_{1j} =& Q_{-j} \bigg(\frac{1}{\sqrt{n}}X_{1j} c_{1j} - \frac{1}{\sqrt{n}}X_{2j} c_{2j}\bigg)\\ 
        \text{where } c_{1j} =& (1 + E_j(1,2))Den(j); \hspace{5mm} c_{2j} =E_j(1,1)Den(j)\\
        Den(j) =& \bigg((1 + E_j(1,2))(1 + E_j(2,1)) - E_j(1,1)E_j(2,2)\bigg)^{-1} 
\end{align*}
and
\begin{align*}\label{defining_dj_Im}\tag{D.18}
        \frac{1}{\sqrt{n}} QX_{2j} =& Q_{-j} \bigg(\frac{1}{\sqrt{n}}X_{2j} d_{1j} - \frac{1}{\sqrt{n}}X_{1j} d_{2j}\bigg)\\ 
        \text{where } d_{1j} =& (1 + E_j(2,1))Den(j);\hspace{5mm} d_{2j} = E_j(2,2)Den(j)
\end{align*}

Using (\ref{defining_cj_Im}) and (\ref{defining_dj_Im}), $Term_2$ of (\ref{T1_minu_T2_expansion_Im}) can be simplified as follows.
\begin{align*}\label{T2_simplification_Im}\tag{D.19}
&{Term}_2
        =\frac{1}{p}\sum_{j=1}^{n}\dfrac{1}{n} (X_{2j}^{*}\Bar{Q}M_nQX_{1j} + X_{1j}^{*}\Bar{Q}M_nQX_{2j}) \\
        =&\sum_{j=1}^{n}\dfrac{1}{p\sqrt{n}} X_{2j}^*\Bar{Q}M_n \bigg(\frac{1}{\sqrt{n}} Q X_{1j}\bigg) +  \sum_{j=1}^{n}\dfrac{1}{p\sqrt{n}} X_{1j}^*\Bar{Q}M_n \bigg( \frac{1}{\sqrt{n}}QX_{2j}\bigg)\\
        =& \sum_{j=1}^{n}\dfrac{1}{p\sqrt{n}} X_{2j}^*\Bar{Q}M_n Q_{-j} \bigg(\frac{X_{1j}c_{1j} - X_{2j}c_{2j}}{\sqrt{n}}\bigg) + \sum_{j=1}^{n}\dfrac{1}{p\sqrt{n}} X_{1j}^*\Bar{Q}M_n Q_{-j}\bigg( \frac{X_{2j}d_{1j} - X_{1j}d_{2j}}{\sqrt{n}}\bigg)\\
        =& \frac{1}{p}\sum_{j=1}^{n}\Bigg[\bigg(c_{1j}F_j(2,1) - c_{2j}F_j(2,2)\bigg) + \bigg(d_{1j}F_j(1,2) - d_{2j}F_j(1,1)\bigg)\Bigg]  \text{ using definition } (\ref{definingFjrs})
\end{align*}

To proceed further, we need the limiting behaviour of $c_{1j}, c_{2j}, d_{1j}, d_{2j}, F_j(r,s), r,s \in \{1,2\}$ for $1 \leq j \leq n$. This is established in Lemma \ref{uniformConvergence_Im} and the summary of results is given below.
\begin{equation*}\label{uniformResults_Im}\tag{D.20}
\left\{ \begin{aligned} 
    &\underset{1\leq j\leq n}{\max }|c_{1j} - v_n| \xrightarrow{a.s.} 0 \hspace{5mm}
    \underset{1\leq j\leq n}{\max }|d_{1j} - v_n|\xrightarrow{a.s.} 0\\
    &\underset{1\leq j\leq n}{\max }|c_{2j} - c_nh_nv_n|\xrightarrow{a.s.} 0 \hspace{5mm}
    \underset{1\leq j\leq n}{\max }|d_{2j} - c_nh_nv_n|\xrightarrow{a.s.} 0\\
    & \underset{1\leq j\leq n}{\max }|F_j(r,r) - m_n| \xrightarrow{a.s.} 0, r \in \{1,2\}\\
    & \underset{1\leq j\leq n}{\max }|F_j(r,s)| \xrightarrow{a.s.} 0  \text{ where } r \neq s, r,s \in \{1,2\} 
\end{aligned} \right.
\end{equation*}

Note that \begin{enumerate}
    \item For sufficiently large $n$, $|v_n|$ is bounded above by Lemma \ref{boundedAwayFromZero_Im} 
    \item By (\ref{h_nBound_Im}), $|h_n| \leq C/|\Im(z)|$ 
    \item Using (R5) of \ref{R123} and Lemma \ref{normBound_Im}, for sufficiently large n,
\begin{align*}\label{mn_bound_Im}\tag{D.21}
|m_n| = \bigg|\frac{1}{n}\operatorname{trace}\{\Sigma_n\Bar{Q}M_nQ\}\bigg| \leq \bigg(\frac{1}{n}\operatorname{trace}(\Sigma_n)\bigg) ||\Bar{Q}M_nQ||_{op} \leq \frac{BC}{\Im ^2(z)}    
\end{align*}

\end{enumerate}

From (\ref{uniformResults_Im}), the above bounds and applying Lemma \ref{lA.6}, we get the following results 
\begin{enumerate}
    \item $\underset{1 \leq j \leq n}{\max}|c_{1j}F_j(2,1)| \xrightarrow{a.s.} 0$; \hspace{5mm} $\underset{1 \leq j \leq n}{\max}|d_{1j}F_j(1,2)| \xrightarrow{a.s.} 0$
    \item $\underset{1 \leq j \leq n}{\max}|c_{2j}F_j(2,2) - c_nh_nv_nm_n| \xrightarrow{a.s.} 0$; \hspace{5mm} $\underset{1 \leq j \leq n}{\max}|d_{2j}F_j(1,1) - c_nh_nv_nm_n| \xrightarrow{a.s.} 0$    
\end{enumerate}

With the above results, Lemma \ref{lA.7} applied on (\ref{T2_simplification_Im}) gives 
\begin{align*}\label{T2_simplification2_Im}\tag{D.22}
|Term_2 - (-2h_nv_nm_n)| \xrightarrow{a.s.} 0    
\end{align*}

Now note that 
\begin{align*}\label{T2_expansion_Im}\tag{D.23}
    -2h_nv_nm_n =& \frac{n}{p} \dfrac{2c_nh_n}{-1 + (c_nh_n)^2}\frac{1}{n}\operatorname{trace}\{\Sigma_n \Bar{Q}M_nQ\} \text{, by definitions } (\ref{defining_vn_Im}), (\ref{defining_mn}) \\
    =& \frac{1}{p} \bigg[\frac{1}{1 + c_nh_n} + \frac{1}{-1 + c_nh_n}\bigg]\operatorname{trace}\{\Sigma_n\Bar{Q}M_nQ\}\\
    =&\frac{1}{p} \operatorname{trace}\{\sigma(c_nh_n)\Sigma_n\Bar{Q}M_nQ\}
\end{align*}
where the last equality follows from definition \ref{definingRho}. Finally from (\ref{sigmaContinuity1}) and (\ref{mn_bound_Im}), we get
\begin{align*}\label{interim1_Im}\tag{D.24}
\bigg|\frac{1}{p} \operatorname{trace}\{\sigma(c_nh_n)\Sigma_n\Bar{Q}M_nQ\} - \frac{1}{p} \operatorname{trace}\{\sigma(c_n\mathbb{E}h_n)\Sigma_n\Bar{Q}M_nQ\}\bigg| \xrightarrow{a.s.} 0    
\end{align*}

Combining (\ref{T2_simplification2_Im}), (\ref{T2_expansion_Im}) and (\ref{interim1_Im}), we get
\begin{align*}
    &|Term_2 - \frac{1}{p} \operatorname{trace}\{\sigma(c_n\mathbb{E}h_n)\Sigma_n\Bar{Q}M_nQ\}| \xrightarrow{a.s.} 0\\
    \implies & |Term_2 - \frac{1}{p}\operatorname{trace}\{[zI_p -zI_p + \sigma(c_n\mathbb{E}h_n)\Sigma_n]\Bar{Q}M_nQ\}| \xrightarrow{a.s.} 0\\
    \implies & |Term_2 - \frac{1}{p} \operatorname{trace}\{(F(z) + zI_p)\Bar{Q}M_nQ\}| \xrightarrow{a.s.} 0\\
    \implies& |Term_2 - Term_1| \xrightarrow{a.s.} 0
\end{align*}
This concludes the proof.
\end{proof}

\subsection{Proof of Theorem \ref{ExistenceA123_Im}}\label{sec:ProofExistenceA123_Im}
\begin{proof}
By Lemma \ref{CompactConvergence_Im}, every sub-sequence of $\{h_n(z)\}$ has a further sub-sequence that converges uniformly in each compact subset of $\mathbb{C}^+$. Let $h_0(z)$ be one such sub-sequential limit corresponding to the sub-sequence $\{h_{k_m}(z)\}_{m=1}^\infty$. For simplicity, denote  $g_m = h_{k_m}, \Tilde{g}_m = \Tilde{h}_{k_m}, \Tilde{\Tilde{g}}_m = \Tilde{\Tilde{h}}_{k_m}, d_m = c_{k_m}$ and $G_m = F^{\Sigma_{k_m}}$ for $m \in \mathbb{N}$. Thus we have $g_m \xrightarrow{a.s.} h_0$. By Theorem \ref{DeterministicEquivalent_Im}, we have, 

$$|\Tilde{g}_m - h_0| = |\Tilde{h}_{k_m} - h_0| \leq |\Tilde{h}_{k_m} - h_{k_m}| + |h_{k_m} - h_0| \rightarrow 0$$

Therefore, $\Tilde{g}_m \rightarrow h_0$. By Lemma \ref{hn_tilde_tilde2_Im}, we have
\begin{align*}\label{gm_tilde_tilde2_Im}\tag{D.25}
&\Tilde{g}_m(z) - \Tilde{\Tilde{g}}_m(z) \rightarrow 0\\
\implies &\displaystyle \Tilde{g}_m(z)- \int\dfrac{\lambda dG_m(\lambda)}{-z + \lambda\sigma(d_m\Tilde{g}_m(z))} \longrightarrow 0\\
\implies &\Tilde{g}_m(z) -\int\dfrac{\lambda d\{G_m(\lambda)-H(\lambda)\}}{-z +\lambda\sigma(d_m\Tilde{g}_m(z))} -
\int\dfrac{\lambda dH(\lambda)}{-z + \lambda\sigma(d_m\Tilde{g}_m(z))} \longrightarrow 0
\end{align*}

For large $m$, the common integrand in the second and third terms of (\ref{gm_tilde_tilde2_Im}) can be bounded above as follows.
\begin{align}\label{integrandBound_Im}\tag{D.26}
    &\bigg|\frac{\lambda}{-z + \lambda\sigma(d_m\Tilde{g}_m)}\bigg|
    \leq \frac{|\lambda|}{|\Im(-z + \lambda\sigma(d_m\Tilde{g}_m))|}
    \leq \frac{|\lambda|}{|\Im(\lambda\sigma(d_m\Tilde{g}_m))|} = \frac{1}{|\Im(\sigma(d_m\Tilde{g}_m))|
    } \xrightarrow{} \frac{1}{|\Im(\sigma(ch_0))|}
\end{align}

The limit in (\ref{integrandBound_Im}) follows because of the following argument. First note that $\Im(ch_0) > 0$. To see this, note that $g_m \xrightarrow{a.s.} h_0, d_m \rightarrow c$ and by Lemma \ref{boundedAwayFromZero_Im}, for sufficiently large m, $\Im(d_mg_m) \geq K_0(c,z,H, \tau) > 0$ a.s. Therefore, $\Im(ch_0) > 0$.
Secondly, $d_m\Tilde{g}_m = c_{k_m}\Tilde{h}_{k_m} \rightarrow ch_0$. By continuity of $\sigma$ at $ch_0$, we have $\sigma(d_m\Tilde{g}_m) \rightarrow \sigma(ch_0)$. Finally by (\ref{ImaginaryOfSigma}), $\Im(\sigma(ch_0)) = -\Im(ch_0)\sigma_2(ch_0) < 0$.

So the second term of (\ref{gm_tilde_tilde2_Im}) can be made arbitrarily small as $G_m \xrightarrow{d} H$. Applying D.C.T. in the third term of (\ref{gm_tilde_tilde2_Im}), we get 
\begin{align*}\label{subsequentialLimit_Im}\tag{D.27}
h_0(z) = \displaystyle\int\dfrac{\lambda dH(\lambda)}{-z + \lambda\sigma(ch_0(z))}
\end{align*}

Thus any subsequential limit ($h_0(z) \in \mathbb{C}^+$) satisfies (\ref{6.2}). By Lemma \ref{Uniqueness_Im}, all sub-sequential limits must be the same, say $h^\infty(z)$. This implies $h_n(z) \xrightarrow{a.s.} h^\infty(z)$ where the convergence is uniform in each compact subset of $\mathbb{C}^+$. Therefore, the limit $h^\infty$ must be analytic in $\mathbb{C}^+$ itself.

To complete the proof, we need to prove that $s_n(z) \xrightarrow{a.s.} s^\infty(z)$. Here we define an intermediate quantity 
\begin{align*}\label{defining_sn_tilde_Im}\tag{D.28}
\Tilde{s}_n(z) := \frac{1}{z}\bigg(\frac{2}{c_n}-1\bigg) + \frac{1}{c_nz}\bigg(\frac{1}{-1 +c_n \Tilde{h}_n(z)} - \frac{1}{1 +c_n \Tilde{h}_n(z)}\bigg)    
\end{align*}

It is clear that $\Tilde{s}_n(z) \rightarrow s^\infty(z)$ since $c_n \rightarrow c$ and $\Tilde{h}_n(z) \rightarrow h^\infty(z)$. So it is sufficient to show $s_n(z) - \Tilde{s}_n(z) \xrightarrow{a.s.} 0$. We will utilise use relationship between the resolvent and the co-resolvent for this. The co-resolvent of $S_n$ is given by 
$$\Tilde{Q}(z) := \bigg(\dfrac{1}{n}\begin{bmatrix}
    X_2^*\\
    X_1^*\\
\end{bmatrix}[X_1:X_2] - zI_{2n}\bigg)^{-1}$$ 

First we simplify the expression a bit. Let $A = \frac{1}{\sqrt{n}}[X_{1}:X_{2}]$ and $B = \frac{1}{\sqrt{n}}[X_{2}:X_{1}]^*$. Then $Q(z) = (AB - zI_p)^{-1}$ and $\Tilde{Q}(z) = (BA - zI_{2n})^{-1}$. Observe that

\begin{align*}\label{defining_Qtilde_Im}\tag{D.29}
(BQA - I_{2n})(BA - zI_{2n}) =& zI_{2n} \\
\implies \Tilde{Q}(z) = \frac{1}{z}BQA - \frac{1}{z}I_{2n}
=& \dfrac{1}{nz}\begin{bmatrix}
    X_2^*\\
    X_1^*\\
\end{bmatrix}Q[X_1:X_2] - \dfrac{1}{z}I_{2n}\\
=& \frac{1}{z}\begin{bmatrix}
    \frac{1}{n}X_2^*QX_1 - I_n & \frac{1}{n}X_2^*QX_2\\
    \frac{1}{n}X_1^*QX_1 & \frac{1}{n}X_1^*QX_2 - I_n
\end{bmatrix}\
=: \frac{1}{z}\begin{bmatrix}
    U & *\\
    * & V\\
\end{bmatrix}    
\end{align*}

We focus only on the two diagonal blocks $U$ and $V$ of $\Tilde{Q}$ (\ref{defining_Qtilde_Im}). Then for $1 \leq j \leq n$, 
\begin{align*}\label{zAjj_Im}\tag{D.30}
U_{jj}
=&\frac{1}{n}X_{2j}^*QX_{1j} - 1\\
=& \frac{1}{n}X_{2j}^*Q_{-j}(c_{1j} X_{1j} - c_{2j} X_{2j}) - 1 \text{, where } c_{1j}, c_{2j} \text{ are as per } (\ref{defining_cj_Im})\\
=& c_{1j} E_j(2,1) - c_{2j}E_j(2,2) - 1
\end{align*}

From Lemma \ref{uniformConvergence_Im}, we have $\underset{1\leq j\leq n}{\max}|E_j(2,1)| \xrightarrow{a.s.} 0$, $\underset{1\leq j\leq n}{\max}|E_j(2,2) - c_nh_n| \xrightarrow{a.s.} 0$,\\    
$\underset{1 \leq j \leq n}{\max}|c_{1j} - v_n| \xrightarrow{a.s.} 0$ and $\underset{1 \leq j \leq n}{\max}|c_{2j} -c_nh_nv_n| \xrightarrow{a.s.} 0$. Using Lemma \ref{lA.6} and Lemma \ref{boundedAwayFromZero_Im}, we have 
\begin{itemize}
    \item $|c_{1j}E_j(2,1)| \xrightarrow{a.s.} 0$
    \item $|c_{2j}E_j(2,2) - (c_nh_n)^2v_n| \xrightarrow{a.s.} 0$
\end{itemize}

Using the above results in (\ref{zAjj_Im}) and by Lemma \ref{lA.5}, we get 
\begin{align*}
&\underset{1\leq j \leq n}{\max}|U_{jj} - (0 - (c_nh_n)^2v_n - 1)| \xrightarrow{a.s.} 0\\
\implies & \underset{1\leq j \leq n}{\max}\bigg|U_{jj} + \dfrac{(c_nh_n)^2}{1 - (c_nh_n)^2} + 1\bigg| \xrightarrow{a.s.} 0\\
\implies & \underset{1\leq j \leq n}{\max}|U_{jj} +(1 - (c_nh_n)^2)^{-1}| \xrightarrow{a.s.} 0
\end{align*}

 Theorem \ref{DeterministicEquivalent_Im} implies that $|h_n - \Tilde{h}_n| \xrightarrow{a.s.} 0$ and by Lemma \ref{boundedAwayFromZero_Im}, $|v_n| = |1 - (c_nh_n)^2|^{-1}$ is bounded a.s. which implies $|1 - (c_n\Tilde{h}_n)^2|^{-1}$ is bounded a.s. for sufficiently large $n$. Moreover by (\ref{h_nBound_Im}), we have $|h_n| \leq C/|\Im(z)|$ which implies $|\Tilde{h}_n| \leq C/|\Im(z)|$. Therefore,

\begin{align*}\tag{D.31}
    \bigg|\dfrac{1}{1 - (c_nh_n)^2} - \dfrac{1}{1 - (c_n\Tilde{h}_n)^2}\bigg| \leq \dfrac{c_n^2|h_n - \Tilde{h}_n|(|h_n| + |\Tilde{h}_n|)}{|1 - (c_nh_n)^2||1 - (c_n\Tilde{h}_n)^2|} \xrightarrow{} 0
\end{align*}

Hence, we have 
$\underset{1 \leq j \leq n}{\max}|U_{jj} + (1 - c_n^2 \Tilde{h}_n^2)^{-1}| \xrightarrow{a.s.} 0$. Similarly, $\underset{1 \leq j \leq  n}{\max} |V_{jj} + (1 - c_n^2 \Tilde{h}_n^2)^{-1}| \xrightarrow{a.s.} 0$. Therefore using Lemma \ref{lA.7},
\begin{align*}\label{Ujj_Vjj_Im}\tag{D.32}
&\frac{1}{2n}\bigg|\sum_{j=1}^{n} \bigg(U_{jj} + V_{jj} + 2(1 - (c_n^2\Tilde{h}_n^2))^{-1}\bigg)\bigg| \xrightarrow{a.s.} 0\\
\implies& \bigg|\frac{1}{2n}\operatorname{trace}(\Tilde{Q}) + z^{-1}(1 - (c_n^2\Tilde{h}_n^2))^{-1}\bigg| \xrightarrow[]{} 0
\end{align*}

Using the identity linking the trace of the resolvent with that of the co-resolvent, we get
\begin{align*}
& \frac{1}{2n}\operatorname{trace}(\Tilde{Q}) = \frac{1}{2n}\operatorname{trace}(Q) + \frac{p-2n}{2nz}\\
\implies & \bigg|\frac{1}{z}\frac{1}{1 - c_n^2 \Tilde{h}_n^2} + \frac{c_n}{2}s_n(z) + \frac{1}{z}\bigg(\frac{c_n}{2} - 1\bigg)\bigg| \xrightarrow{a.s.} 0\\
\implies & \bigg|s_n(z) + \frac{1}{z}\bigg(1 - \frac{2}{c_n}\bigg) + \frac{2}{c_nz}\bigg(\frac{1}{1 - c_n^2 \Tilde{h}_n^2}\bigg)\bigg| \xrightarrow{a.s.} 0\\
\implies & \bigg|s_n(z) - \dfrac{1}{z}\bigg(\dfrac{2}{c_n}-1\bigg) - \dfrac{1}{c_nz}\bigg(\dfrac{1}{-1 +c_n\Tilde{h}_n(z)} - \dfrac{1}{1 +c_n\Tilde{h}_n(z)}\bigg)\bigg| \xrightarrow{a.s.} 0\\
\implies & |s_n(z) - \Tilde{s}_n(z)| \xrightarrow{a.s.} 0 \text{, using the definition of } \Tilde{s}_n \text{ from } (\ref{defining_sn_tilde_Im})
\end{align*}

Therefore, $s_n(z) \xrightarrow{a.s.} s^\infty(z)$. From (\ref{h_nBound_Im}), for sufficiently large $n$, $|h_n| \leq {C}/{|\Im(z)|}$ for $z \in \mathbb{C}^+$. Thus for $y \in \mathbb{R}$, $|h^\infty(\mathbbm{i}y)| \leq {C}/{|y|}$ and $\underset{y \rightarrow \infty}{\lim}h^\infty(\mathbbm{i}y) = 0$. This implies that 
\begin{align*}\tag{D.33}
\underset{y \rightarrow \infty}{\lim}\mathbbm{i}ys^\infty(\mathbbm{i}y) = \frac{2}{c}-1 + \underset{y \rightarrow \infty}{\lim}\frac{1}{c}\bigg(\frac{1}{-1 +ch^\infty(\mathbbm{i}y)} - \frac{1}{1 +ch^\infty(\mathbbm{i}y)}\bigg) = -1    
\end{align*}

Since $s^\infty(.)$ satisfies the necessary and sufficient condition in Theorem 1 of \cite{GeroHill03} (quoted in Proposition \ref{GeroHill}), it is a Stieltjes transform of some probability distribution $F^\infty$ and $F^{S_n} \xrightarrow{d} F^\infty$ a.s. So we have proved (2)-(4) of Section \ref{ProofSketch} under \textbf{A1-A2} of Section \ref{A123}.
\end{proof}

\subsection{Results related to Proof of Existence under General Conditions}

\begin{lemma}\label{l.D1}
    $h^\tau \rightarrow h^\infty$, $s^\tau \rightarrow s^\infty$  as $\tau \rightarrow \infty$
\end{lemma}
\begin{proof}
    Since Theorem \ref{t6.1} holds for $\Tilde{U}_n$, we have $F^{\Tilde{U}_n} \xrightarrow{d} F^{\tau}$ for some LSD $F^{\tau}$ and for $z \in \mathbb{C}^+$, there exists functions $s^\tau(z)$ and $h^\tau(z)$ satisfying (\ref{6.1}) and (\ref{6.2}) with $H^\tau$ replacing $H$ and mapping $\mathbb{C}^+$ to $\mathbb{C}^+$ and analytic on $\mathbb{C}^+$. We have to show existence of analogous quantities for the sequence $\{F^{S_n}\}_{n=1}^\infty$.

 First assume that H has a bounded support. If $\tau > \sup \operatorname{supp}(H)$, then $H^{\tau}(t) = H(t)$ and $H(\tau) = 1$. By the uniqueness property from Lemma \ref{Uniqueness_Im}, $h^{\tau}(z)$ must be the same for all large $\tau$. Hence $s^{\tau}(z)$ and in turn $F^{\tau}(.)$ must also be the same for all large $\tau$.  Denote this common LSD by $F^\infty$ and the common value of $h^\tau$ and $s^\tau$ by $h^\infty$ and $s^\infty$ respectively.
 This proves the case when H has a bounded support. 

Now we analyse the case where H has unbounded support. We will show that there exist functions $h^\infty$, $s^\infty$ that satisfy equations (\ref{6.1}) and (\ref{6.2}) and an LSD $F^\infty$ serving as the limit for the ESDs of $\{S_n\}_{n=1}^\infty$.

 We start by showing that $\mathcal{H} = \{h^\tau: \tau > 0\}$ forms a normal family. Following arguments similar to those used in the proof of Lemma \ref{CompactConvergence_Im}, let $K \subset \mathbb{C}^+$ be an arbitrary compact subset. Then $v_0 > 0$ where $v_0 := \inf\{|\Im(z)|: z \in K\}$. For arbitrary $z \in K$, using (R5) of \ref{R123} and (T4) of Theorem \ref{t6.1}, for sufficiently large $n$, we have
    \begin{align}\label{hn_tau_Bound_Im}\tag{D.34}
        |h_n^\tau(z)| =& \frac{1}{p}|\operatorname{trace}\{\Sigma_n^\tau Q\}| \leq \bigg(\frac{1}{p}\operatorname{trace}(\Sigma_n^\tau)\bigg) ||Q||_{op} 
        \leq \frac{C}{|\Im(z)|} \leq \frac{C}{v_0}
    \end{align}

By Theorem \ref{ExistenceA123_Im}, for any $\tau > 0$, $h^\tau(z)$ is the uniform limit of $h_n^\tau(z) := \frac{1}{n}\operatorname{trace}\{\Sigma_n^\tau Q(z)\}$. Therefore, for $z \in K$,    
\begin{align*}\label{h_tauBound_Im}\tag{D.35}
|h^\tau(z)|\leq \frac{C}{|\Im(z)|} \leq \frac{C}{v_0} 
\end{align*}
Therefore as a consequence of \textit{Montel's theorem}, any subsequence of $\mathcal{H}$ has a further convergent subsequence that converges uniformly on compact subsets of $\mathbb{C}^+$. Let $\{h^{\tau_m}(z)\}_{m=1}^\infty$ be such a sequence with $h_0(z)$ as the subsequential limit where $\tau_m \rightarrow \infty$ as $m \rightarrow \infty$.
By Lemma \ref{boundedAwayFromZero_Im}, for large $m$, $\Im(h^{\tau_m}(z)) \geq K_0(z,c,\tau_m, H^{\tau_m}) > 0$ which implies that $\Im(h_0(z)) \geq 0$ $\forall z \in \mathbb{C}^+$.

\begin{claim}\label{Claim_Im}
We must have $\Im(h_0(z)) > 0$ for all $z \in \mathbb{C}^+$. For a proof of this claim, see \ref{subsec:ProofOfClaim_Im}.    
\end{claim}

By (\ref{ImaginaryOfSigma}), and the fact that $\Im(h_0) > 0$, we have $\Im(\sigma(ch_0)) = -\Im(ch_0)\sigma_2(ch_0)< 0$. Therefore by continuity of $\sigma$ at $ch_0$, 
\begin{align*}\label{integrandBound3_Im}\tag{D.36}
\bigg|\frac{\lambda}{-z + \lambda\sigma(ch^{\tau_m})}\bigg|
    \leq \frac{|\lambda|}{|\Im(-z + \lambda\sigma(ch^{\tau_m}))|}
    \leq \frac{|\lambda|}{|\Im(\lambda\sigma(ch^{\tau_m}))|} = \frac{1}{|\Im(\sigma(ch^{\tau_m}))|} \xrightarrow{} \frac{1}{|\Im(\sigma(ch_0))|}  < \infty  
\end{align*}
as $m \rightarrow \infty$. Now by Theorem \ref{ExistenceA123_Im}, $(h^{\tau_m}, H^{\tau_m})$ satisfy the below equation.

\begin{align*}
   h^{\tau_m}(z) =& \int \dfrac{\lambda dH^{\tau_m}(\lambda)}{-z + \lambda\sigma(ch^{\tau_m})}
  = \int \dfrac{\lambda d\{H^{\tau_m}(\lambda) - H(\lambda)\}}{-z + \lambda\sigma(ch^{\tau_m})} + \int \dfrac{\lambda dH(\lambda)}{-z +\lambda\sigma(ch^{\tau_m})}\\
\end{align*}
Note that the first term of the last expression can be made arbitrarily small as the integrand is bounded by (\ref{integrandBound3_Im}) and $H^{\tau_m} \xrightarrow{d} H$. The same bound on the integrand also allows us to apply D.C.T. in the second term thus giving us
\begin{align*}
  &\underset{m \rightarrow \infty}{\lim} h^{\tau_m}(z) = \underset{m \rightarrow \infty}{\lim} \int \dfrac{\lambda dH(\lambda)}{-z +\lambda\sigma(ch^{\tau_m})}\\
  \implies &  h_0(z) = \int \dfrac{\lambda dH(\lambda)}{-z +\lambda\sigma(ch_0(z))}\label{subsequentialLimit2_Im}\tag{D.37}
\end{align*}

Now $\{\tau_m\}_{m=1}^\infty$ is a further subsequence of an arbitrary subsequence and $\{h^{\tau_m}(z)\}$ converges to $h_0(z) \in \mathbb{C}^+$ that satisfies (\ref{6.2}). By Theorem \ref{Uniqueness_Im}, all these subsequential limits must be the same and that $\{h^\tau(z)\}$ must converge uniformly to this common limit which we denote by $h^\infty(z)$. The uniform convergence of analytic functions $h^\tau$ in each compact subset of $\mathbb{C}^+$ imply that the limit $h^\infty$ must be analytic in $\mathbb{C}^+$.

Notice that $s^\tau(z) \rightarrow s^\infty(z)$ as $\tau \rightarrow \infty$ since 
\begin{align*}\label{s_tau_limit_Im}\tag{D.38}
    \underset{\tau \rightarrow \infty}{\lim}s^\tau(z) 
    =&\underset{\tau \rightarrow \infty}{\lim} \frac{1}{z}\bigg(\frac{2}{c}-1\bigg) + \frac{1}{cz}\bigg(\frac{1}{-1 +ch^\tau(z)} - \frac{1}{1 +ch^\tau(z)}\bigg)\\
    =& \frac{1}{z}\bigg(\frac{2}{c}-1\bigg) + \frac{1}{cz}\bigg(\frac{1}{-1 +ch^\infty(z)} - \frac{1}{1 +ch^\infty(z)}\bigg)
    =s^\infty(z)
\end{align*}

From (\ref{h_tauBound_Im}), $|h^\tau(z)| \leq {C}/{|\Im(z)|}$. Thus, $|h^\infty(z)| \leq {C}/{|\Im(z)|}$ implying that $\underset{y \rightarrow \infty}{\lim}h^\infty(y) = 0$. Therefore, 
$$\underset{y \rightarrow \infty}{\lim}\mathbbm{i}ys^\infty(\mathbbm{i}y) = \bigg(\frac{2}{c} - 1\bigg) + \underset{y \rightarrow \infty}{\lim}\frac{1}{c}\bigg(\frac{1}{-1 +ch^\infty(\mathbbm{i}y)} - \frac{1}{1 +  ch^\infty(\mathbbm{i}y)}\bigg)= -1$$ So, we have established that \begin{itemize}
    \item $h^\tau \rightarrow h^\infty$ and $s^\tau \rightarrow s^\infty$
    \item $h^\infty$ satisfies (\ref{6.2}) and is analytic on $\mathbb{C}^+$
    \item $s^\infty$ satisfies the conditions in Theorem 1 of \cite{GeroHill03} (quoted in Proposition \ref{GeroHill}) for a Stieltjes Transform
\end{itemize}
\end{proof}

\begin{lemma}\label{l.D2}
    $||F^{S_n} - F^{T_n}|| \xrightarrow{a.s.} 0$
\end{lemma} 
\begin{proof}
    Using (R1) and (R3) of Section \ref{R123}, we get
\begin{align*}
  ||F^{S_n} - F^{T_n}|| \leq& \dfrac{1}{p}\operatorname{rank}(S_n - T_n)\\
  =& \frac{1}{p}\operatorname{rank}\bigg(\Lambda_n(Z_1Z_2^* + Z_2Z_1^*)\Lambda_n - \Lambda_n^\tau(Z_1Z_2^* + Z_2Z_1^*)\Lambda_n^\tau\bigg)\\
  \leq& \dfrac{2}{p}\operatorname{rank}(\Lambda_n - \Lambda_n^\tau)\\
  =& \dfrac{2}{p}\operatorname{rank}(\Sigma_n - \Sigma_n^\tau)\\
  =& 2(1 - F^{\Sigma_n}(\tau)) \xrightarrow{n \rightarrow \infty} 2(1 - H(\tau)) \xrightarrow{\text{ as } \tau \rightarrow \infty} 0
\end{align*}
Here $\tau$ approaches $\infty$ only through continuity points of $H$.

\end{proof}

\begin{lemma} \label{l.D3}
    $||F^{T_n} - F^{U_n}|| \xrightarrow{a.s.} 0$, $||F^{U_n}- F^{\Tilde{U}_n}|| \xrightarrow{a.s.} 0$
\end{lemma}
\begin{proof}
    We have $T_n - U_n = \frac{1}{n}\Lambda_n^\tau(Z_1Z_2^* - \Hat{Z}_1\Hat{Z}_2^*)\Lambda_n^\tau + \frac{1}{n}\Lambda_n^\tau(Z_2Z_1^* - \Hat{Z}_2\Hat{Z}_1^*)\Lambda_n^\tau$. Therefore using (R1) and (R3) of Section \ref{R123}, we get
\begin{align*}\label{Tn_Un_Im}\tag{D.39}
        ||F^{T_n} - F^{U_n}|| \leq & \dfrac{1}{p}\operatorname{rank} (T_n - U_n) \\ 
        \leq&\frac{1}{p} \operatorname{rank}(Z_1Z_2^* - \Hat{Z}_1\Hat{Z}_2^*) + \frac{1}{p} \operatorname{rank}(Z_2Z_1^* - \Hat{Z}_2\Hat{Z}_1^*)\\
        \leq& \frac{2}{p}\bigg(\operatorname{rank}(Z_1 - \Hat{Z}_1) + \operatorname{rank} (Z_2 - \Hat{Z}_2)\bigg)
\end{align*}
The rest of the proof is exactly the same as that of Lemma \ref{3.4.3} following equation (\ref{Tn_Un}).
\end{proof}

\subsubsection{Proof of Claim \ref{Claim_Im}}\label{subsec:ProofOfClaim_Im}
\begin{proof}
    Suppose not, then $\exists z_0 \in \mathbb{C}^+$ with $\Im(h_0(z_0)) = 0$. Either, $h_0$ is non-constant in which case by the Open Mapping Theorem, $h_0(\mathbb{C}^+)$ is an open set containing $h_0(z_0)$ which is purely real. This implies that there exists $z_1 \in \mathbb{C}^+$, $\Im(h_0(z_1)) < 0$ which is a contradiction. 

The other case is that $h_0$ is constant in which case. For some $\zeta \in \mathbb{R}$, let $h_0(z) = \zeta, \forall z \in \mathbb{C}^+$. Note that for any $\tau > 0$, using the fact that $\sigma(-\overline{z}) =-\overline{\sigma(z)}$ (see the remark immediately following (\ref{ImaginaryOfSigma})), we get
\begin{align*}
-\overline{h^\tau(z)} = \int \frac{\lambda dH(\lambda)}{\overline{z} - \lambda \overline{\sigma(ch^\tau(z))}} = \int \frac{\lambda dH(\lambda)}{-(-\overline{z}) + \lambda \sigma(-c\overline{h^\tau(z)})} = h^\tau(-\overline{z})
\end{align*}
The last equality follows from Theorem \ref{Uniqueness_Im} since $-\overline{h^\tau(z)} \in \mathbb{C}^+$ and satisfies (\ref{6.2}) with $-\overline{z} \in \mathbb{C}^+$ instead of $z$. Therefore we observe that,
$$-\zeta = -\overline{h_0(z)} = \underset{m \rightarrow \infty}{\lim}-\overline{h^{\tau_m}(z)} = \underset{m \rightarrow \infty}{\lim}h^{\tau_m}(-\overline{z}) = h_0(-\overline{z}) = \zeta$$
implying that $\zeta = 0$ and in turn $h_0(z) = 0$ for all $z \in \mathbb{C}^+$. 

Fix $z = u + \mathbbm{i}v$ with $v > 0$. Recalling $J_1, J_2$ as defined in (\ref{ImOf_h}), we have,
\begin{align*}
    &\Im(h^{\tau_m}) = c\Im(h^{\tau_m})\sigma_2(ch^{\tau_m})J_2(h^{\tau_m}, H^{\tau_m}) + v J_1(h^{\tau_m}, H^{\tau_m})\\
    \implies& \underset{m \rightarrow \infty}{\lim} J_1(h^{\tau_m}, H^{\tau_m})= 0 \text{, using } (\ref{lessThanOne_Im}) \text{ and } v > 0\\
    \implies & \underset{m \rightarrow \infty}{\lim}\int_0^{\infty}\dfrac{\lambda dH^{\tau_m}(\lambda)}{|-z +\lambda\sigma(ch^{\tau_m})|^2} = 0 \label{rightBound_Im}\tag{D.40}
\end{align*}

For arbitrary $M > 0$, choose $m \in \mathbb{N}$ such that $\tau_m > M$. Then noting the relationship between $H$ and $H^{\tau_m}$, we have 
\begin{align*}\label{tempResult_Im}\tag{D.41}
    &\int_0^M\dfrac{\lambda dH(\lambda)}{|-z + \lambda\sigma(ch^{\tau_m})|^2}
    \leq      \int_0^{\tau_m}\dfrac{\lambda dH(\lambda)}{|-z +\lambda\sigma(ch^{\tau_m})|^2}
    = \int_0^{\infty}\dfrac{\lambda dH^{\tau_m}(\lambda)}{|-z +\lambda\sigma(ch^{\tau_m})|^2}
\end{align*}

Since $\Im(\sigma(ch^{\tau_m}) = -\Im(ch^{\tau_m})\sigma_2(ch^{\tau_m}) < 0$, we have for $0 \leq \lambda \leq M$,
\begin{align*}\label{integrandBound2_Im}\tag{D.42}
    \frac{|\lambda|}{|-z + \lambda\sigma(ch^{\tau_m})|^2}\leq \frac{|\lambda|}{(\Im(-z + \lambda\sigma(ch^{\tau_m})))^2} \leq \frac{M}{v^2}
\end{align*}

From (\ref{tempResult_Im}) and (\ref{rightBound_Im}), we get
\begin{align*}\label{lim_is_zero_arbtry_M_Im}\tag{D.43}
0 \leq \underset{m \rightarrow \infty}{\lim} \displaystyle \int_0^M\dfrac{\lambda dH(\lambda)}{|-z +\lambda\sigma(ch^{\tau_m})|^2} \leq \underset{m \rightarrow \infty}{\lim}\int_0^{\infty}\dfrac{\lambda dH^{\tau_m}(\lambda)}{|-z +\lambda\sigma(ch^{\tau_m})|^2} = 0    
\end{align*}

Now applying D.C.T. (because of \ref{integrandBound2_Im}) on the first term in (\ref{tempResult_Im}) and using (\ref{lim_is_zero_arbtry_M_Im}) we get
\begin{align*}
0 = \underset{m \rightarrow \infty}{\lim} \displaystyle \int_0^M\dfrac{\lambda dH(\lambda)}{|-z +\lambda\sigma(ch^{\tau_m})|^2} = \int_0^M\dfrac{\lambda dH(\lambda)}{|-z +\lambda\sigma(0)|^2} = \frac{1}{|z|^2}\int_0^M \lambda dH(\lambda) 
\end{align*}

Since $M > 0$ is arbitrary, it follows that 
$\int_0^\infty \lambda dH(\lambda) = 0$, which implies that $H\{0\} = 1$. This contradicts the assumption that $H$ is not degenerate at 0, and therefore proves the claim that $\Im(h_0(z)) > 0$.
\end{proof}

\begin{lemma} \label{Continuity_Im}
    The solution to (\ref{6.2}) has a continuous dependence on H, the distribution function.
\end{lemma}

\begin{proof}
 For a fixed $c > 0$ and $z \in \mathbb{C}^+$, let $h, \ubar{h}$ be the unique numbers in $\mathbb{C}^+$ corresponding to distribution functions $H$ and $\ubar{H}$ respectively that satisfy (\ref{6.2}). Following \cite{PaulSilverstein2009}, we have 
\begin{align*}
    h - \ubar{h}
    = & \int\dfrac{\lambda dH(\lambda)}{-z + \lambda\sigma(ch)} - \int\dfrac{\lambda d\ubar{H}(\lambda)}{-z + \lambda\sigma(c\ubar{h})}\\
    =& \underbrace{\int\dfrac{\lambda d\{H(\lambda) - \ubar{H}(\lambda)\}}{-z +\lambda\sigma(ch)}}_{:=T_1} + \int\dfrac{\lambda d\ubar{H}(\lambda)}{-z +\lambda\sigma(ch)} - \int\dfrac{\lambda d\ubar{H}(\lambda)}{-z + \lambda\sigma(c\ubar{h})}\\
    =& T_1 + \int\dfrac{\lambda^2(\sigma(c\ubar{h})- \sigma(ch))}{(-z + \lambda\sigma(ch))(-z + \lambda\sigma(c\ubar{h}))}d\ubar{H}(\lambda)\\
    =& T_1 + \int\dfrac{\dfrac{\lambda^2c(h-\ubar{h})}{(1 +ch)(1 +c\ubar{h})} + \dfrac{\lambda^2c(h-\ubar{h})}{(-1 +ch)(-1 +c\ubar{h})}}{(-z + \lambda\sigma(ch))(-z + \lambda\sigma(c\ubar{h}))}d\ubar{H}(\lambda)\\
    =& T_1 + (h-\ubar{h}) \underbrace{\int\dfrac{\dfrac{\lambda^2c}{(1 +ch)(1 +c\ubar{h})} + \dfrac{\lambda^2c}{(-1 +ch)(-1 +c\ubar{h})}}{(-z + \lambda\sigma(ch))(-z + \lambda\sigma(c\ubar{h}))}d\ubar{H}(\lambda)}_{:=\gamma}\\
    =& T_1 + (h-\ubar{h})\gamma
\end{align*}
Note that $\Im(\sigma(ch)) = -\Im(ch)\sigma_2(ch) < 0$ and the integrand in $T_1$ is bounded by $1/|\Im(\sigma(ch))|$. So by making $\ubar{H}$ closer to $H$, $T_1$ can be made arbitrarily small. Now, if we can show that $|\gamma| < 1$, this will essentially prove the continuous dependence of the solution to (\ref{h_main_eqn}) on H. 

\begin{align*}
    \gamma =& \underbrace{\int\dfrac{\dfrac{\lambda^2c}{(1 +ch)(1 +c\ubar{h})}}{(-z + \lambda\sigma(ch))(-z + \lambda\sigma(c\ubar{h}))}d\ubar{H}(\lambda)}_{:=G_1} + \underbrace{\int\dfrac{ \dfrac{\lambda^2c}{(-1 +ch)(-1 +c\ubar{h})}}{(-z + \lambda\sigma(ch))(-z \lambda\sigma(c\ubar{h}))}d\ubar{H}(\lambda)}_{:=G_2}\\
    =& G_1 + G_2
\end{align*}                                     

By $\Ddot{H}$older's Inequality we have,
\begin{align*}
    |G_1| &\leq \displaystyle \sqrt{\underbrace{\int\dfrac{c\lambda^2|1 +ch|^{-2}d\ubar{H}(\lambda)}{|-z + \lambda\sigma(ch)|^2}}_{:= P_1}}
    \sqrt{\underbrace{\int\dfrac{c\lambda^2|1 +c\ubar{h}|^{-2}d\ubar{H}(\lambda)}{|-z + \lambda\sigma(c\ubar{h})|^2}}_{:=P_2}} = \sqrt{P_1 \times P_2}
\end{align*}
From the definitions used in (\ref{ImOf_h}), we have $|P_2| = c|1 +c\ubar{h}|^{-2}J_2(\ubar{h}, \ubar{H})$
and 
\begin{align*}
    |P_1| =& c|1 +ch|^{-2}\int\dfrac{\lambda^2 d\ubar{H}(\lambda)}{|-z + \lambda\sigma(ch)|^2}\\
    & = c|1 +ch|^{-2}\bigg(\underbrace{\int\dfrac{\lambda^2 d\{\ubar{H}(\lambda)-H(\lambda)\}}{|-z + \lambda\sigma(ch)|^2}}_{:=K_1} + \int\dfrac{\lambda^2 dH(\lambda)}{|-z + \lambda\sigma(ch)|^2}\bigg)\\
    & =c|1 +ch|^{-2}K_1 + c|1 +ch|^{-2}J_2(h, H)\\
    & < \epsilon + c|1 +ch|^{-2}J_2(h, H)
\end{align*}

for some arbitrarily small $\epsilon > 0$. The last inequality follows since the integrand in $K_1$ is bounded by ${|\Im(\sigma(ch))|^{-2}}$, we can arbitrarily control the first term by taking $\ubar{H}$ sufficiently close to $H$ in the Levy metric. The argument for bounding $|G_2|$ is exactly the same. 

Therefore we have $|G_1| < \sqrt{\epsilon + c|1 +ch|^{-2}J_2(h, H)}\sqrt{c|1 +c\ubar{h}|^{-2}J_2(\ubar{h}, \ubar{H})}$. \\
Similarly, we get $|G_2| < \sqrt{\epsilon + c|-1 +ch|^{-2}J_2(h, H)}\sqrt{c|-1 +c\ubar{h}|^{-2}J_2(\ubar{h}, \ubar{H})}$.

Thus, using the inequality $\sqrt{ac} +\sqrt{bd} \leq \sqrt{a+b}\sqrt{c+d}$ with equality iff $a=b=c=d=0$, we have \begin{align*}
    &|G_1| + |G_2|\\
    < & \sqrt{\epsilon + c|1 +ch|^{-2}J_2(h, H)}\sqrt{c|1 +c\ubar{h}|^{-2}J_2(\ubar{h}, \ubar{H})} +\\
    &\sqrt{\epsilon + c|-1 +ch|^{-2}J_2(h, H)}\sqrt{c|-1 +c\ubar{h}|^{-2}J_2(\ubar{h}, \ubar{H})}\\
    \leq & \sqrt{2\epsilon + (c|1 +ch|^{-2} + c|-1 +ch|^{-2})J_2(h, H)}\sqrt{(c|1 +c\ubar{h}|^{-2} + c|-1 +c\ubar{h}|^{-2})J_2(\ubar{h}, \ubar{H})}\\
    =& \sqrt{2\epsilon + c\sigma_2(ch)J_2(h,H)}\sqrt{c\sigma_2(c\ubar{h})J_2(\ubar{h}, \ubar{H})}
\end{align*}

From (\ref{lessThanOne_Im}), we have $c\sigma_2(ch)J_2(h, H) < 1$ and $c\sigma_2(c\ubar{h})J_2(\ubar{h}, \ubar{H}) < 1$. By choosing $\epsilon > 0$ arbitrarily small, we finally have $|\gamma| = |G_1 + G_2| \leq |G_1| + |G_2| < 1$ for $\ubar{H}$ sufficiently close to H. This completes the proof.
\end{proof}

\end{document}